\documentclass[11pt]{article}
\usepackage[selected=true,shrink=15,stretch=30,letterspace=30]{microtype}
\usepackage{amsmath}
\usepackage{amssymb}
\usepackage{amsthm}
\usepackage{graphicx}
\usepackage{paralist}
\usepackage{multirow}
\usepackage[normalem]{ulem}
\usepackage[svgnames]{xcolor}
\usepackage{bm}
\usepackage{ifpdf}
\ifpdf
\usepackage[unicode=true,pdftex,colorlinks=true, pdfborderstyle={/S/D/D[2 1]/W .5},linkcolor=blue,citecolor=violet,urlcolor=Navy,pdfversion=1.5,pdflang=en-GB
,pdfauthor={Alexander Dyachenko and Dmitrii Karp},pdfsubject={MSC(2010) 33C05 (primary), 30B70, 47B50 (secondary)},%pdfpagemode=UseOutlines,%pdfa=true
,pdftitle={Ratios of the Gauss hypergeometric functions with parameters shifted by integers: Part I}
,pdfkeywords={Gauss hypergeometric function, Integral representations, Stieltjes class, Generalized Nevanlinna classes}
]{hyperref}
\else
\usepackage[unicode=true,colorlinks=true,linkcolor=blue,citecolor=violet,urlcolor=Navy]{hyperref}
\fi
\newcommand{\coloneqq}{\mathrel{\mathop:}=}
\newcommand{\eqqcolon}{=\mathrel{\mathop:}}

\newcommand{\diag}{\operatorname{diag}}
\newcommand{\Log}{\operatorname{Log}}
\renewcommand{\Re}{\operatorname{Re}}
\renewcommand{\Im}{\operatorname{Im}}

\newcounter{dummy}

\def\a{\mathbf{a}}

\def\R{\mathbb{R}}
\def\C{\mathbb{C}}

\def\N{\mathbb{N}}
\def\Z{\mathbb{Z}}

\def\NN{\mathcal{N}}
\def\NU{\mathfrak{N}}
\def\SU{\mathfrak{S}}

\DeclareMathOperator*{\sign}{sign}
\newcommand\smalldiv{\smash{\raisebox{0.29ex}{\scalebox{0.8}{/}}}}
\newtheorem{theorem}{\hspace*{\parindent}Theorem}
\newtheorem{lemma}[theorem]{\hspace*{\parindent}Lemma}
\newtheorem{corollary}[theorem]{\hspace*{\parindent}Corollary}
\theoremstyle{definition}
\newtheorem{definition}{\hspace*{\parindent}Definition}
\newtheorem{remark}[theorem]{\hspace*{\parindent}Remark}
\numberwithin{theorem}{section}
\numberwithin{equation}{section}

\title{Ratios of the Gauss hypergeometric functions with parameters shifted by integers: part I}
\author{Alexander Dyachenko$^{\rm a}$~~and Dmitrii\:Karp$^{\rm b}$\footnote{Corresponding author. E-mail: D. Karp -- \emph{dmitrika@hit.ac.il}, A.\:Dyachenko --  \emph{diachenko@sfedu.ru}}
\\[10pt]
\\
\small{\textit{$\phantom{1}^a$Department of Mathematics, University College London, London WC1E 6BT, UK}}
\\
\small{\textit{$\phantom{1}^b$Department of Mathematics,  Holon Institute of Technology, Holon, Israel}}
}
\date{}

\begin{document}
\maketitle

\begin{center}
    \parbox{12cm}{ \small\textbf{Abstract.} Given real parameters~$a,b,c$ and integer
        shifts~$n_1,n_2,m$, we consider the ratio
        $R(z)={ }_2F_1(a+n_1,b+n_2;c+m;z)/{ }_2F_1(a,b;c;z)$ of the Gauss hypergeometric
        functions. We find a formula for~$\Im R(x\pm i0)$ with~$x>1$ in terms of real
        hypergeometric polynomial~$P$, beta density and the absolute value of the Gauss
        hypergeometric function. This allows us to construct explicit integral representations
        for~$R$ when the asymptotic behaviour at unity is mild and the denominator does not
        vanish. Moreover, for arbitrary $a,b,c$ and $\omega\le 1$ the
        product~$P(z+\omega)R(z+\omega)$ is proved to belong to the generalized Nevanlinna
        class~$\NN_{\kappa}^\lambda$. We give an in-depth analysis of the case $n_1=0$,
        $n_2=m=1$ known as the Gauss ratio. Furthermore, we establish a few general facts
        relating generalized Nevanlinna classes to Jacobi and Stieltjes continued fractions, as
        well as to factorization formulae for these classes. The results are illustrated with a
        large number of examples.}
\end{center}

\bigskip


\bigskip


\bigskip

\section{Introduction}

The Gauss hypergeometric functions (\cite{Gauss}, \cite[Chapter~II]{HTF1},
\cite[Chapter~{\href{https://dlmf.nist.gov/15}{15}}]{DLMF})
\begin{equation}\label{eq:2F1_def}
{}_2F_1(a,b;c;z)={}_2F_1\left(\!\begin{array}{l}a,b\\c \end{array}\vline\,\,z\right)=\sum\limits_{n=0}^{\infty}\frac{(a)_n(b)_n}{(c)_nn!}z^n
\end{equation}
and ${ }_2F_1(a+n_1,b+n_2;c+m;z)$, ~$n_1,n_2,m\in\Z$, are called contiguous in a wide sense
\cite{EbisuIwasaki}. Any three functions of this type satisfy a linear relation with
coefficients rational in $a,b,c,z$. If $n_1,n_2,m\in\{-1,0,1\}$ the coefficients of this
relation are linear in $z$ and the functions are called contiguous in a narrow sense. Such a
contiguous relation was used by Euler to derive a continued fraction (much later termed
T-fraction) for the ratio ${}_2F_1(a,b+1;c+1;z)/{}_2F_1(a,b;c;z)$.
Gauss described all three-term relations among the functions contiguous in the narrow sense and
found another continued fraction for the above ratio which has the form~\cite[p.~134]{Gauss}
(see also~\cite[(89.9)]{Wall} or~\cite[p. 123]{Perron})
\begin{equation}\label{eq:gen_cont_fr}
G(z)=\frac{F(a,b+1;c+1;z)}{F(a,b;c;z)}=\cfrac{\alpha_0}{1-\cfrac{\alpha_{1}z}{1-\cfrac{\alpha_{2}z}{1-\cdots}}},
\end{equation}
where $\alpha_0=1$, and for $n\ge0$,
\begin{equation}\label{eq:gen_cont_fr_cf}
\alpha_{2n+1}=\frac{(a+n)(c-b+n)}{(c+2n)(c+2n+1)},~~~\alpha_{2n+2}=\frac{(b+n+1)(c-a+n+1)}{(c+2n+1)(c+2n+2)}.
\end{equation}

The main protagonist of this paper is the following generalization of the Gauss ratio \eqref{eq:gen_cont_fr}
\begin{equation}\label{eq:gen-ratio-def}
R_{n_1,n_2,m}(z)=\frac{{}_2F_1(a+n_1,b+n_2;c+m;z)}{{}_2F_1(a,b;c;z)},
\end{equation}
where $n_1,n_2,m\in\Z$ are arbitrary. Our main goal is to derive an integral representation for
this ratio under relatively general conditions, and to employ the continued
fraction~\eqref{eq:gen_cont_fr}--\eqref{eq:gen_cont_fr_cf} to show that a similar
representation exists when those conditions fail. In order to explain in more detail which
specific representations and properties of~$R_{n_1,n_2,m}(z)$ will be established in this paper,
we need to introduce some preliminaries.

Any continued fraction of the  form given by the right hand side of \eqref{eq:gen_cont_fr} 
corresponds to (it may be turned into) a unique power series~$c_0+c_1z+c_2z^2+\dots$
with~$\alpha_0=c_0$. The coefficients of this formal power series then satisfy the
conditions~$D_n^{(0)},D_n^{(1)}\ne 0$ for all~$n$, where
\begin{equation}\label{eq:Delta_def}
    D_n^{(p)}\coloneqq\det(c_{i+j+p})_{i,j=0}^{n-1}.
\end{equation}
These conditions are also sufficient for turning the power series back into a continued fraction as
in~\eqref{eq:gen_cont_fr} via:
\begin{equation}\label{eq:alpha_via_Delta}
        \alpha_{2n-1}=\frac{D^{(0)}_{n-1}D^{(1)}_{n}}{D^{(0)}_{n}D^{(1)}_{n-1}}
        \quad\text{and}\quad
        \alpha_{2n}=\frac{D^{(0)}_{n+1}D^{(1)}_{n-1}}{D^{(0)}_nD^{(1)}_n}
        ,\quad\text{where}\quad
        D^{(0)}_{0}=D^{(1)}_{0}=1
        \text{ and }
        n=1,2,\dots,
\end{equation}
see~\cite[§§21,23]{Perron} for the details. If a fraction as in~\eqref{eq:gen_cont_fr} is
terminating, then it corresponds to a rational function, and there is an index~$n_0$ such
that~$D_n^{(0)},D_{n-1}^{(1)}\ne 0$ for all~$n\le n_0$ and~$D_n^{(p)}=0$ for all~$n>n_0$
and~$p\ge 0$.

We need the following digest of certain classical results by Stieltjes~\cite{Stieltjes}, Van
Vleck~\cite[pp.~148--151]{Perron} or~\cite[Theorem 54.2 with footnote~19]{Wall},
Blumenthal~\cite[pp.~122--124]{Chihara} and others adapted to our situation:

\begin{theorem}\label{th:CF_conv}
    Let~\(\varphi(z)\) be a continued fraction of the form~\eqref{eq:gen_cont_fr} with arbitrary
    coefficients~$\alpha_0$, $\alpha_1$, $\ldots\in\C\setminus\{0\}$. Then the following assertions are
    true:
    \begin{enumerate}[\upshape (a)]
    \item\label{item:Worpitzky}%
        If~$\sup_{n} |\alpha_n|=l_0$, then for any positive~$r<\frac 1{4l_0}$ the continued
        fraction~$\varphi(z)$ and the power series corresponding via~\eqref{eq:Delta_def}--\eqref{eq:alpha_via_Delta} uniformly converge in the
        disc~$|z|\le r$ to an analytic function, say~$f(z)$.
    \item\label{item:VanVleck}%
        If~$\exists\lim_{n\to\infty} \alpha_n\eqqcolon l\in[0,+\infty)$, then~$f(z)$ may be
        analytically continued to a function meromorphic
        in whole~$\C\setminus\big[\frac{1}{4l},+\infty\big)$, or in~$\C$ when~$l=0$.
        Moreover,~$\varphi(z)$ uniformly converges to~$f(z)$ on compact subsets
        of~$\C\setminus\big[\frac{1}{4l},+\infty\big)$ excluding neighbourhoods of its poles. In
        such cases we identify~$\varphi(z)$ with~$f(z)$ and write~$\varphi(z)=f(z)$.
    \item\label{item:Stieltjes_H}%
        If~$0< 4\alpha_n\le \gamma<+\infty$ for all~$n$, then there exists a unique function~$\mu(x)$
        non-decreasing on~$[0,\gamma]$ such that
        \[
            c_n=\int_{[0,\gamma]} x^n d\mu(x) < \infty,
            \quad n=0,1,\dots.
        \]  
        In this case,~$\varphi(z)$ uniformly converges on compact subsets
        of~$\C\setminus\big[\frac1{\gamma},+\infty\big)$ and may be expressed as
        \begin{equation}\label{eq:gen_cont_fr_int}
            \varphi(z) = \sum_{n=0}^\infty c_nz^n
            \quad\text{for~$|z|<\frac 1\gamma$, and}\quad
            \varphi(z)= \int_{[0,\gamma]}\frac{d\mu(x)}{1-x z}
            \quad\text{for~$z\in\C\setminus\big[\frac 1\gamma,+\infty\big)$}
            .
        \end{equation}
    \item\label{item:Blumenthal}%
        If both~\eqref{item:VanVleck} and~\eqref{item:Stieltjes_H} hold, and~$l>0$, then the
        points of growth of~$\mu(x)$ are dense in~$\big[0,4l\big]$.
    \end{enumerate}
\end{theorem}
The coefficients~$(c_n)_{n=0}^\infty$ in Theorem~\ref{th:CF_conv}~\eqref{item:Stieltjes_H} are
called moments. The restriction~$\gamma<+\infty$ separates out the (scaled) Hausdorff
moment problem (the problem consists in finding~$\mu(x)$ from~$(c_n)_{n=0}^\infty$). On lifting
this restriction, one obtains the more general Stieltjes moment problem, where~$\gamma=+\infty$,
the power series in~\eqref{eq:gen_cont_fr_int} turns to be only asymptotic, and an analogous
integral representation persists (namely~\eqref{eq:StieltjesClass} with $C=0$ under a suitable summability
condition on~$\mu(s)$). However, the continued fraction~$\varphi(z)$ may then diverge
(and oscillate for~$z<0$), in which case~$\mu(x)$ is determined non-uniquely (the indeterminate
moment problem). More details on moment problems may be found
in~\cite{Akhiezer,Chihara,Perron,Stieltjes,Wall}.

In the case of the Gauss ratio we have $\lim_{n\to\infty}\alpha_n=1/4$,
while~$\sup_n|\alpha_n|\eqqcolon\gamma/4 \ge 1/4$. So, if $\alpha_{n}>0$, ~$n=1,2,\ldots$, then
there is a unique positive measure $d\mu(s)$ on~$[0,\gamma]$ whose support is dense in~$[0,1]$
and has at most finitely%
\footnote{%
    Theorem~\ref{th:CF_conv} only implies that the measure~$d\mu(s)$ is discrete
    in~$(1,\gamma]$. The fact that~$d\mu(s)$ has at most finitely many atoms in this interval
    directly follows from that~${ }_2F_1(a,b;c;z)$ has finitely many zeros in~$[0,1)$. The
    latter is given by Theorem~\ref{th:2F1zeros}, a corollary of~\cite{Runckel}.} %
many points in~$(1,\gamma]$, such that
\begin{equation}\label{eq:Gauss_St_int}
    G(z)=\int_{[0,\gamma]}\frac{d\mu(s)}{1-sz}
    .
\end{equation}
Moreover, \cite[Theorem~69.2]{Wall} asserts that one may take~$\gamma=1$
in~\eqref{eq:Gauss_St_int} iff $\alpha_n=(1-g_{n-1})g_{n}$ for all~$n\ge 1$ with some numbers
$g_n\in[0,1]$ (the equality cases correspond to rational~$G(z)$). It is immediate to see that the condition $\alpha_n>0$ is satisfied for the Gauss
fraction for all $n$ when $-1<b<c$ and~$0<a<c+1$. The more restrictive condition $g_n\in[0,1]$
holds true if $0\le{a}\le{c+1}$, $0\le{b}\le{c}$, see \cite[Proof of Theorem~1.1]{Kuestner} for
details. Surprisingly enough, the representing measure $d\mu$ for the Gauss fraction has only
been computed in 1984 by Vitold Belevitch \cite{Bel}.  Around the same time Jet Wimp \cite{Wimp87} constructed explicit formulae for the convergents of the Gauss continued fraction in terms of hypergeometric polynomials.  

The class of all functions possessing the integral representation 
\begin{equation}\label{eq:StieltjesClass}
    f(z)=-\frac{C}{z}+\int_{[0,\infty)}\frac{d\mu(s)}{1-sz}
\end{equation}
for some $C\ge0$ and a positive measure $\mu$ supported on $[0,\infty)$ making the integral
convergent is known as the Stieltjes cone $\mathcal{S}$ \cite{Berg}. It is a much studied and
important class in many areas from analysis and operator theory to combinatorics and
probability. The Stieltjes cone is a subclass of another famous class in function theory
commonly known as the Pick-Nevanlinna class $\NN_0$ comprising functions $f$ holomorphic in
$\C\setminus\R$ satisfying $\Im(f(z))\ge0$ for all $\Im(z)>0$ and possessing the symmetry
property
\begin{equation} \label{eq:real_func}
    f(\overline z)=\overline {f(z)}
    \quad\text{for all $z\in\mathbb C$, where~$f(z)$ is defined}.
\end{equation}
Any function~$f(z)$ meromorphic in~$\mathbb C\setminus\mathbb R$ possessing the above symmetry
will be called \emph{real} throughout the paper. It is well-know that $f\in{\NN_0}$ in fact
satisfies~$\Im(f(z))>0$ for all $\Im(z)>0$ unless it is identically equal to a real constant.
The elements $f$ of $\NN_0$ are characterized by the canonical integral representation
\begin{equation}\label{eq:N0canonical}
f(z)=\nu_1z+\nu_2+\int_\R\left(\frac{1}{t-z}-\frac{t}{t^2+1}\right)\,d\sigma(t),
\end{equation}
where $\nu_1\ge0$, $\nu_2\in\R$ and $d\sigma(t)$ is nonnegative locally finite Borel measure on
$\R$ such that $\int_{\R}d\sigma(t)/(1+t^2)<\infty$. If $f\in{\NN_0}$ is holomorphic in
$I\cup(\C\setminus\R)$ for an open interval $I\!\subset\!\R$ then it is a consequence%
\footnote{Another reasoning may be found in the beginning of the proof of
    Lemma~\ref{lm:NkappaM2a} on page~\pageref{pg:proof_lm:NkappaM2a}.}
% \textcolor{red}{[BakanHedenmalm, ref.13]}
of the Schwarz reflection principle that $\mathrm{supp}(\sigma)\subset\R\!\setminus\!{I}$. In
particular, it can be shown~\cite[pp.~127--128]{Akhiezer} that $f\in\mathcal{S}$ iff both
$f\in{\NN_0}$ and $zf\in{\NN_0}$. Another way to characterize the elements of $\NN_0$ is in
terms of the Hermitian form \cite[Chapter~3]{Akhiezer}
\begin{equation}\label{eq:quadratic_form}
    [\xi_1,\dots,\xi_n]\cdot H_f\cdot
    \begin{bmatrix}
        \,\overline\xi_1\\
        \vdots\\
        \,\overline\xi_n
    \end{bmatrix}
    ,
    \quad\text{where}\quad
    H_f=H_f(z_1,\dots,z_n)
    \coloneqq \left[\frac{f(z_k)-\overline{f(z_j)}}{z_k-\overline{z_j}}\right]_{k,j=1}^{n}.
\end{equation}
Namely, $f\in\NN_0$ if and only if for every choice of the points~$z_1,\dots,z_n$ this Hermitian
form has no negative squares, that is~$H_f$ has no negative eigenvalues. We call~$H_f$ the
Pick matrix of~$f(z)$.
Based on this characterization Krein and Langer \cite{KL} defined a generalized class known
as~$\NN_{\kappa}$:
\begin{definition} $f\in \NN_\kappa$ whenever it is meromorphic
in~$\mathbb{C}_+\coloneqq\{z\in\C:\Im z>0\}$, the Hermitian form \eqref{eq:quadratic_form} has at most~$\kappa\in\N_0$
negative squares for arbitrarily chosen~$n\in\N$ and~$(z_k)_{k=1}^n\subset \mathbb{C}_+$, and
there are precisely~$\kappa$ negative squares for a certain choice of~$n$ and~$(z_k)_{k=1}^n$.
Here~$\N_0\coloneqq\{0,1,\dots\}$ and~$\N\coloneqq\{1,2,\dots\}$.
\end{definition}
Functions in generalized Nevanlinna classes may be represented in a form more general than~\eqref{eq:N0canonical}. One of such representations was given by Krein and Langer~\cite{KL,KL3}, the other
representation~\cite{DLLS} is cited here in Theorem~\ref{th:DLLS} and may be combined with the
formula~\eqref{eq:N0canonical}.

\begin{definition}
We will say that~$f\in \NN_\kappa^\lambda$ whenever the function~$f(z)$ lies
in~$\NN_\kappa$ and simultaneously~$zf(z)$ lies in~$\NN_\lambda$, cf.~\cite{Derkach,DeKo}.
\end{definition}
In this new notation the Stiletjes class is $\mathcal{S}=\NN_0^0$. Kre\u{\i}n and Langer also considered the
case~$\NN_\kappa^+ \coloneqq \NN_\kappa^0$ with an integral representation remarkably simpler
than for~$\NN_\kappa$, see~\cite{Dyach} for the details. There is a moment problem corresponding
to the class~$\NN_\kappa^\lambda$ in the same way as the Stieltjes moment problem corresponds
to~$\mathcal S$. Analogously to the Stieltjes case, its solutions may be studied via continued
fractions, see e.g.~\cite{KL3,DeKo}.

It is convenient to introduce the unions
of the generalized Nevanlinna classes as follows
\begin{equation}\label{eq:Nevanlinna-U}
    \NU=\bigcup\limits_{\kappa\ge0}\NN_\kappa
    \quad\text{and}\quad
    \SU=\bigcup\limits_{\kappa,\lambda\ge0}\NN_\kappa^\lambda.
\end{equation}
This paper gives a necessary and sufficient condition for $\pm R_{n_1,n_2,m}(z)$ to belong to
the class~$\SU$. In particular, we prove that the Gauss ratio $R_{0,1,1}\in\SU$ or $-R_{0,1,1}\in\SU$ for all real values of parameters. For general shifts, our
condition is given in terms of positivity of a certain explicitly given polynomial~$P_r(x)$ of
degree
\[
    r=(n_1+n_2-m)_{+}+(m)_{+}-\min(n_1,n_2)-1,
    \quad\text{where}\quad
    (x)_+\coloneqq\max\{0,x\}.
\]
We also determine a constant~$B_{n_1,n_2,m}$ such that
$B_{n_1,n_2,m} P_r\big(1/(z+\omega)\big)R_{n_1,n_2,m}(z+\omega)$ belongs
to~$\SU$ for each~$\omega\le 1$. In fact, an explicit factor times~$B_{n_1,n_2,m} P_r(1/z)$ turns out to be equal to
$\Im(R_{n_1,n_2,m}(x\pm{i0}))$ for the values of $x>1$. This further enabled us to derive an
integral representation for $R_{n_1,n_2,m}(z)$ under certain asymptotic restrictions and the
assumption that~$R_{n_1,n_2,m}(z)$ is analytic in the cut plane~$\C\setminus[1,+\infty)$ and on
the banks of the branch cut. The latter amounts to non-vanishing of~${ }_2F_1(a,b;c;z)$ there,
which can be rendered in terms of the parameters~$a$, $b$,~$c$ by a remarkable theorem due to
Runckel~\cite{Runckel}.

Given values of $n_1$, $n_2$, $m$, our integral representation provides a criterion of whether
$R_{n_1,n_2,m}(z)$ belongs to the Stieltjes class~$\mathcal{S}$ and its representing measure.
Finally, to exemplify our results, for 15 triples $(n_1,n_2,m)$ we furnish explicit integral
representations of the ratios~$R_{n_1,n_2,m}(z)$ and conditions of their applicability, as well
as conditions for $\pm R_{n_1,n_2,m}(z)$ to belong to~$\SU$.

In the course of the proof we also establish a number of general facts regarding the classes
$\NN_{\kappa}^{\lambda}$ which are of their own merit. A part of these facts is scattered over
the literature and often presented from a different viewpoint; another part may be considered
folklore, so we were unable to locate corresponding proofs elsewhere. In particular, we give
conditions for Jacobi and Stieltjes continued fractions to belong to~$\SU$ and an algorithm for
computing their indices $\kappa$, $\lambda$. Moreover, we study integral representations for the
class~$\SU$ and deduce conditions for the combination $r_1f+r_2$ to be in $\SU$ once $f\in\SU$
and $r_1$, $r_2$ are rational. To make our exposition more or less self-contained, certain basic
properties of generalized Nevanlinna classes are also established.

% Namely, we give conditions for Jacobi and Stieltjes continued fractions to belong to~$\SU$ and
% an algorithm for computing their indices $\kappa$, $\lambda$. Furthermore, we prove a new
% factorization theorem for the class $\NN_{\kappa}$ and formulate conditions for the
% combination $r_1f+r_2$ to be in $\SU$ once $f$ and $r_1$, $r_2$ are rational.

The ratios of Gauss hypergeometric functions is a recurring theme in literature. Their
relation to certain generalized Nevanlinna classes is probably best studied for the case of
the Jacobi polynomials
\[
    P_n^{(\alpha,\beta)}(z)={}_2F_1\big(-n,1+\alpha+\beta+n;\alpha+1;(1-z)/2\big),
\]
where it reflects the interlacing or partial interlacing of their zeros (see e.g.~\cite{ADL,DJ}).
Some researchers also go beyond the polynomials case, see for instance~\cite{Segura}. A general result
connected to our Theorem~\ref{th:R011Nkappa} was obtained in~\cite{Derevyagin}.

There are many intriguing open questions related to our work. For instance, the case when the
shifts are no longer integer is also of interest for applications, but requires additional
tools. For the Jacobi polynomials, certain relevant results are presented in~\cite{DJM}. For the
non-polynomial case, there are only very fragmentary results of this type, such
as~\cite[Lemma~4.5]{LSW}.

Another interesting topic is rational approximations to the ratios~$R_{n_1,n_2,m}(z)$, which are
connected to orthogonal or multiple orthogonal polynomials. As we mentioned above, Wimp gave
explicit expressions for the diagonal Pad\'{e} approximants to the Gauss ratio in \cite{Wimp87}.
These approximants are related to the so-called associated Jacobi polynomials defined as
solution of the same recurrence relation but with the index shifted by a positive number (for
integer shifts but more general weights see~\cite{Nevai}). The orthogonality measure computed by
Wimp bears noticeable resemblance to the measures appearing in our work. Further results in this
direction are in~\cite{WimpBeck}. We would also like to mention the recent work~\cite{LiLo}
investigating multiple orthogonal polynomials with respect to a pair of particular measures,
that corresponds to $R_{0,1,0}(z)$. A compelling problem is to obtain
counterparts of our results for the ratios of the generalized hypergeometric
functions~${ }_pF_q$ which, for certain integer shifts, have explicitly known branched continued
fractions generalizing the Gauss ratio~\cite[Sections~13--14]{PSZ}. A similar problem may be
posed, mutatis mutandis, for basic hypergeometric functions, cf.~\cite[Section~15]{PSZ}. The
basic analogue of the Gauss continued fraction has been considered in detail in
\cite{AgarwalSahoo,BariczSwami}.

This paper is organized as follows. In Section~\ref{sec:boundary-values-n-integrals}, we obtain
the values of~$\Im(R_{n_1,n_2,m}(x\pm{i0}))$ for~$x>1$ and exploit them to write the
corresponding integral representation. The basic ingredients are Theorem~\ref{th:2F1zeros}
presenting a corollary of~\cite{Runckel} and Lemma~\ref{lemma:Schwarz_asympt}
connecting~$\Im(R_{n_1,n_2,m}(x\pm{i0}))$ with a Cauchy-type integral. Then
Subsection~\ref{sec:asymptotic-behaviour} deals with the asymptotic behaviour
of~$R_{n_1,n_2,m}(z)$ near the point~$z=1$ and at infinity. Theorems~\ref{th:2F1identity}
and~\ref{th:2F1ratioboundary} from Subsection~\ref{sec:main-theorems} are at the heart of our
work: they derive a formula for~$\Im(R_{n_1,n_2,m}(x\pm{i0}))$. The corresponding integral
representations are established in Theorem~\ref{th:2F1ratio-repr}.
Section~\ref{sec:gener-nevanl-class} is concerned with the generalized Nevanlinna classes:
Theorem~\ref{th:RnmNkappa} answers when~$\pm R_{n_1,n_2,m}\in\SU$,
Theorem~\ref{th:B_P_Rnm_Nkappa} exhibit expressions that belong to~$\SU$ for all real~$a,b,c$
such that~$-c\notin\N_0$, and Theorem~\ref{th:R011Nkappa} settles the case of the Gauss
ratio~$R_{0,1,1}(z)$. Subsection~\ref{sec:line-fract-transf} briefly introduces properties of
generalized Nevanlinna classes and their relation to continued fractions, which imply then the
proof of Theorem~\ref{th:R011Nkappa}. Subsection~\ref{sec:line-fract-transf} uses the
multiplicative representation~\cite{DLLS} to obtain the proofs of Theorems~\ref{th:RnmNkappa}
and~\ref{th:B_P_Rnm_Nkappa}. The last section of this paper -- Section~\ref{sec:examples} --
illustrates our study with examples.

In Part II, we plan to give more general integral representations that handle singularity at~$z=1$ and possible zeros of the denominator. 
We hope it will also provide us with a method for calculating ~$\varepsilon\in\{-1,1\}$ and~$\kappa,\lambda$ such that~$\varepsilon R_{n_1,n_2,m}\in\NN_\kappa^\lambda$ whenever~$\pm R_{n_1,n_2,m}\in\SU$.
Part~II will also give deeper and more detailed examination of the particular examples.

\section{Boundary values and integral representation}\label{sec:boundary-values-n-integrals}

Given~$\xi\in\R$, let~$\lfloor\xi\rfloor$ be the maximal integer number~$\le\xi$. Note that
if~$\xi$ is non-integer, then $\lfloor -\xi \rfloor=-\lfloor\xi\rfloor-1$. We need the following
corollary of an important result due to Runckel \cite{Runckel}.

\begin{theorem}\label{th:2F1zeros}
Suppose $c\ne0$ and any of the following conditions is true:

\begin{enumerate}[\upshape (I)]
\item\label{item:Run1} $-1<\min(a,b)\le{c}\le\max(a,b)\le0$
\item\label{item:Run2} $-1<\min(a,b)\le0\le\max(a,b)\le{c}$
\item\label{item:Run3} $-1<c\le\min(a,b)\le0\le\max(a,b)<c+1$
\item\label{item:Run4} $0\le\min(a,b)\le{c}$ and  $0\le\max(a,b)<c+1$
\item\label{item:Run5} $a,b,c,c-a,c-b$ are non-integer negative numbers, such
    that~$\lfloor\xi_1\rfloor+1=\lfloor\xi_4\rfloor$
    and~$\lfloor\xi_2\rfloor=\lfloor\xi_3\rfloor$, where~$\xi_1,\dots,\xi_4$ are the
    numbers~$a,b,c-a,c-b$ taken in  non-decreasing order:
    \[
        \min(a,b,c-a,c-b) = \xi_1\le\xi_2\le\xi_3\le\xi_4 = \max(a,b,c-a,c-b).
    \]
\end{enumerate}
Then $z\to{}_2F_1(a,b;c;z)$ does not vanish in $\C\setminus[1,\infty)$ as well as on the banks of
the branch cut.

For any real numbers $a,b,c$ with $-c\notin\N_0$, the number of zeros of the function $z\to{ }_2F_1(a,b;c;z)$ in $\C\setminus[1,\infty)$ and on the banks of the branch cut is finite.
\end{theorem}

\begin{remark}
    The conditions of the theorem are not necessary as is seen from the example
    ${}_2F_1(-n,c,c,z)=(1-z)^n$,
    $n\in\N$, $-c\notin\N_0$, non-vanishing in $\C\setminus[1,\infty)$. One can show that, apart from this example, the conditions are in fact necessary.
\end{remark}

\begin{remark}
    Under condition~\eqref{item:Run5}, one necessarily has~$c-\xi_4=\xi_1<\xi_2$
    and~$c-\xi_2=\xi_3<\xi_4$. Indeed: $\xi_1+\xi_4=c=\xi_2+\xi_3$ in view
    of~$a+(c-a)=c=b+(c-b)$. So, if we had one of the equalities~$\xi_1=\xi_2$ and~$\xi_3=\xi_4$,
    we automatically had the other. However, assuming the last two equalities together will
    contradict to~$\lfloor\xi_1\rfloor+1=\lfloor\xi_4\rfloor$ on account
    of~$\lfloor \xi_2\rfloor=\lfloor \xi_3\rfloor$.
\end{remark}

In fact,~\eqref{item:Run5} is generated by the following two basic cases
\[
    \begin{aligned}
    &-k-1<a<\min(b,c-b)\le \max(b,c-b)<-k<c-a<-k+1
    % \quad\text{for all}
    ,\quad k\in\N,
    \quad\text{and}
    \\
    &-k-1<a<-k<\min(b,c-b)\le \max(b,c-b)<c-a<-k+1
    % \quad\text{for all}
    ,\quad k\in\N,
    \end{aligned}
\]
further extended through the symmetry~$a\leftrightarrow b$ and Euler's
transformation \eqref{eq:Euler_ID} exchanging $(a,b)\leftrightarrow (c-a,c-b)$.

\begin{proof}[Proof of Theorem~\ref{th:2F1zeros}] The assertion that the number of zeros is finite follows
from~\cite[p.~54]{Runckel} in the degenerate case~$\{-a,-b,a-c,b-c\}\cap\N_0\ne\varnothing$,
and from~\cite[Theorem, p.~56]{Runckel} otherwise.

For another assertion, note that if $a=0$ or $b=0$ while $c\ne0$, then ${}_2F_1(a,b;c;z)=1$
and the claim is obvious. Also if $a=c$, then ${}_2F_1(a,b;c;z)=(1-z)^{-b}$
and similarly for~$b=c$, so that again the conclusion is true. Next,
according to \cite[Theorem, p.~56]{Runckel}
$$
\#\{\text{zeros of}~{}_2F_1(a,b;c;z)~\text{in}~\C\setminus[1,\infty)\}=0
$$ 
if $c\ge{a+b}$ and $a\ge{b}>0$. Exchanging the roles of $a$, $b$ we get the condition $a,b>0$
and $c\ge{a+b}$ which implies $a<c$, ~$b<c$. Note next, that the number of zeros of
${}_2F_1(a,b;c;z)$ in $\C\setminus[1,\infty)$ is equal to the number of zeros of
${}_2F_1(c-a,c-b;c;z)$ by the celebrated Euler's identity
\begin{equation} \label{eq:Euler_ID}
    {}_2F_1(a,b;c;z)=(1-z)^{c-a-b}{}_2F_1(c-a,c-b;c;z).
\end{equation}
Hence, we conclude that $c-a,c-b>0$ and $c\ge{c-a+c-b}$ ($\Leftrightarrow~a+b\ge{c}$) are also
sufficient conditions. Union of these two sets of conditions in view of the aforementioned
degenerate case yields $0\le a,b\le{c}$.
% here we may have~$c>a,b$ and~$a+b ≥ c$  which imply~$a+b > a$ and $a+b>b$,
% and therefore~$c>a,b>0$ and~$a+b ≥ c$

Next, again by \cite[Theorem, p.~56]{Runckel} if $a<0$, ~$b\ge{a}$, ~$c>a$ and $c\ge{a+b}$, then
$$
\#\{\text{zeros of}~{}_2F_1(a,b;c;z)~\text{in}~\C\setminus[1,\infty)\}
=\lfloor-a\rfloor+\big(1+S\big)/2,\quad S\coloneqq \sign(\Gamma(a)\Gamma(b)\Gamma(c-b)\Gamma(c-a)).
$$
This implies that for $-1<a<0$ and $\Gamma(b)\Gamma(c-b)>0$ the quantity on the right hand side is equal to zero.  Condition
$\Gamma(b)\Gamma(c-b)>0$ is satisfied if $c>b>0$ or if $-1<c<b<0$. Combining these conditions and exchanging the roles of $a$, $b$  (in view of the remark about $ab=0$), we get the conditions \eqref{item:Run1} and \eqref{item:Run2}.  Finally, applying conditions \eqref{item:Run1} and \eqref{item:Run2} to the function ${}_2F_1(c-a,c-b;c;z)$ we arrive at \eqref{item:Run3} and $0\le\min(a,b)\le{c}\le\max(a,b)<c+1$.
Combined with $0\le a,b\le{c}$ this leads to~\eqref{item:Run4}.

To obtain the condition~\eqref{item:Run5}, let us start with the case~$\xi_1=a$. Then~$a\le b$,
and hence~$c-b\le c-a$. Moreover,~$a\le c-b$ implies~$b\le c-a$ and further~$b+a\le c$.
Consequently, we have~$\xi_4=c-a$ and the assumptions of~\eqref{item:Run5} yield
\[
    \lfloor a\rfloor+1=\lfloor c-a\rfloor,\quad
    \lfloor b\rfloor=\lfloor c-b\rfloor,\quad\text{and}\quad
    a<\min(b,c-b)\le \max(b,c-b)<c-a<0.
\]
None of the numbers~$a,b,c-a,c-b$ is integer, so the above relations
give~$\Gamma(a)\Gamma(c-a)<0$ and~$\Gamma(b)\Gamma(c-b)>0$, which imply that~$S$ defined as
above equals~$-1$. Due to~$a+b\le c<a\le b<0$, we may apply the last option in~\cite[Theorem,
p.~56]{Runckel} asserting
\begin{equation}\label{eq:Runckel3}
    \#\{\text{zeros of}~{}_2F_1(a,b;c;z)~\text{in}~\C\setminus[1,\infty)\}
    =\lfloor-a\rfloor+\big(1+S\big)/2+S\cdot\lfloor a-c+1\rfloor.
\end{equation}
The right-hand side is equal to $-\lfloor a\rfloor -1 +\big(1-1\big)/2+\lfloor c-a\rfloor=0$, so
the case~$\xi_1=a$ is proved.%
\footnote{In fact, our reasoning is reversible: if~$a+b\le c<a\le b<0$,
    ~$\{-a,-b,a-c,b-c,c\}\cap\N_0=\varnothing$ and the right-hand side of~\eqref{eq:Runckel3} is
    zero, then we necessarily arrive at~\eqref{item:Run5} with~$\xi_1=a$. One only needs to note
    that~\eqref{eq:Runckel3} implies~$S=-1$ in this case.}

The case~$\xi_1=b$ follows by exchanging the roles of~$a$ and~$b$. Now, if~$\xi_1=c-a$
or~$\xi_1=c-b$, we apply Euler's identity~\eqref{eq:Euler_ID} to reduce the proof
of~\eqref{item:Run5} to, respectively, the cases~$\xi_1=a$ or~$\xi_1=b$.
\end{proof}

Integral representations of the forms~\eqref{eq:Gauss_St_int} or~\eqref{eq:N0canonical} may
already be too restrictive for the Gauss ratio~$G(z)=R_{0,1,1}(z)$ when none of the conditions
\eqref{item:Run1}--\eqref{item:Run5} of Theorem~\ref{th:2F1zeros} is satisfied. In particular,~\eqref{eq:Gauss_St_int}
and~\eqref{eq:N0canonical} are analytic in~$\C\setminus\R$, while~$R_{0,1,1}(z)$ may have
complex poles for certain values of~$a,b,c$. Moreover, the measure~$d\sigma(t)$ related
to~$R_{0,1,1}(z)$ via the Stieltjes-Perron formula~\eqref{eq:Stieltjes-Perron} may grow too fast
for the integral in~\eqref{eq:N0canonical} to be convergent. Introduction of the arbitrary
integer shifts~$n_1,n_2,m$ makes the problem even more complicated, since the prospective
measure determined via the Stieltjes-Perron formula may easily be signed, see
Section~\ref{sec:examples} for examples (e.g. Example~\hyperref[example:5]{5}).

To account the complex poles is not an easy task, unless we exactly know the zeros of the
denominator of~$R_{n_1,n_2,m}(z)$. As a certain compromise, Section~\ref{sec:gener-nevanl-class}
shows that the ratio~$R_{n_1,n_2,m}(z)$ may always be expressed as the product~$h(z)g(z)$,
where~$h(z)$ is a real rational function and~$g\in\NN_0$ has the form~\eqref{eq:N0canonical} (in
fact even the form~\eqref{eq:Gauss_St_int} with~$\gamma=1$ up to addition of another real rational
function, see Theorems~\ref{th:B_P_Rnm_Nkappa} and~\ref{th:DLLS}).

This section aims at constructing explicit integral representations for~$R_{n_1,n_2,m}(z)$ in
the case when there are no poles in~$\C\setminus[1,\infty)$ as well as on the banks of the
branch cut. If so, the corresponding signed measure (or charge) turns to be supported
on~$[1,+\infty)$ and have an analytic density. Thus, to obtain the integral representation we
only need to deal with the asymptotic behaviour of~$R_{n_1,n_2,m}(z)$ near the points~$z=1$
and~$z=\infty$. Such representations up to rational correction terms may be expressed via the
boundary values using the standard Schwarz formula; more specifically, the following fact is
true:

\begin{lemma}\label{lemma:Schwarz_asympt}
    Let~$f(z)$ be a real analytic function defined in the cut plane~$\C\setminus[1,+\infty)$ and suppose that $u(x)\coloneqq\frac 1\pi \Im f(x + i0)$ is continuous on $(1,+\infty)$. Suppose that there exists~$n\in\N_0$ such
    that
    \begin{equation}\label{eq:Schwarz_asympt}
        \lim_{|z-1|\to 1}\big|f(z)(1-z)\big|=\lim_{|z|\to\infty}\big|f(z)z^{-n}\big|=0
    \end{equation}
    and~$u(x)x^{-n-1}$ is absolutely integrable over~$(1,+\infty)$. Then
    \begin{equation}\label{eq:Schwarz_plus}
        f(z)
        = \sum_{k=0}^{n-1}\frac{f^{(k)}(0)z^k}{k!}
        + z^n \int_{1}^{+\infty}\frac{u(x)\,dx}{(x-z)x^n}.
    \end{equation}
\end{lemma}
\begin{proof}
    Let~$C$ be the closed contour consisting of a small circle around the point~$z=1$ of
    radius~$\epsilon<1/2$, then (the upper bank of) the interval~$(1+\epsilon+i0,1/\epsilon+i0)$ followed by a large
    circle~$|z|=1/\epsilon$ and (the lower bank of) the interval~$(1+\epsilon-i0,1/\epsilon-i0)$. The contour is traversed in the direction that leaves the bounded domain inside it on the left (so that the large circle is traversed counterclockwise). The Cauchy formula for the
    Taylor coefficients of~$f(z)$ implies
    \[
        \begin{aligned}
            2\pi i f^{(n+k)}(0)/(n+k)! &= {\textstyle\oint_C}f(z)z^{-(n+k+1)}dz
            \\
            &=
            \oint_{|z-1|=\epsilon} \frac{f(z)\,dz}{z^{n+k+1}}
            +\int_{1+A}^{1/\epsilon} \frac{f(x+i0)-f(x-i0)}{x^{n+k+1}}dx
            +\oint_{|z|=1/\epsilon} \frac{f(z)\,dz}{z^{n+k+1}}
            .
        \end{aligned}
    \]
    Here~$k\in\N_0$; note also that~$f(x+i0)-f(x-i0)=2i\Im f(x+i0)$. On letting~$\epsilon\to+0$, the
    first and the last integrals on the right-hand side vanish due to~\eqref{eq:Schwarz_asympt},
    and hence
    \[
        \frac{f^{(n+k)}(0)}{(n+k)!}
        =
        \frac1{\pi}\int_{1}^{+\infty} \frac{\Im f(x+i0)}{x^{n+k+1}}dx
        =
        \int_{1}^{+\infty} \frac{u(x)\,dx}{x^{n+k+1}}
        \eqqcolon c_k
        ,\quad k\in\N_0.
    \]
    The Taylor series uniformly converges on compact subsets of the unit disc (in
    fact,~$c_k\to 0$ as~$k\to\infty$ since for each~$x>1$ the integrands monotonically tend to
    zero). Therefore, if~$|z|<1$ we have
    \[
        \frac 1{z^n}\left(f(z)-\sum_{k=0}^{n-1}\frac{f^{(k)}(0)z^k}{k!}\right)
        = \sum_{k=0}^{\infty}c_kz^k
        = \int_{1}^{+\infty}\left(
            \sum_{k=0}^{\infty}\frac{z^k}{x^{k+1}}\right)\frac{u(x)\,dx}{x^n}
        =
        \int_{1}^{+\infty}\frac{u(x)\,dx}{(x-z)x^n}.
    \]
    The ratio~$x/(x-z)$ is bounded in~$z$ on compact subsets of~$\C\setminus[1,+\infty)$ uniformly in~$x>1$, so the
    integral on the right-hand side uniformly converges there to an analytic function. This
    analytic function coincides with~$f(z)$ in the unit disc, and hence
    in~$\C\setminus[1,+\infty)$.
\end{proof}

\subsection{Asymptotic behaviour}\label{sec:asymptotic-behaviour}

In this section, we will record the behavior of $R_{n_1,n_2,m}(z)$ in the neighborhood of the singular points $z=1$ and $z=\infty$.  
We begin with some remarks.  Define
\begin{equation}\label{eq:A-defined}
\begin{split}
&A(x_1,x_2,x_3)=\frac{\Gamma(x_3)\Gamma(x_2-x_1)}{\Gamma(x_3-x_1)\Gamma(x_2)},~~A_1=A(a+n_1,b+n_2,c+m),
\\
&A_2=A(b+n_2,a+n_1,c+m),~~A_3=A(a,b,c),~~A_4=A(b,a,c).
\end{split}
\end{equation}
Is it clear that $A(x_1,x_2,x_3)=0$ iff $x_1-x_3\in\N_0$ or $-x_2\in\N_0$. In this case ${}_2F_{1}(x_1,x_2;x_3;z)={}_2F_{1}(x_2,x_1;x_3;z)$ reduces to a polynomial possibly times a power of $1-z$. Hence, the condition $A_3A_4=0$ is equivalent to the fact that ${}_2F_1(a,b;c;z)$ reduces to a polynomial, possibly times a power of $1-z$; similar claim holds for $A_1A_2=0$ and ${}_2F_1(a+n_1,b+n_2;c+m;z)$.
Note, finally, that  the condition $A_1A_2A_3A_4\ne0$  is equivalent to
\begin{equation}\label{eq:noninteger}
\{a,a+n_1,b,b+n_2,c-a,c+m-a-n_1,c-b,c+m-b-n_2\}\cap-\N_0=\emptyset.
\end{equation}

Denote, as customary, $(x)_{-}=\min(x,0)$. Our first lemma deals with the singular point $z=1$.  
\begin{lemma}\label{lm:ratio-asymp1}
Suppose that condition \eqref{eq:noninteger} holds true for some $a,b,c\in\C$ and  $n_1,n_2,m\in\Z$.  Denote $\eta=c-a-b$, $q=m-n_1-n_2$ and write $\delta_{x,y}$ for the Kronecker delta.  Then 
\begin{equation}\label{eq:fat1}
R_{n_1,n_2,m}(z)=M\frac{(1-z)^{(\eta+q)_{-}}[1-\delta_{\eta+q,0}+\delta_{\eta+q,0}\log(1-z)]}{(1-z)^{(\eta)_{-}}[1-\delta_{\eta,0}+\delta_{\eta,0}\log(1-z)]}
(1+o(1))
\end{equation}
as $z\to1$ with some constant $M\ne0$ independent of $z$.  If \eqref{eq:noninteger} is violated, formula \eqref{eq:fat1}  should be modified as follows:

 \begin{compactenum}[\upshape (a)]
\item If $-a\in\N_{0}$ and$\:\smalldiv$or $-b\in\N_{0}$, then the denominator should be replaced
    by $1$.
\item If $-(a+n_1)\in\N_0$ and$\:\smalldiv$or $-(b+n_2)\in\N_0$, then the numerator should be
    replaced by $1$.
\item If $-a,-b\notin\N_0$ but $a-c\in\N_0$ and$\:\smalldiv$or $b-c\in\N_{0}$, then the
    denominator should be replaced by $(1-z)^{\eta}$.
\item If $-a-n_1,-b-n_2\notin\N_0$, but $a+n_1-c-m\in\N_{0}$ and$\:\smalldiv$or
    $b+n_2-c-m\in\N_{0}$, then the numerator should be replaced by $(1-z)^{\eta+q}$.
\end{compactenum}
\end{lemma}
\begin{proof}  Suppose first  that \eqref{eq:noninteger} is satisfied.
Then, if $\eta=c-a-b\notin\Z$, according to \cite[{\href{https://dlmf.nist.gov/15.8\#E4}{(15.8.4)}}]{DLMF} we have
$$
{}_2F_1(a,b;c;1-z)=\frac{\Gamma(c)\Gamma(c-a-b)}{\Gamma(c-a)\Gamma(c-b)}{}_2F_1(a,b;1+\eta;z)
+\frac{\Gamma(c)\Gamma(a+b-c)}{\Gamma(a)\Gamma(b)}z^{\eta}{}_2F_1(c-a,c-b;1-\eta;z).
$$
If $\eta=c-a-b=s\in\N_0$, then according to \cite[2.10(12-13)]{HTF1} or \cite[{\href{https://dlmf.nist.gov/15.8\#E10}{(15.8.10)}}]{DLMF} we have
\begin{multline*}
{}_2F_1(a,b;c;1-z)=\frac{\Gamma(c)\Gamma(c-a-b)}{\Gamma(c-a)\Gamma(c-b)}
\sum\limits_{n=0}^{s-1}\frac{(a)_n(b)_n}{(1-\eta)_{n}n!}z^n
+\frac{(-1)^s\Gamma(c)}{\Gamma(a)\Gamma(b)s!}z^{\eta}\sum\limits_{n=0}^{\infty}\frac{(c-a)_n(c-b)_n}{(1+\eta)_{n}n!}H_nz^n
\\
-\frac{(-1)^s\Gamma(c)}{\Gamma(a)\Gamma(b)s!}z^{\eta}\log(z){}_2F_1(c-a,c-b;1+\eta;z),
\end{multline*}
where the sum over the empty index set equals zero, and
$$
H_n=\psi(n+1)+\psi(n+s+1)-\psi(a+n+s)-\psi(b+n+s),~~~~\psi(z)=\Gamma'(z)/\Gamma(z).
$$
If $\eta=c-a-b=-s$, $s\in\N_0$, according to \cite[2.10(14-15)]{HTF1} we have
\begin{multline*}
\hspace{-1em}{}_2F_1(a,b;c;1-z)=\frac{\Gamma(c)\Gamma(a+b-c)z^{\eta}}{\Gamma(a)\Gamma(b)}\sum\limits_{n=0}^{s-1}
\frac{(c-a)_n(c-b)_n}{(1+\eta)_{n}n!}z^n+\frac{(-1)^s\Gamma(c)}{\Gamma(c-a)\Gamma(c-b)s!}
\sum\limits_{n=0}^{\infty}\frac{(a)_n(b)_n}{(1-\eta)_{n}n!}\hat{H}_nz^n
\\
-\frac{(-1)^s\Gamma(c)}{\Gamma(c-a)\Gamma(c-b)s!}\log(z){}_2F_1(a,b;1-\eta;z),
\end{multline*}
and
$$
\hat{H}_n=\psi(n+1)+\psi(n+s+1)-\psi(a+n)-\psi(b+n).
$$
These formulae imply that
$$
{}_2F_1(a,b;c;1-z)=\left\{\!\!\!\begin{array}{l}
A(1+\alpha_{1}z+\alpha_2z^2+\cdots)+Bz^{\eta}(1+\beta_{1}z+\beta_2z^2+\cdots),~~~~~~\eta\notin\Z;
\\[5pt]
\hat{A}(1+\hat{\alpha}_{1}z+\hat{\alpha}_2z^2+\cdots)+\hat{B}z^{\eta}\log(z)(1+\hat{\beta}_{1}z+\hat{\beta}_2z^2+\cdots),~~~\eta\in\N_{0};
\\[5pt]
\tilde{A}z^{\eta}(1+\tilde{\alpha}_{1}z+\tilde{\alpha}_2z^2+\cdots)+\tilde{B}\log(z)(1+\tilde{\beta}_{1}z+\tilde{\beta}_2z^2+\cdots),~~-\eta\in\N,
\end{array}
\right.
$$
where the constants $A,\hat{A},\tilde{A},B,\hat{B},\tilde{B}$ do not vanish due to condition \eqref{eq:noninteger}. In a similar fashion,
\begin{multline*}
%{}_2F_1\!\left(\!\!\begin{array}{l}a+n_{1},b+n_{2}\\c+m\end{array}\!\!\vline\:1-z\!\right)
{}_2F_1\left(a+n_{1},b+n_{2};c+m;1-z\right)
\\=
\!\left\{\!\!\!\begin{array}{l}
C(1+\delta_{1}z+\delta_2z^2+\cdots)+Dz^{\eta+q}(1+\gamma_{1}z+\gamma_2z^2+\cdots),~~~~\eta+q\notin\Z;
\\[5pt]
\hat{C}(1+\hat{\delta}_{1}z+\hat{\delta}_2z^2+\cdots)+\hat{D}z^{\eta+q}\log(z)(1+\hat{\gamma}_{1}z+\hat{\gamma}_2z^2+\cdots),~\eta+q\in\N_{0};
\\[5pt]
\tilde{C}z^{\eta+q}(1+\tilde{\delta}_{1}z+\tilde{\delta}_2z^2+\cdots)+\tilde{D}\log(z)(1+\tilde{\gamma}_{1}z+\tilde{\gamma}_2z^2+\cdots),~-\eta-q\in\N,
\end{array}
\right.
\end{multline*}
where the constants $C,\hat{C},\tilde{C},D,\hat{D},\tilde{D}$ do not vanish due to condition \eqref{eq:noninteger}.
Substituting these formulae into definition \eqref{eq:gen-ratio-def} of the function $R_{n_1,n_2,m}(z)$ and analyzing the principal asymptotic term in each of the five possible cases (1) $\eta\notin\Z$; (2) $\eta\in\N_{0}$ and $\eta+q\in\N_{0}$; (3) $\eta\in\N_{0}$ and $-\eta-q\in\N$; (4) $-\eta\in\N$ and $\eta+q\in\N_{0}$; (5) $-\eta\in\N$ and $-\eta-q\in\N$, we arrive at formula \eqref{eq:fat1}.

If condition \eqref{eq:noninteger} is violated, then the claims (a)-(d) of the lemma follow from the following two facts:
(1) if $-a\in\N_{0}$ and$\:\smalldiv$or $-b\in\N_{0}$, then ${}_2F_1(a,b;c;z)$ reduces to a polynomial;
(2) If $-a,-b\notin\N_0$, but $a-c\in\N_0$ and$\:\smalldiv$or $b-c\in\N_{0}$, then Euler's transformation
$$
{}_2F_1(a,b;c;z)=(1-z)^{\eta}{}_2F_1(c-a,c-b;c;z)
$$
implies that ${}_2F_1(a,b;c;z)=(1-z)^{\eta}\times\text{polynomial}$. In view of a similar statement for ${}_2F_1(a+n_{1},b+n_{2};c+m;z)$ we arrive at the conclusions contained in claims (a)-(d) of the lemma on the basis of case-by-case analysis.
\end{proof}

The above lemma implies that in all possible cases the asymptotics takes the form
\begin{equation}\label{eq:asymp1}
R_{n_1,n_2,m}(z)=M(1-z)^{\nu}[\log(1-z)]^{\epsilon}(1+o(1))~~~\text{as}~~z\to1,
\end{equation}
where $\epsilon\in\{-1,0,1\}$ and $\nu=(\eta+q)_{-}-(\eta)_{-}$ if condition \eqref{eq:noninteger} is satisfied. Otherwise, the formula for $\nu$ is modified according to statements (a)-(d) of Lemma~\ref{lm:ratio-asymp1}.  Suppose $\nu>-1$. Take $\theta\in(-1,\nu)$.  Clearly,
$$
R_{n_1,n_2,m}(z)=M(1-z)^{\nu}[\log(1-z)]^{\epsilon}(1+o(1))=M(1-z)^{\theta}o(1)(1+o(1)),
$$
so that in a neighborhood of $z=1$ we get the estimate
\begin{equation}\label{eq:asymp1-2}
|R_{n_1,n_2,m}(z)|\le M|1-z|^{\theta}.
\end{equation}

For the neighborhood of infinity we will break the result in two sub-cases.  First, define
\begin{equation}\label{eq:A12A34}
A_{1,2}=\left\{\!\!\begin{array}{l}A_1,~\text{if}~\min(a+n_1,b+n_2)=a+n_1\\
A_2,~\text{if}~\min(a+n_1,b+n_2)=b+n_2\end{array}\right.\!\!\!,
~~~~~
A_{3,4}=\left\{\!\!\begin{array}{l}A_3,~\text{if}~\min(a,b)=a\\
A_4,~\text{if}~\min(a,b)=b\end{array}\right.\!\!\!.
\end{equation}

\begin{lemma}\label{lm:asymp-infnolog}
Suppose $a-b\notin\Z$, $-c\notin\N_0$, $A_1,\ldots,A_4, A_{1,2}, A_{3,4}$ are defined in \eqref{eq:A-defined} and \eqref{eq:A12A34}, respectively. Then the asymptotic expansion of $R_{n_1,n_2,m}(-z)$ has the form
\begin{equation}\label{eq:asymp-infnolog}
R_{n_1,n_2,m}(-z)\sim \frac{A_{1,2}}{A_{3,4}}z^{\alpha-\gamma}\left(1+\sum\limits_{k=1}^{\infty}\frac{a_k}{z^{\sigma_k}}\right)~~\text{as}~~z\to\infty,
\end{equation}
where $a_k\ne0$ and
\begin{enumerate}[\upshape (1)]
\item $\alpha=-\min(a+n_1,b+n_2)$ if $A_1A_2\ne0$, $\alpha=-a-n_1$ if $A_2=0$, $\alpha=-b-n_2$ if $A_1=0$  \emph{(}the case $A_1=A_2=0$ is impossible under conditions of the lemma\emph{)}\emph{;}
\item $\gamma=-\min(a,b)$ if $A_3A_4\ne0$, $\gamma=-a$ if $A_4=0$, $\gamma=-b$ if $A_3=0$ \emph{(}the case $A_3=A_4=0$ is impossible under conditions of the lemma\emph{)}\emph{;}
\item the term $\sigma_k$ of the positive increasing sequence $\{\sigma_k\}_{k\in\N}$ equals the $k$-th smallest element of the multi-set
$$
\left\{k_1\varepsilon+k_2\delta+k_3:~k_1\in\{0,1\},~k_2\in\N_0,~k_3\in\N_0, k_1k_2k_3\ne0\right\},
$$
where  $\varepsilon=|a+n_1-b-n_2|$ if $A_1A_2\ne0$ or $\varepsilon=0$
otherwise and $\delta=|a-b|$ if $A_3A_4\ne0$ or $\delta=0$  otherwise.
\end{enumerate}
\end{lemma}
\begin{proof} According to \cite[(2.3.12)]{AAR} or
    \cite[{\href{https://dlmf.nist.gov/15.8\#E2}{(15.8.2)}}]{DLMF} as long as $a-b\notin\Z$ we
    have
\begin{multline*}
{}_2F_1(a,b;c;-z)=A(a,b,c)z^{-a}(1+\hat{\alpha}_{1}z^{-1}+\hat{\alpha}_{2}z^{-2}+\cdots)
+A(b,a,c)z^{-b}(1+\tilde{\alpha}_{1}z^{-1}+\tilde{\alpha}_{2}z^{-2}+\cdots)
\\
=A_{3,4}z^{\gamma}\left(1+\alpha'_{1}z^{-\delta}+\alpha'_{2}z^{-\delta-1}+\cdots+\bar{\alpha}_{1}z^{-1}
+\bar{\alpha}_{2}z^{-2}+\cdots\right),
\end{multline*}
where $\gamma=-\min(a,b)$, $\delta=|a-b|$,
$$
\hat{\alpha}_j=\frac{(a)_j(a-c+1)_j}{(a-b+1)_jj!},~~~\tilde{\alpha}_j=\frac{(b)_j(b-c+1)_j}{(b-a+1)_jj!}
$$
and $A_{3,4}$ is defined in \eqref{eq:A12A34}.  Hence,
$$
[{}_2F_1(a,b;c;-z)]^{-1}
=(A_{3,4})^{-1}z^{-\gamma}\left(1+\sum_{k=1}^{\infty}\frac{\alpha_k}{z^{\hat{\sigma}_k}}\right),
$$
where $\hat{\sigma}_k$ is the $k$-th smallest number in the (multi)-set $\{k_1\delta+k_2: k_1,k_2\in\N_{0}, k_1k_2\ne0\}$

In a similar fashion,
$$
{}_2F_1(a+n_1,b+n_1;c+m;-z)=A_{1,2}z^{\alpha}\left(1+\beta'_{1}z^{-\varepsilon}+\beta'_{2}z^{-\varepsilon-1}+\cdots
+\bar{\beta}_{1}z^{-1}+\bar{\beta}_{2}z^{-2}+\cdots\right),
$$
where  $\alpha=-\min(a+n_1,b+n_2)$ if $A_1A_2\ne0$, $\alpha=-a-n_1$ if $A_2=0$, $\alpha=-b-n_2$
if $A_1=0$\emph{;} and $A_{1,2}$ is defined in \eqref{eq:A12A34}. Multiplying these
two expansions we arrive at \eqref{eq:asymp-infnolog}.
\end{proof}

The condition $a-b\notin\Z$ in Lemma~\ref{lm:asymp-infnolog} ensures that no logarithms appear in the asymptotics. If, on the contrary $a-b\in\Z$ so that also $a+n_1-b-n_2\in\Z$ the asymptotic expansions of the hypergeometric functions in both numerator and denominator of $R_{n_1,n_2,m}(-z)$ will contain logarithmic terms if \eqref{eq:noninteger} holds true (i.e. $A_1A_2A_3A_4\ne0$). We will treat this situation in the lemma below.  If \eqref{eq:noninteger} is violated, however, then either numerator (if $A_1A_2=0$) or denominator (if $A_3A_4=0$) or both reduce to a polynomial possibly time a power of $(1-z)$ in which case logarithmic terms are missing.
More precisely by Euler's transformation if $-x_1,-x_2\notin\N_0$, $n\in\N_0$, then
\begin{equation}\label{eq:asymp-Azero}
{}_2F_{1}(x_1,x_2;x_1-n;-z)=\frac{{}_2F_{1}(-n,x_1-x_2-n;x_1-n;-z)}{(1+z)^{n+x_2}}=
\frac{(1+x_2-x_1)_{n}}{z^{x_2}(1-x_1)_n}\left(1+\frac{f_1}{z}+\frac{f_2}{z^2}+\cdots\right).
\end{equation}
If $-x_1\in\N_0$ and/or $-x_2\in\N_0$ then ${}_2F_{1}(x_1,x_2;x_3;-z)$ is a polynomial of degree determined by the minimal among the integer values of $-x_{j}$, $j=1,2$.

\begin{lemma}\label{lm:asymp-inflog}
Suppose $n_2-n_1\ne{a-b}\in\Z\setminus\{0\}$, $A_1A_2A_3A_4\ne0$ and $A_{1,2}$, $A_{3,4}$ is defined \eqref{eq:A12A34}. Then the asymptotic expansion of $R_{n_1,n_2,m}(-z)$ as $z\to\infty$ has the form
\begin{equation}\label{eq:asymp-inflog}
R_{n_1,n_2,m}(-z)\sim \frac{A_{1,2}}{A_{3,4}}z^{\alpha-\gamma}\left(1+\sum\limits_{k=1}^{\min(\delta,\varepsilon)-1}\frac{a_k}{z^{k}}
+\sum\limits_{k=\min(\delta,\varepsilon)}^{\infty}\frac{a_k}{z^{k}}\left[
1+b_{1,k}\log(z)+\cdots+b_{k,k}\log^{k}(z)\right]\right),
\end{equation}
where the sum over the empty index set is zero, $a_k$ and $b_k$ are real numbers \emph{(}possibly vanishing\emph{)}, $\alpha=-\min(a+n_1,b+n_2)$\emph{;} $\gamma=-\min(a,b)$\emph{;} and $\varepsilon=|a+n_1-b-n_2|$  and $\delta=|a-b|$ are positive integers.
\end{lemma}
\begin{proof} Indeed if $|a-b|\ge1$ we apply
    \cite[{\href{https://dlmf.nist.gov/15.8\#E8}{(15.8.8)}}]{DLMF} which can be written in the
    form:
$$
{}_2F_1(a,b;c;-z)=A_{3,4}z^{\gamma}\left(1+\sum_{j=1}^{\infty}\frac{f_j}{z^{j}}
+\log(z)\sum_{k=\delta}^{\infty}\frac{e_k}{z^{k}}\right).
$$
where as before $\gamma=-\min(a,b)$, $\delta=|a-b|\in\N$, and $A_{3,4}$ is defined in \eqref{eq:A12A34}. Hence,
\begin{equation}\label{eq:ablog}
[{}_2F_1(a,b;c;-z)]^{-1}=(A_{3,4})^{-1}z^{-\gamma}\left(1+\sum_{j=1}^{\infty}\frac{\hat{f}_j}{z^{j}}
\left[1+\hat{e}_{j,1}\frac{\log(z)}{z^{\delta-1}}+\hat{e}_{j,2}\frac{\log^{2}(z)}{z^{2(\delta-1)}}
+\cdots+\hat{e}_{j,j}\frac{\log^{j}(z)}{z^{j(\delta-1)}}\right]\right).
\end{equation}
In a similar fashion,
\begin{equation}\label{eq:an1bn2log}
{}_2F_1(a+n_1,b+n_2;c+m;-z)=A_{1,2}z^{\alpha}\left(1+\sum_{j=1}^{\infty}\frac{g_j}{z^{j}}
+\log(z)\sum_{k=\epsilon}^{\infty}\frac{q_k}{z^{k}}\right),
\end{equation}
where as before $\alpha=-\min(a+n_1,b+n_2)$,  $\varepsilon=|a+n_1-b-n_2|\in\N$ and $A_{1,2}$ is defined in \eqref{eq:A12A34}.
\end{proof}
    
Note that in the above lemma $\alpha-\gamma\in\Z$. The remaining cases not covered by Lemmas~\ref{lm:asymp-infnolog} and \ref{lm:asymp-inflog} are the following.  If $a=b$, but $-a,a-c\notin\N_0$ according to
\cite[{\href{https://dlmf.nist.gov/15.8\#E8}{(15.8.8)}}]{DLMF} we have
$$
{}_2F_1(a,a;c;-z)=\frac{\log(z)\Gamma(c)}{\Gamma(a)\Gamma(c-a)z^{a}}\left(1+\frac{f_0}{\log(z)}
+\sum_{k=1}^{\infty}\frac{e_k}{z^{k}}\left[1+\frac{f_k}{\log(z)}\right]\right),
$$
so that
\begin{equation}\label{eq:aa}
[{}_2F_1(a,a;c;-z)]^{-1}=\frac{\Gamma(a)\Gamma(c-a)z^{a}}{\Gamma(c)\log(z)}\left(1+\sum\limits_{k=1}^{\infty}\frac{f_0^k}{[\log(z)]^k}
+\sum_{j=1}^{\infty}\frac{\hat{f}_j}{z^{j}}\left[1+\frac{\hat{e}_{j,1}}{\log(z)}+\cdots++\frac{\hat{e}_{j,j}}{[\log(z)]^{j}}\right]\right).
\end{equation}

In a similar fashion if $a+n_1=b+n_2$, but $-a-n_1,a+n_1-c-m\notin\N_0$ we will have
\begin{equation}\label{eq:an1an1}
{}_2F_1(a+n_1,a+n_1;c+m;-z)=\frac{z^{-a-n_1}\log(z)\Gamma(c+m)}{\Gamma(a+n_1)\Gamma(c-a+m-n_1)}\left(1+\frac{g_0}{\log(z)}
+\sum_{k=1}^{\infty}\frac{q_k}{z^{k}}\left[1+\frac{g_k}{\log(z)}\right]\right).
\end{equation}
Hence, when both $a=b$ and $a+n_1=b+n_2$, but there no nonnegative integers among the numbers $-a,a-c,-a-n_1,a+n_1-c-m$ the asymptotic expansion of $R_{n_1,n_2,m}(-z)$ is obtained by multiplication of \eqref{eq:aa} and \eqref{eq:an1an1}. If  $a=b$ but $a+n_1\ne b+n_2$ we have to multiply  \eqref{eq:aa} by \eqref{eq:an1bn2log} or, if $a+n_1=b+n_2$ but  $a\ne b$, then multiply \eqref{eq:an1an1} by \eqref{eq:ablog}.  
If $a-c\in\N_0$ or/and $a+n_1-c-m\in\N_0$ formula \eqref{eq:asymp-Azero} should be used for the corresponding asymptotic expansion. Finally, if $-a\in\N_0$ and/or $-b\in\N_0$, then  ${}_2F_1(a,b;c;z)$ reduces to a polynomial; a similar claim applies to ${}_2F_1(a+n_1,b+n_2;c+m;z)$ if $-a-n_1\in\N_0$ and/or $-b-n_2\in\N_0$.

\subsection{Main theorems}\label{sec:main-theorems}

For any integer $r$ define the Pochhammer symbol by $(z)_r=\Gamma(z+r)/\Gamma(z)$. Given three integers $n_1,n_2,m\in\Z$ define the following related quantities:
\begin{equation}\label{eq:nmrelated}
\begin{split}
\underline{n}=\min(n_1,n_2),&~~\overline{n}=\max(n_1,n_2),~~ p=(m-n_1-n_2)_{+},~~  l=(n_1+n_2-m)_{+}
\\
&r=l+(m)_{+}-\underline{n}-1=\left\{\!\!\!
\begin{array}{l}
\max(m-\underline{n},\overline{n})-1,~~m\ge0
\\
\max(-\underline{n},\overline{n}-m)-1,~~m\le0.
\end{array}
\right.
\end{split}
\end{equation}
Note that $p-l=m-n_{1}-n_{2}$ and  $r$ may only be negative when $n_1=n_2=m=0$ in which case $r=-1$.  The key fact that we will need is a more precise version of a particular case of \cite[Theorem~1]{KKAMS2019}  which (after some change of notation) reads:
\begin{theorem}\label{th:2F1identity}
 Assume that $n_1,n_2,m\in\Z$. Then 
\begin{multline}\label{eq:2F1identity}
\frac{(\gamma-\alpha)_{-n_2}(\gamma-\beta)_{m-n_2}t^{n_1}}{(\gamma-1)_{n_1-n_2+1}}
{}_{2}F_{1}\!\left(\!\begin{array}{l}1-\gamma+\alpha,1-\gamma+\beta\\2-\gamma\end{array}\!\!\vline\,\,t\right)
{}_{2}F_{1}\!\left(\!\begin{array}{l}\gamma-\alpha-n_2,\gamma-\beta+m-n_2\\\gamma+n_1-n_2\end{array}\!\!\vline\,\,t\right)
+
\\
\frac{(1-\alpha)_{-n_1}(1-\beta)_{m-n_1}t^{n_2}}{(1-\gamma)_{n_2-n_1+1}}
{}_{2}F_{1}\!\left(\!\begin{array}{l}\alpha,\beta\\\gamma\end{array}\!\!\vline\,\,t\right)
{}_{2}F_{1}\!\left(\!\begin{array}{l}1-\alpha-n_1,1-\beta+m-n_1\\2-\gamma+n_2-n_1\end{array}\!\!\vline\,\,t\right)
=\frac{t^{\underline{n}}P_{r}(t)}{(1-t)^{p}},
\end{multline}
where $P_{r}(t)$ is a polynomial of degree $r$ \emph{(}$P_{-1}\equiv0$\emph{)} given by 
\begin{subequations}\label{eq:Prexplicit}
\begin{equation}
P_{r}(t)=(-1)^{\overline{n}}\sum_{k=0}^{r}(-t)^k\!\!\sum\limits_{j=(k-p)_{+}-\overline{n}}^{k-\overline{n}}\!\!\!\!\!\!
(-1)^{j}\binom{p}{k-\overline{n}-j}K_{j}.
\end{equation}
The coefficients $K_j$ are given by  
\begin{multline}
K_{j}=\frac{(1-\alpha)_{j}(1-\beta)_{m+j}}{(1-\gamma)_{n_2+j+1}(j+n_1)!}
{}_{4}F_{3}\!\left(\!\left.\!\begin{array}{l}-j-n_1,\alpha,\beta,\gamma-1-n_2-j
\\
\alpha-j,\beta-m-j,\gamma\end{array}\!\!\vline\,\,1\right.\right)
\\
+
\frac{(\gamma-\alpha)_{j}(\gamma-\beta)_{m+j}}{(\gamma-1)_{n_{1}+j+1}(j+n_2)!}
{}_{4}F_{3}\!\left(\!\left.\!\begin{array}{l}-j-n_2,1-\gamma+\alpha,1-\gamma+\beta,1-\gamma-n_{1}-j
\\
1-\gamma+\alpha-j,1-\gamma+\beta-m-j,2-\gamma\end{array}\!\!\vline\,\,1\right.\right)\!,
\end{multline}
\end{subequations}
where we use the convention $1/(-i)!=0$ for $i\in\N$.  This polynomial can also be computed by multiplying the left hand side of \eqref{eq:2F1identity} by $t^{-\underline{n}}(1-t)^{p}$ and calculating the first $r+1$ Taylor coefficients on the left hand side.
\end{theorem}

\begin{remark}
    The particular ${}_2F_1$
    case of our general identity \cite[Theorem~1]{KKAMS2019} used above has been essentially
    discovered by Ebisu in \cite{Ebisu}. Formula \eqref{eq:2F1identity} can be derived by
    combining Theorem~3.7 with Proposition~3.4 from \cite{Ebisu}.
\end{remark}

\begin{remark}
    Our identity from \cite[Theorem~1]{KKAMS2019} does not contain explicit expression
    \eqref{eq:Prexplicit} for the polynomial $P_r$. This expression is found in \cite{CKP}.
    It can also be computed by taking the limit $q\to1$ in \cite[Theorem~2]{Yamaguchi}. For specific values of
    $n_1,n_2,m$, the second method of computing $P_{r}(t)$ indicated in the above theorem is
    more practical.
\end{remark}

In the following theorem we give an explicit formula for the imaginary part of
$R_{n_1,n_2,m}(z)$ on the banks of the branch cut $[1,\infty)$. Note that for~$x>1$ the
function~${ }_{2}F_{1}(a,b;c;x\pm i0)$ may vanish at finitely many points in the degenerate
case~$\{-a,-b,c-a,c-b\}\cap\N_0\ne\varnothing$, and does not vanish otherwise, see resp.
Theorem~\ref{th:2F1zeros} and~\cite[Lemma~2, p.~54]{Runckel}.

\begin{theorem}\label{th:2F1ratioboundary}
Assume that $n_1,n_2,m\in\Z$. In terms of notation \eqref{eq:nmrelated}, on the branch cut $x>1$ we have 
\begin{subequations}\label{eq:2F1ratioboundary_with_B}
\begin{equation}\label{eq:2F1ratioboundary}
\Im[R_{n_1,n_2,m}(x\pm i0)]=\pm{\pi}B_{n_1,n_2,m}(a,b,c)\frac{x^{l-\underline{n}-c}(x-1)^{c-a-b-l}P_{r}(1/x)}{|{}_{2}F_{1}(a,b;c;x)|^{2}},
\end{equation}
where 
\begin{equation}\label{eq:B-defined}
B_{n_1,n_2,m}(a,b,c)=-\frac{\Gamma(c)\Gamma(c+m)}{\Gamma(a)\Gamma(b)\Gamma(c-a+m-n_1)\Gamma(c-b+m-n_2)}
\end{equation}
\end{subequations}
and $P_r(t)$ is the polynomial \eqref{eq:Prexplicit} with $\alpha=a$, $\beta=1-c+a$, $\gamma=1-b+a$.  Note that $|{}_{2}F_{1}(a,b;c;x)|^2$ takes the same values on both banks of the branch cut.
\end{theorem}
\begin{proof}
The boundary values of the generalized hypergeometric function on the cut $[1,\infty)$ have been found in \cite[Theorem~3]{KPITSF2017}. For ${}_2F_1$ this theorem takes the form ($x>1$):
$$
{}_{2}F_1\left(a,b;c;x\pm{i0}\right)=-\frac{\pi\Gamma(c)}{\Gamma(a)\Gamma(b)}
G^{2,1}_{3,3}\left(\frac{1}{x}\,\vline\begin{array}{c}1,3/2,c\\a,b,3/2\end{array}\!\!\right)\pm
\pi{i}\frac{\Gamma(c)}{\Gamma(a)\Gamma(b)}G^{2,0}_{2,2}\left(\frac{1}{x}\,\vline\begin{array}{c}1,c\\a,b\end{array}\!\!\right),
$$
where $G^{m,n}_{p,q}$ is Meijer's $G$ function
\cite[section~{\href{https://dlmf.nist.gov/16.17}{16.17}}]{DLMF}. As
$$
\Im\left(\frac{\alpha+i\beta}{\gamma+i\delta}\right)=\frac{\beta\gamma-\alpha\delta}{|\gamma+i\delta|^2},
$$
denoting $\phi_{\pm}(x)=\Im[R_{n_1,n_2,m}(x\pm i0)]$, we get
\begin{multline*}
\phi_{\pm}(x)\!=\!\Im\left[\frac{{}_{2}F_{1}(a+n_1,b+n_2;c+m;x\pm i0)}{{}_{2}F_{1}(a,b;c;x\pm i0)}\right]
\!=\!\pm\frac{\pi^2\Gamma(c)\Gamma(c+m)}{|{}_{2}F_{1}(a,b;c;x)|^{2}\Gamma(a)\Gamma(b)\Gamma(a+n_1)\Gamma(b+n_2)}
\\
\biggl\{G^{2,1}_{3,3}\left(\frac{1}{x}\,\vline\begin{array}{c}1,3/2,c+m\\a+n_1,b+n_2,3/2\end{array}\!\!\right)
G^{2,0}_{2,2}\left(\frac{1}{x}\,\vline\begin{array}{c}1,c\\a,b\end{array}\!\!\right)
-G^{2,1}_{3,3}\left(\frac{1}{x}\,\vline\begin{array}{c}1,3/2,c\\a,b,3/2\end{array}\!\!\right)
G^{2,0}_{2,2}\left(\frac{1}{x}\,\vline\begin{array}{c}1,c+m\\a+n_1,b+n_2\end{array}\!\!\right)\biggr\}.
\end{multline*}
Meijer's $G$ function here can be expanded as follows \cite[Proof of Theorem~3]{KPITSF2017}:
\begin{multline*}
-G^{2,1}_{3,3}\left(t\,\vline\begin{array}{c}1,3/2,c\\a,b,3/2\end{array}\!\!\right)
\!=\!\frac{\Gamma(b-a)\Gamma(a)t^{a}}{\pi\Gamma(c-a)}
{_{2}F_{1}}\!\left(\!\begin{array}{l}a,1-c+a\\1-b+a\end{array}\!\!\vline\,\,t\right)\cos(\pi{a})
\\
+\frac{\Gamma(a-b)\Gamma(b)t^{b}}{\pi\Gamma(c-b)}
{_{2}F_{1}}\!\left(\!\begin{array}{l}b,1-c+b\\1-a+b\end{array}\!\!\vline\,\,t\right)\cos(\pi{b})
\end{multline*}
and
\begin{multline*}
G^{2,0}_{2,2}\left(t\,\vline\begin{array}{c}1,c\\a,b\end{array}\!\!\right)
=\frac{\Gamma(b-a)\Gamma(a)t^{a}}{\pi\Gamma(c-a)}
{_{2}F_1}\!\left(\!\!\begin{array}{l}a,1-c+a\\1-b+a\end{array}\!\!\vline\,\,t\right)\sin(\pi{a})
\\
+\frac{\Gamma(a-b)\Gamma(b)t^{b}}{\pi\Gamma(c-b)}
{_{2}F_1}\!\left(\!\!\begin{array}{l}b,1-c+b\\1-a+b\end{array}\!\!\vline\,\,t\right)\sin(\pi{b}).
\end{multline*}
Substituting these expansion into the above formula for $\phi_{\pm}(x)$ and collecting terms, the expression in braces becomes:
\begin{multline*}
\frac{\Gamma(b-a)\Gamma(a+n_1-b-n_2)\Gamma(b+n_2)\Gamma(a)x^{-a-b-n_2}}{\pi^2\Gamma(c-a)\Gamma(c+m-b-n_2)}
{_{2}F_{1}}\!\left(\!\begin{array}{l}a,1-c+a\\1-b+a\end{array}\!\!\vline\,\,\frac{1}{x}\right)
\\
\times{_{2}F_1}\!\left(\!\!\begin{array}{l}b+n_2,1-c-m+b+n_2\\1-a-n_1+b+n_2\end{array}\!\!\vline\,\,\frac{1}{x}\right)
\sin(\pi(b+n_2-a))
\\
+\frac{\Gamma(a-b)\Gamma(b+n_2-a-n_1)\Gamma(a+n_1)\Gamma(b)x^{-b-a-n_1}}{\pi^2\Gamma(c-b)\Gamma(c+m-a-n_1)}
{_{2}F_{1}}\!\left(\!\begin{array}{l}b,1-c+b\\1-a+b\end{array}\!\!\vline\,\,\frac{1}{x}\right)
\\
\times{_{2}F_1}\!\left(\!\!\begin{array}{l}a+n_1,1-c-m+a+n_1\\1-b-n_2+a+n_1\end{array}\!\!\vline\,\,\frac{1}{x}\right)
\sin(\pi(a+n_1-b)).
\end{multline*}
Then, writing $t=1/x$, applying Euler's transformation and the reflection formula $\Gamma(z)\Gamma(1-z)=\pi/(\sin(\pi{z}))$, we get
\begin{multline*}
\frac{|{}_{2}F_{1}(a,b;c;1/t)|^{2}\Gamma(a)\Gamma(b)\Gamma(a+n_1)\Gamma(b+n_2)}{\Gamma(c)\Gamma(c+m)}\phi_{+}(1/t)=
\\
=\frac{\Gamma(a-b)\Gamma(b+n_2-a-n_1)\Gamma(a+n_1)\Gamma(b)t^{a+b+n_1}}{\Gamma(c-b)\Gamma(c+m-a-n_1)}
{_{2}F_{1}}\!\left(\!\begin{array}{l}b,1-c+b\\1-a+b\end{array}\!\!\vline\,\,t\right)
\\
{_{2}F_1}\!\left(\!\!\begin{array}{l}a+n_1,1-c-m+a+n_1\\1-b-n_2+a+n_1\end{array}\!\!\vline\,\,t\right)
\sin(\pi(a+n_1-b))
\\
+\frac{\Gamma(b-a)\Gamma(a+n_1-b-n_2)\Gamma(b+n_2)\Gamma(a)t^{a+b+n_2}}{\Gamma(c-a)\Gamma(c+m-b-n_2)}
{_{2}F_{1}}\!\left(\!\begin{array}{l}a,1-c+a\\1-b+a\end{array}\!\!\vline\,\,t\right)
\\
{_{2}F_1}\!\left(\!\!\begin{array}{l}b+n_2,1-c-m+b+n_2\\1-a-n_1+b+n_2\end{array}\!\!\vline\,\,t\right)
\sin(\pi(b+n_2-a))
\\
=\frac{\pi\Gamma(a-b)\Gamma(b-a+n_2-n_1)\Gamma(a+n_1)\Gamma(b)t^{a+b+n_1}(1-t)^{c-a-b+m-n_1-n_2}}
{\Gamma(c-b)\Gamma(c-a+m-n_1)\Gamma(a-b+n_1)\Gamma(1+b-a-n_1)}\times
\\
{_{2}F_{1}}\!\left(\!\begin{array}{l}b,1-c+b\\1-a+b\end{array}\!\!\vline\,\,t\right)
{_{2}F_1}\!\left(\!\!\begin{array}{l}1-b-n_2,c-b+m-n_2\\1+a-b+n_1-n_2\end{array}\!\!\vline\,\,t\right)
\\
+\frac{\pi\Gamma(b-a)\Gamma(a-b+n_1-n_2)\Gamma(b+n_2)\Gamma(a)t^{a+b+n_2}(1-t)^{c-a-b+m-n_1-n_2}}
{\Gamma(c-a)\Gamma(c+m-b-n_2)\Gamma(b-a+n_2)\Gamma(1+a-b-n_2)}\times
\\
{_{2}F_{1}}\!\left(\!\begin{array}{l}a,1-c+a\\1-b+a\end{array}\!\!\vline\,\,t\right)
{_{2}F_1}\!\left(\!\!\begin{array}{l}1-a-n_1,c-a+m-n_1\\1+b-a+n_2-n_1\end{array}\!\!\vline\,\,t\right).
\end{multline*}

Further, writing $a=\alpha$, $b=1-\gamma+\alpha$, $c=1-\beta+\alpha$ after tedious but elementary transformations using
the relations
$$
(1-z)_{-k}=\frac{(-1)^k}{(z)_{k}}~~\text{and}~~(z-r)_{k}=\frac{(z)_{k-r}}{(z)_{-r}}=(-1)^{k}\frac{(1-z)_{r}}{(1-z)_{r-k}}
$$
the above expression reduces to:
\begin{multline*}
\frac{|{}_{2}F_{1}(a,b;c;1/t)|^{2}\Gamma(a)\Gamma(b)\Gamma(a+n_1)\Gamma(b+n_2)}{\Gamma(c)\Gamma(c+m)}\phi_{+}(1/t)
\\
=-\frac{{\pi}t^{2\alpha-\gamma+1}(1-t)^{\gamma-\alpha-\beta+m-n_1-n_2}
\Gamma(\alpha+n_1)\Gamma(1+\alpha-\gamma+n_2)}{\Gamma(1-\beta+m-n_1)\Gamma(\gamma-\beta+m-n_2)}\times
\\
\biggl\{
\frac{(\gamma-\alpha)_{-n_2}(\gamma-\beta)_{m-n_2}t^{n_1}}{(\gamma-1)_{n_1-n_2+1}}
{}_{2}F_{1}\!\left(\!\begin{array}{l}1-\gamma+\alpha,1-\gamma+\beta\\2-\gamma\end{array}\!\!\vline\,\,t\right)
{}_{2}F_{1}\!\left(\!\begin{array}{l}\gamma-\alpha-n_2,\gamma-\beta+m-n_2\\\gamma+n_1-n_2\end{array}\!\!\vline\,\,t\right)
\\
+\frac{(1-\alpha)_{-n_1}(1-\beta)_{m-n_1}t^{n_2}}{(1-\gamma)_{n_2-n_1+1}}
{}_{2}F_{1}\!\left(\!\begin{array}{l}\alpha,\beta\\\gamma\end{array}\!\!\vline\,\,t\right)
{}_{2}F_{1}\!\left(\!\begin{array}{l}1-\alpha-n_1,1-\beta+m-n_1\\2-\gamma+n_2-n_1\end{array}\!\!\vline\,\,t\right)
\biggr\}
\\
=-\frac{{\pi}t^{2\alpha-\gamma+1+\underline{n}}(1-t)^{\gamma-\alpha-\beta-l}
\Gamma(\alpha+n_1)\Gamma(1+\alpha-\gamma+n_2)}{\Gamma(1-\beta+m-n_1)\Gamma(\gamma-\beta+m-n_2)}P_{r}(t),
\end{multline*}
where the ultimate equality is an application of Theorem~\ref{th:2F1identity} with the notation introduced in \eqref{eq:nmrelated}.

Now substituting back $\alpha=a$, $\beta=1-c+a$, $\gamma=1-b+a$ we get
$$
|{}_{2}F_{1}(a,b;c;1/t)|^{2}\phi_{+}(1/t)
=-\frac{{\pi}\Gamma(c)\Gamma(c+m)t^{a+b}(1-t)^{c-a-b}}{\Gamma(a)\Gamma(b)\Gamma(c-a+m-n_1)\Gamma(c-b+m-n_2)}
\frac{t^{\underline{n}}P_{r}(t)}{(1-t)^{l}}.
$$
It remains to plug back $x=1/t$ to arrive at (\ref{eq:2F1ratioboundary}).
\end{proof}
    
 Lemmas~\ref{lm:asymp-infnolog} and \ref{lm:asymp-inflog} and the subsequent remarks show that the asymptotic expansion of $R_{n_1,n_2,m}(z)$ at infinity is a combination of terms of the form $Az^{\mu}[\log(z)]^{k}$, where $A$ and $\mu$ are real numbers while $k$ is an integer. Condition \eqref{eq:asympinf1} in the  theorem below requires that each exponent $\mu$ in such terms satisfying $\mu\ge{N}$, $N\in\N_0$, is, in fact,  integer and contains no logarithm (i.e. $k=0$).  The following theorem is the main result of this section.

\begin{theorem}\label{th:2F1ratio-repr}
Suppose conditions of Theorem~\ref{th:2F1zeros} are satisfied, so that ${}_2F_{1}(a,b;c;z)\ne0$ for
$z\in\C\setminus[1,\infty)$. Assume further that $\nu>-1$ in \eqref{eq:asymp1} and for some $N\in\N_0$ the asymptotic expansion of $R_{n_1,n_2,m}(z)$ at infinity has the form
\begin{equation}\label{eq:asympinf1}
R_{n_1,n_2,m}(z)-Q_{a,b,c}(z)=o(z^N)~\text{as}~z\to\infty,
\end{equation}
where  $Q_{a,b,c}(z)$ is a \emph{(}possibly vanishing\emph{)} polynomial with real coefficients and the lowest degree $N$.
Then the following representation holds true
\begin{multline}\label{eq:mainrepresentionN}
R_{n_1,n_2,m}(z)=Q_{a,b,c}(z)+\sum_{k=0}^{N-1}\frac{R_{n_1,n_2,m}^{(k)}(0)}{k!}z^k
\\
+z^NB_{n_1,n_2,m}(a,b,c)\int_1^{\infty}\frac{x^{l-\underline{n}-c-N}(x-1)^{c-a-b-l}P_{r}(1/x)}{|{}_{2}F_{1}(a,b;c;x)|^{2}(x-z)}dx,
\end{multline}
where $r,l$ and $B_{n_1,n_2,m}(a,b,c)$ retain their meanings from Theorem~\ref{th:2F1ratioboundary} and $P_r$ is defined in \eqref{eq:Prexplicit}. In particular, if \eqref{eq:asympinf1} holds with $N=0$ we obtain
\begin{equation}\label{eq:mainrepresention}
R_{n_1,n_2,m}(z)=Q_{a,b,c}(z)+B_{n_1,n_2,m}(a,b,c)\int_1^{\infty}\frac{x^{l-\underline{n}-c}(x-1)^{c-a-b-l}P_{r}(1/x)}{|{}_{2}F_{1}(a,b;c;x)|^{2}(x-z)}dx.
\end{equation}
If $n_1,n_2\ge0$ and \eqref{eq:noninteger} holds, then \eqref{eq:asympinf1} is true with $N=0$ and $Q_{a,b,c}(z)=Q_{a,b,c}$ being a constant.
\end{theorem}

\begin{remark}
    Note that the choice of $N$ and $Q_{a,b,c}(z)$ in \eqref{eq:asympinf1} is not unique. In
    particular, it follows from Lemmas~\ref{lm:asymp-infnolog} and \ref{lm:asymp-inflog} that we
    can always take $Q_{a,b,c}(z)=0$ by choosing large enough~$N$.
\end{remark}

\begin{remark}
    The first two terms of the Taylor expansion of $R_{n_1,n_2,m}(z)$ are given by
    $$
    R_{n_1,n_2,m}(z)=1+\frac{(an_2+bn_1+n_1n_2)c-abm}{c(c+m)}z+O(z^2)
    $$
\end{remark}

\begin{remark}
Substitution $x=1/t$ brings formula \eqref{eq:mainrepresention} to the form (we write $B=B_{n_1,n_2,m}(a,b,c)$ for brevity):
\begin{equation}\label{eq:mainrepresention1}
R_{n_1,n_2,m}(z)=Q_{a,b,c}(z)+\sum_{k=0}^{N-1}\frac{R_{n_1,n_2,m}^{(k)}(0)}{k!}z^k+z^NB\int_0^{1}\frac{t^{a+b+\underline{n}+N-1}(1-t)^{c-a-b-l}P_{r}(t)}{|{}_{2}F_{1}(a,b;c;1/t)|^{2}(1-zt)}dt.
\end{equation}
This form turns out to be more convenient in most applications.  Moreover, taking $z=0$ or $z=1$ we get the following curious integral evaluations:
\begin{equation}\label{eq:integralz0}
\int_0^{1}\frac{t^{a+b+\underline{n}+N-1}(1-t)^{c-a-b-l}P_{r}(t)}{|{}_{2}F_{1}(a,b;c;1/t)|^{2}}dt=\frac{R_{n_1,n_2,m}^{(N)}(0)-Q_NN!}{N!B},
\end{equation}
where $Q_N$ denotes the coefficient at $z^N$ in $Q_{a,b,c}(z)$, and
\begin{equation}\label{eq:integralz1}
\int_0^{1}\frac{t^{a+b+\underline{n}+N-1}(1-t)^{c-a-b-l-1}P_{r}(t)}{|{}_{2}F_{1}(a,b;c;1/t)|^{2}}dt
=\frac{R_{n_1,n_2,m}(1)-Q_{a,b,c}(1)}{B}
-\frac{1}{B}\sum_{k=0}^{N-1}\frac{R_{n_1,n_2,m}^{(k)}(0)}{k!},
\end{equation}
where, in view of the Gauss summation formula,
$$
R_{n_1,n_2,m}(1)=\frac{(c)_m(c-a-b)_{m-n_1-n_2}}{(c-a)_{m-n_1}(c-b)_{m-n_2}}.
$$
Multiplying the integrand in \eqref{eq:integralz1} by $(1-t)$, splitting the result in two summands, and using both formulae \eqref{eq:integralz0} and \eqref{eq:integralz1}, we also obtain
\begin{equation}\label{eq:integralz01}
\int_0^{1}\frac{t^{a+b+\underline{n}+N}(1-t)^{c-a-b-l-1}P_{r}(t)}{|{}_{2}F_{1}(a,b;c;1/t)|^{2}}dt
=\frac{R_{n_1,n_2,m}(1)-Q_{a,b,c}(1)+Q_N}{B}
-\frac{1}{B}\sum_{k=0}^{N}\frac{R_{n_1,n_2,m}^{(k)}(0)}{k!}.
\end{equation}
\end{remark}

\begin{remark}
The absolute value of ${}_2F_{1}$ on the branch cut in the integrands in  \eqref{eq:mainrepresentionN} and \eqref{eq:mainrepresention} can be computed as follows ($x>1$):
\begin{multline*}
|{}_2F_{1}(a,b;c;x)|^2=\frac{\pi^2\Gamma(c)^2}{\Gamma(a)^2\Gamma(b)^2}
\biggl\{
\frac{(x-1)^{2(c-a-b)}}{[\Gamma(c-a-b)]^2}
\biggl[{}_{2}F_{1}\!\left(\!\begin{array}{l}c-a,  c-b\\c-a-b\end{array}\!\!\vline\,1-x\right)\biggr]^2
\\
+\biggl[\frac{\Gamma(b-a)\Gamma(a)x^{-a}}{\Gamma(c-a)\Gamma(1/2-a)\Gamma(1/2+a)}
{}_{2}F_{1}\!\left(\!\begin{array}{l}a, 1-c+a\\1-b+a\end{array}\!\!\vline\,1/x\right)
\\
+\frac{\Gamma(a-b)\Gamma(b)x^{-b}}{\Gamma(c-b)\Gamma(1/2-b)\Gamma(1/2+b)}
{}_{2}F_{1}\!\left(\!\begin{array}{l}b,1-c+b\\1-a+b\end{array}\!\!\vline\,1/x\right)\biggr]^2
\biggr\}.
\end{multline*}
\end{remark}

\begin{proof}[Proof of Theorem~\ref{th:2F1ratio-repr}]
Define $f(z)=R_{n_1,n_2,m}(z)-Q_{a,b,c}(z)$. As lowest degree term in $Q_{a,b,c}(z)$ is $z^N$, in view of Theorem~\ref{th:2F1zeros} we see that the function
$$
\hat{f}_{N}(z)=R_{n_1,n_2,m}(z)-Q_{a,b,c}(z)-\sum_{k=0}^{N-1}\frac{f^{(k)}(0)}{k!}z^k=R_{n_1,n_2,m}(z)-Q_{a,b,c}(z)-\sum_{k=0}^{N-1}\frac{R_{n_1,n_2,m}^{(k)}(0)}{k!}z^k
$$ 
is holomorphic in $z\in\C\setminus[1,\infty)$ and has no singularities on the banks of the branch cut other than $z=1$ and $z=\infty$. We aim to apply Lemma~\ref{lemma:Schwarz_asympt} to the function $\hat{f}_{N}(z)$.
Condition $\nu>-1$ implies that the first term in \eqref{eq:Schwarz_asympt} vanishes for $f=\hat{f}_{N}(z)$, while \eqref{eq:asympinf1} leads to the second equality in \eqref{eq:Schwarz_asympt} with $n=N$.  Denote $u(x)=\Im(\hat{f}_N(x+i0))$.  As $Q_{a,b,c}(z)$ has real coefficients we conclude that $u(x)=\Im[R_{n_1,n_2,m}(x+i0)]$.  Then  Lemmas~\ref{lm:asymp-infnolog} and \ref{lm:asymp-inflog} imply that the asymptotics must have one of the forms
$$
z^{-N}\hat{f}_{N}(z)=\frac{C}{\log(z)}\left(1+O\left([\log(z)]^{-1}\right)\right)~\text{as}~z\to\infty
$$
or for some $\tau>0$
$$
z^{-N}\hat{f}_{N}(z)=\frac{C}{z^{\tau}}\left(1+o(1)\right)~\text{as}~z\to\infty.
$$
Then, in view of 
$$
\left|\Im\frac{1}{\log(x+i0)}\right|\le\frac{\pi}{\log^2|x|+\pi^2},
$$
we obtain
$$
\frac{u(x)}{x^{N+1}}=O\left(\frac{1}{x\log^2(x)}\right)~~~\text{or}~~~
\frac{u(x)}{x^{N+1}}=O\left(\frac{1}{x^{1+\tau}}\right)~\text{as}~x\to\infty.
$$
Together with the condition $\nu>-1$ this leads to absolute integrability of
$x^{-N-1}u(x)$ on $(0,+\infty)$.  Hence, we are in the position to apply Lemma~\ref{lemma:Schwarz_asympt} leading to \eqref{eq:mainrepresentionN} in view of Theorem~\ref{th:2F1ratioboundary}.  The ultimate claim follows directly from Lemmas~\ref{lm:asymp-infnolog} and \ref{lm:asymp-inflog}.
\end{proof}

\section{Generalized Nevanlinna classes}\label{sec:gener-nevanl-class}

Integral representation~\eqref{eq:mainrepresention} shows that the ratio $R_{n_1,n_2,m}(z)$
belongs, possibly up to an additive polynomial term, to the Stieltjes class
$\mathcal{S}=\NN_{0}^{0}$ under the condition of positivity of the measure. The natural question
is then to determine classes that extend~$\mathcal{S}$ and contain~$R_{n_1,n_2,m}(z)$ if some of
the conditions required for~\eqref{eq:mainrepresention} are violated. To answer this question we
adopt the machinery of the generalized Nevanlinna classes. Functions of these classes may be
expressed in a very convenient multiplicative form due to Dijksma, Langer, Luger and Shondin,
see Theorem~\ref{th:DLLS} below. Another option is to use the representation by Kre\u{\i}n and
Langer~\cite[Satz~3.1]{KL} based on the standard `additive' regularization approach.

For the general ratio $\pm R_{n_1,n_2,m}(z)$ we give a simple and explicit criterion for being a
generalized Nevanlinna function: our Theorem~\ref{th:RnmNkappa} provides conditions under
which~$\pm R_{n_1,n_2,m}(z)$ belongs to the class~$\SU$ or (see Remark~\ref{rem:RnmNkappa})
to~$\NU$. Theorem~\ref{th:B_P_Rnm_Nkappa} shows how to modify $R_{n_1,n_2,m}(z)$ in order to
obtain expressions that lie in~$\SU$ for the case of arbitrary integer shifts~$n_1,n_2,m$ and
parameters~$a,b,c\in\R$ only obeying the natural constraint~$-c\in\N_0$.
Theorem~\ref{th:R011Nkappa} not only proves
that~$\varepsilon R_{0,1,1}\in\NN_{\kappa}^{\lambda}$ for some indices~$\kappa,\lambda$
and~$\varepsilon=\pm1$, which is a particular case for Theorems~\ref{th:RnmNkappa}
and~\ref{th:B_P_Rnm_Nkappa} -- it also gives a method for calculating these indices in terms of the parameters~$a,b,c$. For some other shifts there is a similar way
to determine~$\kappa$, $\lambda$ and~$\varepsilon$, see Examples~2--4 and~8--9 in
Section~\ref{sec:examples}.

Recall that each rational function~$f$ has a unique (up to multiplication by a common constant) representation as a ratio of two polynomials with no common zeros. The maximal degree
of these polynomials is called the degree of~$f$ and denoted by~$\deg f$. Together with a
possible zero or pole at infinity, $f$ has exactly~$\deg f$ zeros and~$\deg f$ poles
(counting with their multiplicities and, resp., orders).

\begin{theorem}\label{th:RnmNkappa}
    Suppose that~$a,b,c\in\R$ and~$-c\notin\N_0$. The function $R_{n_1,n_2,m}\in\SU$ if and only
    if
    \begin{equation}\label{eq:boundIneq}
        B_{n_1,n_2,m}(a,b,c)P_r(t)\ge0
        , \quad \text{for all} \quad
        t\in(0,1),
    \end{equation}
    where $P_r(t)$ is the polynomial from \eqref{eq:Prexplicit} and $B_{n_1,n_2,m}$ is given in
    \eqref{eq:B-defined}. In turn, if the reverse inequality holds above for all $t\in(0,1)$,
    then~\eqref{eq:boundIneq} is equivalent to $-R_{n_1,n_2,m}\in\SU$.

    In particular, $R_{n_1,n_2,m}(z)$ is a rational function if and only
    if~$B_{n_1,n_2,m}P_r(t)=0$ for all~$t$. If so,
    $R_{n_1,n_2,m}\in{\NN_\kappa^\lambda}$ and
    $-R_{n_1,n_2,m}\in \NN_{K-\kappa}^{\Lambda-\lambda}$ for some~$\kappa\le K$
    and~$\lambda\le\Lambda$, where~$K$ and~$\Lambda$ are the degrees of~$R_{n_1,n_2,m}$
    and~$z R_{n_1,n_2,m}$, respectively.
\end{theorem}

\begin{remark}\label{rem:RnmNkappa}
    It is seen from the formulae~\eqref{eq:2F1ratioboundary_with_B} and~\eqref{eq:f_Nkappa_pos}
    that the first part of Theorem~\ref{th:RnmNkappa} has the following equivalent form: the
    condition
    $$
    \varepsilon\lim_{y\to+0}\Im R_{n_1,n_2,m}(x+iy)\ge0
    $$
    for all but finitely many points $x\in\R$ and~$\varepsilon\in\{-1,1\}$ is necessary
    for~$\varepsilon R_{n_1,n_2,m}\in\NU$ and sufficient for
    $\varepsilon R_{n_1,n_2,m}\in\SU\subset\NU$. The next theorem shows that this condition may
    be removed through multiplication by an appropriate rational factor fixing the sign
    of~$\Im R_{n_1,n_2,m}(x+i0)$.
\end{remark}

\begin{theorem}\label{th:B_P_Rnm_Nkappa}
    Suppose that~$a,b,c\in\R$, ~$-c\notin\N_0$, and~$n_1,n_2,m\in\Z$. If $P_r(z)$ is the
    corresponding polynomial from \eqref{eq:Prexplicit}, ~$B_{n_1,n_2,m}$ is defined
    by~\eqref{eq:B-defined} and~$B_{n_1,n_2,m}P_r(z)\not\equiv 0$, then both functions
    \[
        \frac{R_{n_1,n_2,m}(z+\omega)}{B_{n_1,n_2,m}P_r\big(1/(z+\omega)\big)}
        \quad\text{and}\quad
        B_{n_1,n_2,m} P_r\Big(\frac 1{z+\omega}\Big) R_{n_1,n_2,m}(z+\omega)
    \]
    belong to~$\SU$ for any~$\omega\le1$.
\end{theorem}

Expressing the indices~$\kappa,\lambda$ of the particular
classes~$\NN_{\kappa}^\lambda\subset\SU$ emerging in Theorems~\ref{th:RnmNkappa}
and~\ref{th:B_P_Rnm_Nkappa} through the shifts~$n_1,n_2,m$ and the parameters and~$a,b,c$ does
not look as an easy task. It is nevertheless doable in various specific cases. In particular, if
a function has a regular C-fraction, then the coefficients of this continued fraction make it
possible to determine the corresponding indices~$\kappa$ and~$\lambda$ uniquely (see
Corollary~\ref{cr:SFracRational} and Theorem~\ref{th:N_kl_CFraction}). The following theorem
extending~\cite[Theorem~3.4]{Derevyagin} applies this idea to the Gauss ratio~$R_{0,1,1}(z)$,
whose continued fraction is given in~\eqref{eq:gen_cont_fr}--\eqref{eq:gen_cont_fr_cf}. We will
also employ this result as an intermediary step in our proofs of Theorems~\ref{th:RnmNkappa}
and~\ref{th:B_P_Rnm_Nkappa} (alternative to, e.g., direct derivation of the multiplicative
representation presented in Theorem~\ref{th:DLLS}).

\begin{theorem}\label{th:R011Nkappa}
    For any real values of parameters $a,b,c$ with~$-c\notin\N_0$, there are
    $\kappa,\lambda\in\N_0$ such that~$\varepsilon R_{0,1,1}(z)\in{\NN_\kappa^\lambda}$ for a
    certain~$\varepsilon\in\{-1,1\}$.
    
    More specifically, if~$\{-a,-b-1,a-c-1,b-c\}\cap\N_0=\varnothing$, put
    \[
        \begin{aligned}
        \theta_{j}
        &\coloneqq\sign \big[(c+2j)\,\Gamma(a+j)\,\Gamma(c-b+j)\,\Gamma(b+j+1)\,\Gamma(c-a+j+1)\big]
        \quad\text{and}
        \\
        \eta_{j}
        &\coloneqq
        \sign\big[(c+2j-1)\,\Gamma(a+j)\,\Gamma(c-b+j)\,\Gamma(b+j)\,\Gamma(c-a+j)\big]
        \end{aligned}
    \]
    for~$j\in\N_0$. Then~$R_{0,1,1}(z)$ is not rational, $\varepsilon=\theta_0$, and $\lambda$
    and $\kappa$ are, respectively, the number of negative entries in the sequences
    \((\theta_j)_{j=0}^\infty\) and \( (\eta_j)_{j=1}^\infty\).
    
    If~$\{-a,-b-1,a-c-1,b-c\}\cap\N_0\ne\varnothing$, then~$R_{0,1,1}(z)$ is a rational function
    of degree~$K=\big\lfloor\frac{s+1}{2}\big\rfloor$ satisfying
    $\varepsilon R_{0,1,1}\in \NN_{\kappa}^{\lambda}$
    and~$-\varepsilon R_{0,1,1}\in \NN_{K-\kappa}^{\Lambda-\lambda}$, where
    \[
        s_1\coloneqq\min\big(\{-a,b-c\}\cap\N_0\big),
        \quad
        s_2\coloneqq\min\big(\{-b,a-c\}\cap\N\big),
        \quad
        s\coloneqq\min\{2s_1,2s_2-1\}
    \]
    and~$\Lambda=\big\lfloor\frac{s+2}{2}\big\rfloor$ equals the degree of~$zR_{0,1,1}(z)$;
    here we assume that~$\min(\varnothing)=+\infty$.
    If~$s=2s_1>0$, put
    \[
        \varepsilon_{2j+1+\delta}
        \coloneqq
        \sign
        \frac{(a+j+\delta)_{s_1-j-\delta}(c-b+j+\delta)_{s_1-j-\delta}(b+j+1)_{s_1-j}(c-a+j+1)_{s_1-j}}
        {(c+2j+\delta)(c+2s_1)}
    \]
    for~$j=0,\dots,s_1-1$ and~$\delta\in\{0,1\}$.  Then~$\varepsilon=\varepsilon_1$ and $\lambda$
    \emph{(}resp.~$\kappa$\emph{)} is the number of negative entries in the sequence
    $\varepsilon_{1}$, $\varepsilon_{3}$, \dots, $\varepsilon_{2s_1-1}$
    \emph{(}resp.~$\varepsilon_{2}$, $\varepsilon_{4}$, \dots,~$\varepsilon_{2s_1}$\emph{)},
    and~$\lambda=\kappa=0$, ~$\varepsilon=1$ if~$s_1=0$.
    If~$s=2s_2-1$, put
    \[
        \varepsilon_{2j+\delta}
        \coloneqq
        \sign
        \frac{(a+j)_{s_2-j}(c-b+j)_{s_2-j}(b+j+\delta)_{s_2-j-\delta}(c-a+j+\delta)_{s_2-j-\delta}}
        {(c+2j-1+\delta)(c+2s_2-1)}
    \]
    and~$\varepsilon=\varepsilon_{1}$, where~$j=0,\dots,s_2-1$, ~$\delta\in\{0,1\}$ and~$j+\delta\ne0$.
    Then $\lambda$ \emph{(}resp.~$\kappa$\emph{)}
    is the number of negative entries in the sequence $\varepsilon_1$,
    $\varepsilon_3$, \dots, $\varepsilon_{2s_2-1}$ \emph{(}resp.
    $\varepsilon_{2}$, $\varepsilon_{4}$, \dots, $\varepsilon_{2s_2-2}$\emph{)} if~$s_2>1$,
    and~$\lambda=1/2-\varepsilon/2$, ~$\kappa=0$ if~$s_2=1$.
\end{theorem}

The proofs of Theorems~\ref{th:RnmNkappa} and \ref{th:B_P_Rnm_Nkappa} will hinge on
Theorem~\ref{th:R011Nkappa}. The next subsection presents some background material and closes
with the proof of Theorem~\ref{th:R011Nkappa}. Several results in this subsection represent
rather general facts connecting generalized Nevanlinna classes with continued fractions. The
main result in this direction is formulated in Theorem~\ref{th:N_kl_CFraction} which gives
explicit formulae for the Nevanlinna indices in terms of the coefficients of continued
fractions. Our starting point is a continued fraction -- it simplifies our task so we do not
need to deal with degenerate cases of the corresponding moment problem, cf.~\cite{DeKo}.

\subsection{Nevanlinna indices of continued fractions and proof of Theorem~\ref{th:R011Nkappa}}\label{sec:line-fract-transf}

Let us recall some basic properties of Hermitian matrices, i.e. matrices satisfying $H=H^*$,
where $H^*$ denotes conjugate-transpose of $H$. These matrices have real spectrum, they posses a
complete system of eigenvectors that may be chosen to be orthogonal. The inertia of a Hermitian
matrix~$H$ is defined as the triplet comprising the number of positive, vanishing, and negative
eigenvalues of $H$, respectively. Sylvester's law of inertia asserts that the inertia of a
Hermitian matrix~$H$ equals the inertia of~$PHP^*$, where~$P$ is an arbitrary non-degenerate
matrix. Moreover, for each Hermitian matrix~$H$ there is a unitary
transformation~$U=(U^*)^{-1}$, such that~$UHU^*$ is diagonal (the so-called reduction to
principal axes). Jacobi's algorithm for
reduction of Hermitian matrices (which relies on the LDU-decomposition, see~\cite[p.~38 and p.~300]{Ga}) shows that the
signs of their eigenvalues equal the signs of the ratios of consecutive leading principal minors
(if these minors do not vanish). As a result, we therefore have:

\begin{lemma}\label{lm:Hsum}
    Given two $n\times n$ Hermitian matrices~$H_1$ and~$H_2$, let the number of their negative
    eigenvalues be~$\kappa_1$ and~$\kappa_2$, respectively. Then~$H_1+H_2$ has at
    most~$\kappa_1+\kappa_2$ negative eigenvalues. If~$H_1$ has~$\lambda_1$ positive
    eigenvalues, then~$H_1+H_2$ cannot have less than~$\kappa_2-\lambda_1$ negative eigenvalues.
\end{lemma}
\begin{proof}
    Let~$U$ diagonalize~$H_1$:
    \[
        UH_1U^*
        =
        \diag\{d_1,\dots,d_n\} \coloneqq\sum_{j=1}^nd_j \bm{e}_j\cdot\bm{e}_j^{\textup{T}},
    \]
    where~$\bm{e}_1,\dots,\bm{e}_n$ is the standard column basis in~$\R^n$,
    and~$\left(\bm{e}_j\right)^{\textup{T}}$ is the transpose of~$\bm{e}_j$. It is enough to
    check the case~$UH_1U^*=d_j \bm{e}_j\cdot\left(\bm{e}_j\right)^{\textup{T}}$, then the
    general case will follow by consecutive application of that particular result for every~$j$.
    Moreover, we may also assume~$j=n$, which can always be achieved through permuting the rows
    of~$U$.

    For each~$\varepsilon>0$ small enough, no leading principal minor of the
    matrix~$G(\varepsilon)\coloneqq UH_2U^*+\varepsilon I$ vanishes (as nontrivial polynomial
    in~$\varepsilon$), and this matrix has precisely~$\kappa_2$ negative eigenvalues (by
    continuity: the spectrum of~$G(\varepsilon)$ is just the spectrum of~$G(0)$ shifted to the
    right by~$\varepsilon$).

    Denote the~$j$th-order leading principal minor of~$G(\varepsilon)$ by~$\Delta_j$,
    ~$\Delta_0\coloneqq 1$, and observe that the
    matrix~$\widetilde G(\varepsilon)\coloneqq d_n
    \bm{e}_n\cdot\left(\bm{e}_n\right)^{\textup{T}}+G(\varepsilon)$ has the same leading
    principal minors as~$G(\varepsilon)$ with the only exception for
    \[
        \det\widetilde G(\varepsilon)=\Delta_n+d_n\Delta_{n-1}.
    \]
    According to Jacobi's algorithm, the
    sequence~$\Delta_{1},\frac{\Delta_{2}}{\Delta_{1}},\dots,\frac{\Delta_{n}}{\Delta_{n-1}}$
    has exactly~$\kappa_2$ negative entries. The analogous sequence
    for~$\widetilde G(\varepsilon)$ comprises the same entries, except for the last one equal
    to~$\frac{\Delta_{n}}{\Delta_{n-1}}+d_n$. Thus,~$\widetilde G(\varepsilon)$ has
    either~$\kappa_2$ or~$\kappa_2-\sign d_n$ negative eigenvalues for all small
    enough~$\varepsilon>0$. The same is true for number~$\kappa$ of negative eigenvalues of~$\widetilde G(0)$, since the spectrum of~$\widetilde G(\varepsilon)$ is the right-shifted by~$\varepsilon$ spectrum of~$\widetilde G(0)$. In other words, due to
    \[
    \sign d_n=
    \begin{cases}
    \lambda_1,&\text{if } d_n\ge 0,\\
    -\kappa_1,&\text{if } d_n\le 0,
    \end{cases}
    \]
    the number~$\kappa$ satisfies~$\kappa_2-\lambda_1\le\kappa\le\kappa_2+\kappa_1$
    (the positive eigenvalues of~$\widetilde G(\varepsilon)$ remain
    nonnegative as~$\varepsilon\to+0$). Since the inertia of~$H_1+H_2$ coincides with the
    inertia of~$\widetilde G(0)$, the matrix~$H_1+H_2$ has between~$\kappa_2-\lambda_1$
    and~$\kappa_2+\kappa_1$ negative eigenvalues. The lemma is thereby proved for our rank-one
    matrix~$H_1$. As is noted above, then we also obtain the lemma in the general case.
\end{proof}

Next we obtain certain basic properties of the generalized Nevanlinna classes.

\begin{lemma}\label{lemma:Nkappa_Moebius}
    Given~$a,c\in\mathbb R$ and~$b,d>0$, let either~$\displaystyle h(z)= a + \frac{b}{c - z d}$
    or~$h(z)= a + b z$. Then a function~$f(z)$ belongs to the class~$\NN_\kappa$ precisely
    when~$h\circ f \left(z\right)\coloneqq h\left(f(z)\right)$ belongs to~$\NN_\kappa$.
\end{lemma}
\begin{proof}
    Take any finite sequence of numbers~$z_1,z_2,\dots,z_n\in\mathbb{C}_+$ such that none of
    them is a pole of~$f(z)$. If~$h(z)=a + b z$, then the Pick matrix $H_{h\circ f}$ defined in \eqref{eq:quadratic_form} satisfies
    \[
        H_{h\circ f}\left(z_1,\dots,z_n\right)
        = b H_{f}\left(z_1,\dots,z_n\right),
    \]
    so the lemma is trivial. If~$h(z)$ is not a polynomial, we have
    \[
        h(f(z_i))-\overline{h(f(z_j))}
        =
        a+\frac{b}{c - f(z_i)d} -a-\frac{b}{c - \overline{f(z_j)}d}
        =
        \frac{\sqrt{b d}\cdot\big(f(z_i)-\overline{f(z_j)}\big)\cdot\sqrt{b d}}
        {\big(c - f(z_i)d\big)\cdot\big(c - \overline{f(z_j)}d\big)}
    \]
    and, therefore,
    \begin{equation}\label{eq:H_Moebius}
        H_{h\circ f}\left(z_1,\dots,z_n\right)
        = D \cdot H_{f}\left(z_1,\dots,z_n\right) \cdot D^*
        ,
        \quad\text{where}\quad
        D\coloneqq \diag\left[\frac{\sqrt{b d}}{c - f(z_i)d}\right]_{i=1}^n
    \end{equation}
    and~$D^*$ stands for the conjugate transpose of~$D$. So, the lemma follows by Sylvester's
    law of inertia.
\end{proof}

\begin{lemma}\label{lemma:Nkappa_Moebius2}
    Let~$f(z)$ be a real function, and let
    \[
        g(z)\coloneqq-f(-z), \quad h(z)\coloneqq-f\left(\frac1z\right)
        \quad\text{and}\quad
        \widetilde h(z)\coloneqq\dfrac1{f\left(\frac1z\right)}.
    \]
    Then~$f\in \NN_\kappa \iff g\in \NN_\kappa \iff h\in \NN_\kappa \iff \widetilde h\in \NN_\kappa$.
\end{lemma}
\begin{proof}
    Indeed, $H_g(z_1,\dots,z_n)$ is precisely the transpose of~$H_f(w_1,\dots,w_n)$
    with~$w_i\coloneqq -\overline{z_i}$:
        \[
        \frac{g(z_i)-\overline{g(z_j)}}{z_i-\overline{z_j}}
        =
        \frac{\overline{f(-z_j)}-f(-z_i)}{z_i-\overline{z_j}}
        =
        \frac{\overline{f(\overline{w_j})}-f(\overline{w_i})}{-\overline{w_i}+w_j}
        =
        \frac{f(w_j)-\overline{f(w_i)}}{w_j-\overline{w_i}}
        ,
    \]
    and $H_h(z_1,\dots,z_n)$ is diagonally similar to~$H_g(\zeta_1,\dots,\zeta_n)$
    with~$\zeta_i\coloneqq -\frac 1{z_i}$:
    \[
        \frac{h(z_i)-\overline{h(z_j)}}{z_i-\overline{z_j}}
        =
        \frac1{z_i}\frac{g(-1/z_i)-g(-1/\overline{z_j})}{-1/z_i+1/\overline{z_j}}\frac1{\overline{z_j}}
        =
        \zeta_i\cdot\frac{g(\zeta_i)-\overline{g(\zeta_j)}}{\zeta_i-\overline{\zeta_j}}
        \cdot\overline{\zeta_j}
        .
    \]
    Here~$i,j$ run over~$1,\dots,n$. Therefore,
    \( f\in \NN_\kappa\iff g\in \NN_\kappa\iff h\in \NN_\kappa. \)

    Now, $\widetilde h(z)=-\frac 1{h(z)}$, and hence
    $\widetilde h\in \NN_\kappa\iff h\in \NN_\kappa$ by the case  $a=c=0$, ~$b=d=1$ of  Lemma~\ref{lemma:Nkappa_Moebius}.
\end{proof}

\begin{lemma}[{cf.~\cite[p.~17]{Pick16}}]\label{lm:Nkappa_rational}
    If~$f(z)$ is a real rational function and~$n\ge m\coloneqq\deg f$, then the rank
    of~$H_{f}(z_1,\dots,z_n)$ equals~$m$ provided that the points~$z_1,\dots,z_n$ are distinct.
\end{lemma}
\begin{proof}
    Let~$m>0$, otherwise the lemma is trivial. Take any
    \[
        \delta\in\R\setminus\{f(\xi_1),\dots,f(\xi_r),f(\infty)\},
    \]
    where~$\xi_1,\dots,\xi_r$ are the real solutions to the equation~$f'(z)=0$. Put
    \[
        g_\delta(z)\coloneqq \frac 1{\delta- f(z)}
        ,
    \]
    so that~$\deg g_\delta =m$ and, according to~\eqref{eq:H_Moebius} with $h(z)=1/(\delta-z)$, the
    matrices~$H_f(z_1,\dots,z_n)$ and~$H_{g_\delta}(z_1,\dots,z_n)$ have the same inertia.
    Moreover,~$g_\delta(z)$ only has simple poles, and~$g_\delta(\infty)$ is finite. As a
    result, there exists~$l$, such that~$0\le 2l\le m$ and for some numbers~$A_0\in\R$, ~$a_{2k-1}=\overline{a_{2k}}$ and~$A_{2k-1}=\overline{A_{2k}}\ne0$
    when~$0<k\le l$, and~$a_k\in\R$, $A_k\in\R\setminus\{0\}$ when~$k>2l$,
    \[
        g_\delta(z)
        =
        A_0 + \sum_{k=1}^{m}\frac{A_k}{z-a_k}
        =
        A_0
        + \sum_{k=1}^{l}\left(\frac{\overline{A_{2k}}}{z-\overline{a_{2k}}}
        +\frac{A_{2k}}{z-a_{2k}}\right)
        + \sum_{k=2l+1}^{m}\frac{A_k}{z-a_k}.
    \]
    Then
    \[
        \frac{g_\delta(z_i)-\overline{g_\delta(z_j)}}{z_i-\overline{z_j}}
        =
        \frac{1}{z_i-\overline{z_j}}
            \sum_{k=1}^m\left(\frac{A_{k}}{z_i-a_k}-\frac{A_k}{\overline{z_j}-a_k}\right)
        =
        -\sum_{k=1}^m\frac{A_k}{(z_i-a_k)(\overline{z_j}-a_k)}
        .
    \]

    Now, choose arbitrary~$a_{m+1},\dots,a_n\in\R$ so that the numbers~$a_1,\dots,a_n$ are
    distinct. Observe that $H_{g_\delta}(z_1,\dots,z_n)= P_n \cdot M \cdot P_n^*$, where
    \[
        M\coloneqq
        - \diag\left(
            \left[\begin{smallmatrix}0&\overline{A_{2}}\\A_{2}&0\end{smallmatrix}\right],
            \dots,
            \left[\begin{smallmatrix}0&\overline{A_{2l}}\\A_{2l}&0\end{smallmatrix}\right],
            A_{2l+1},\dots,A_m,0,\dots,0\right)
        ,
        \quad
        P_n\coloneqq\left[\frac{1}{z_j-a_k}\right]_{j,k}^{n}
    \]
    and~$P_n^*$ denotes the conjugate transpose of~$P_n$. Since~$P_n$ is the Cauchy matrix, the
    well-known formula for its determinant yields~$\det P_n \ne 0$. By Sylvester's law of
    inertia, the inertia of the Pick matrix~$H_{g_\delta}(z_1,\dots,z_n)$ is the same as that
    of~$M$. The latter, however, equals~$m$, since all the numbers $A_1,\dots,A_m$ are distinct from zero.
\end{proof}

\begin{remark}
    The conditions of Lemma~\ref{lm:Nkappa_rational} exclude the case~$n<\deg f$, in which the
    rank of~$H_{f}(z_1,\dots,z_n)$ may be less than~$n$: for instance,
    if~$f(z)=\prod_{j=1}^n(z-z_j)(z-\overline{z_j})$, then~$\deg f=2n$ and~$\operatorname{rank} H_{f}(z_1,\dots,z_n)=0$.
\end{remark}

\begin{remark}
    As is seen from the above identity~$H_{g_\delta}(z_1,\dots,z_n)= P_n \cdot M \cdot P_n^*$,
    one can relate the inertia of~$H_{g_\delta}(z_1,\dots,z_n)$ to the location of poles
    of~$g_\delta(z)$, and (for real poles) to the signs of the corresponding residues. Moreover,
    this relation may be extended to multiple poles. The works~\cite{DHdS,DLLS} present a
    far-reaching generalization of this idea, which is cited here as Theorem~\ref{th:DLLS}.
\end{remark}

\begin{remark}
    On letting~$z_1,\dots,z_n\to\infty$ and~$f(z)=\sum_{k=0}^\infty c_kz^{-k-1}$, the (scaled)
    matrix~$H_{f}(z_1,\dots,z_n)$ tends to the Hankel matrix~$D_n^{(0)}$ as
    in~\eqref{eq:Delta_def}. Lemma~\ref{lm:Nkappa_rational} is therefore related to Kronecker's
    theorem: $D_n^{(0)}\ne 0=D_{n+1}^{(0)}=D_{n+2}^{(0)}=\cdots$ if and only if~$f(z)$
    represents a rational function of degree~$n$.
\end{remark}

Rational functions may be seen as building blocks for the generalized Nevanlinna classes. From
this viewpoint, it looks natural to study these classes via continued fractions.

\begin{lemma}[{e.g.~\cite{Derevyagin03} or~\cite[Prop.~3.3]{Derevyagin}}]\label{lemma:Nkappa_Jacobi1}
    Given~$\varepsilon,a,b>0$ and~$c\in\mathbb R$, let~$f(z)$ be a function meromorphic
    in~$\C_+\coloneqq\{z\in\C:\Im z>0\}$ such that
    \[
        f(iy)=o(y)\quad\text{as}\quad y\to+\infty.
    \]
    Define
    \[
        g_\varepsilon(z)\coloneqq
        - \frac{\varepsilon b}{a z - c + \varepsilon f(z)}
        \quad
        \left({}=
        - \frac{b}{\frac{a}{\varepsilon}z - \frac{c}{\varepsilon} + f(z) }\right)
        .
    \]
    Then~$g_{-\varepsilon}(z)$ does not belong to~$\NN_0$. Moreover, the following three
    conditions are equivalent:
    \[
        f\in \NN_\kappa, \quad g_\varepsilon\in \NN_\kappa
        \quad\text{and}\quad
        g_{-\varepsilon}\in \NN_{\kappa+1}.
    \]
\end{lemma}
\begin{remark}
    The choice~$f(z)=2 z$ and~$\varepsilon=a$ yields~$f,g_{-\varepsilon}\in \NN_0$, so
    Lemma~\ref{lemma:Nkappa_Jacobi1} \textbf{fails} to be generally true without the condition
    on the growth of~$f(z)$ at infinity. (It is missed out in~\cite{Derevyagin}, but
    fortunately does not affect other results of~\cite{Derevyagin}).
\end{remark}
\begin{proof}
    Lemma~\ref{lemma:Nkappa_Moebius} implies that~$g_{\pm\varepsilon}\in \NN_\varkappa$ for
    some~$\varkappa$ if and only if~$f+h_{\pm\varepsilon}\in \NN_\varkappa$,
    where
    \[
        h_\varepsilon(z)\coloneqq \frac a\varepsilon z,
        \quad\text{and hence}\quad
        f(z)+h_{\pm\varepsilon}(z) = \pm\frac{c}{\varepsilon} - \frac{b}{g_{\pm\varepsilon}(z)}.
    \]

    Observe that as~$y\to+\infty$
    \begin{equation}
        \label{eq:Nkappa_J1}
        H_{f+h_\varepsilon}[z_1,\dots,z_n,iy]=
        \left[\begin{array}{ccc|c}
                \multicolumn{3}{c|}{\multirow{3}{*}{$H_f[z_1,\dots,z_n]
                +\left[\frac{a}{\varepsilon}\right]_{k,m=1}^{n}$}
                }&o(1)+\frac{a}{\varepsilon}\\[-2pt]
                \multicolumn{3}{c|}{}&\vdots\\
                \multicolumn{3}{c|}{}&o(1)+\frac{a}{\varepsilon}\\
                \hline
                o(1)+\frac{a}{\varepsilon}&\hdots&o(1)+\frac{a}{\varepsilon}&o(1)+\frac{a}{\varepsilon}
              \end{array}\right]
          \to
        P\cdot
        \left[\begin{array}{ccc|c}
                \multicolumn{3}{c|}{\multirow{3}{*}{$H_f[z_1,\dots,z_n]$}}&0\\[-2pt]
                \multicolumn{3}{c|}{}&\vdots\\
                \multicolumn{3}{c|}{}&0\\
                \hline
                \phantom{0}0\phantom{0}&\hdots&0&\frac{a}{\varepsilon}
              \end{array}\right]
        \cdot P^*
        ,
    \end{equation}
    where
    \[
        P\coloneqq
        \left[\begin{array}{ccc|c}
                \multicolumn{3}{c|}{\multirow{3}{*}{$I_n$}}&1\\[-2pt]
                \multicolumn{3}{c|}{}&\vdots\\
                \multicolumn{3}{c|}{}&1\\
                \hline
                0&\hdots&0&1
              \end{array}\right].
    \]
    If~$y>0$ is large enough, the $1\times1$ matrix~$H_{f+h_{-\varepsilon}}[iy]$ turns out to be negative definite,
    so~$g_{-\varepsilon}\notin\NN_0$.
    
    For arbitrary~$n\in\N$ and~$z_1,\dots,z_{n}\in\C_+$, let~$\lambda$ be the number of
    negative eigenvalues of $H_{f}[z_1,\dots,z_n]$. The
    matrix~$H_{h_\varepsilon}[z_1,\dots,z_{n}]$ is positive semidefinite of rank one, thus
    \[
        H_{f+h_\varepsilon}[z_1,\dots,z_{n}]
        =
        H_{f}[z_1,\dots,z_{n}] +
        H_{h_\varepsilon}[z_1,\dots,z_{n}]
    \]
    has~$\lambda$ or~$\lambda-1$ negative eigenvalues by Lemma~\ref{lm:Hsum}. Therefore, the
    inclusion~$f\in\NN_{\kappa_1}$ for some~$\kappa_1$ implies~$g_\varepsilon\in\NN_{\kappa_2}$
    with~$\kappa_2\le \kappa_1$, while~$g_\varepsilon\in\NN_{\kappa_2}$ for some~$\kappa_2$
    implies~$f\in\NN_{\kappa_1}$ with~$\kappa_1\le \kappa_2+1$, in view of the equivalence
    $g_{\varepsilon}\in \NN_\varkappa \iff f+h_{\varepsilon}\in
    \NN_\varkappa$.

    Let~$f\in\NN_{\kappa_1}$. Then there exist numbers~$z_1,\dots,z_n\in\C_+$ such
    that~$H_{f}[z_1,\dots,z_n]$ has precisely~$\kappa_1$ negative eigenvalues. Then
    $H_{f+h_\varepsilon}[z_1,\dots,z_n,iy]$ has~$\kappa_1$ negative eigenvalues due
    to~\eqref{eq:Nkappa_J1} whenever~$y>0$ is large enough. As a
    result,~$g_\varepsilon\in \NN_{\kappa_2}$ with~$\kappa_2\ge\kappa_1$, and
    hence~$\kappa_2=\kappa_1$.

    Analogously, for any~$z_1,\dots,z_n\in\C_+$ and~$y>0$ the
    matrix~$H_{h_{-\varepsilon}}[z_1,\dots,z_n]$ is negative semidefinite of rank one, by
    Lemma~\ref{lm:Hsum} the number of negative eigenvalues
    of~$H_{f+h_{-\varepsilon}}[z_1,\dots,z_n]$ may only equal~$\lambda$ or~$\lambda+1$. In
    particular,~$f\in\NN_{\kappa_1}$ for some~$\kappa_1$
    implies~$g_{-\varepsilon}\in\NN_{\kappa_3}$ with~$\kappa_3\le \kappa_1+1$,
    while~$g_{-\varepsilon}\in\NN_{\kappa_3}$ for some~$\kappa_3$ implies~$f\in\NN_{\kappa_1}$
    with~$\kappa_1\le \kappa_3$.

    If~$f\in\NN_{\kappa_1}$, the formula~\eqref{eq:Nkappa_J1} for the above choice
    of~$z_1,\dots,z_n\in\C_+$, ~$\varepsilon\mapsto-\varepsilon$, and sufficiently large~$y>0$
    yields that~$H_{f+h_{-\varepsilon}}[z_1,\dots,z_n,iy]$ has~$\kappa_1+1$ negative
    eigenvalues. So,~$g_{-\varepsilon}\in \NN_{\kappa_3}$ with~$\kappa_3\ge\kappa_1+1$, and
    hence~$\kappa_3=\kappa_1+1$.
\end{proof}

\begin{lemma}\label{lemma:JFracRational}
    Given numbers~$m,\lambda\in\N_0$, ~$\varepsilon_{2m+1}\in\{-1,1\}$ and a
    function~$\psi_{2m+1}\in \NN_\lambda$ satisfying~$\psi_{2m+1}(iy)=o(y)$ as~$y\to+\infty$,
    let
    \begin{equation}\label{eq:JFracRational}
        \psi(z)\coloneqq
        -\cfrac{\alpha_0}{z-\beta_0-\cfrac{\alpha_1\alpha_2}
            {z-\beta_1-\raisebox{-.45em}{$\ddots$}{{
                        \genfrac{}{}{0pt}{}{\vphantom{a}}{\displaystyle{}-\cfrac{\alpha_{2m-1}\alpha_{2m}}
                        {z-\beta_m+\varepsilon_{2m+1}\psi_{2m+1}(z)}}}}}}
        ,
    \end{equation}
    where the coefficients~$\beta_0,\dots,\beta_m$ and~$\alpha_0,\dots,\alpha_{2m}\ne0$ are real. Put
    \begin{equation}\label{eq:eps_alpha_rel}
        \varepsilon_{j}\coloneqq \varepsilon_{2m+1}\sign \prod_{k=j}^{2m} \alpha_k
        \quad\text{for}\quad j=0,1,\dots,2m,
    \end{equation}
    and let~$\kappa$ be the number of negative entries in the
    sequence~$\varepsilon_1,\varepsilon_3,\dots,\varepsilon_{2m+1}$, that
    is
    $$
       \kappa\coloneqq \frac {m+1}2 - \sum_{j=0}^{m}\frac{\varepsilon_{2j+1}}2.
    $$
    Then~$\varepsilon_0\psi\in \NN_{\kappa+\lambda}$.
    
    The particular  choice~$\psi_{2m+1}(z)\equiv 0$ and~$\varepsilon_{2m+1}=1$
    yields~$\varepsilon_0\psi\in \NN_{\kappa}$
    and~$-\varepsilon_0\psi\in \NN_{m+1-\kappa}=\NN_{\deg\psi-\kappa}$
    for~$\kappa=m/2 - \sum_{j=1}^{m}(\varepsilon_{2j-1}/2)$.
\end{lemma}
\begin{proof}
    It is enough only to consider the case~$m\ge 1$: for~$m=0$ the proof immediately follows
    from Lemma~\ref{lemma:Nkappa_Jacobi1}. In view of
    $\varepsilon_{0}=\varepsilon_{1}\sign\alpha_{0}$ and
    $\varepsilon_{2j-1}=\varepsilon_{2j+1}\sign(\alpha_{2j-1}\alpha_{2j})$ we see that
    \[
        \varepsilon_{0}\varepsilon_{1}\alpha_{0}=\varepsilon_{1}^2\left|\alpha_{0}\right|>0
        \quad\text{and}\quad
        \varepsilon_{2j-1}\varepsilon_{2j+1} \alpha_{2j-1}\alpha_{2j}
        =\varepsilon_{2j+1}^2\left|\alpha_{2j-1}\alpha_{2j}\right|>0
        ,
    \]
    for all~$j=1,2\dots,m$. In other
    words,~$\varepsilon_{0}\alpha_{0}=\varepsilon_{1}\left|\alpha_{0}\right|$
    and~$\alpha_{2j-1}\alpha_{2j}
    =\varepsilon_{2j-1}\varepsilon_{2j+1}\left|\alpha_{2j-1}\alpha_{2j}\right|$, so
    \[
        \varepsilon_{0}\psi(z)=
        -\cfrac{\varepsilon_{1}|\alpha_0|}
        {z-\beta_0-\cfrac{\varepsilon_{1}\varepsilon_{3}|\alpha_1\alpha_2|}
            {z-\beta_1-\raisebox{-.45em}{$\ddots$}
                {{\genfrac{}{}{0pt}{}{\vphantom{a}}{\displaystyle{}-\cfrac{\varepsilon_{2m-1}\varepsilon_{2m+1}
                            |\alpha_{2m-1}\alpha_{2m}|}
                            {z-\beta_m+\varepsilon_{2m+1}\psi_{2m+1}(z)}}}}}}.
    \]
    Define
    \begin{equation}\label{eq:psi_recur}
        \psi_{2j-1}(z)
        \coloneqq
        -\frac{\varepsilon_{2j+1}|\alpha_{2j-1}\alpha_{2j}|}
        {z-\beta_j+\varepsilon_{2j+1}\psi_{2j+1}(z)}
        ,\quad
        j=m,m-1,\dots,1.
    \end{equation}
    In particular, for each~$j$ the estimate~$\psi_{2j+1}=o(z)$ as~$z\to\infty$
    implies~$\psi_{2j-1}=O(z^{-1})$. By Lemma~\ref{lemma:Nkappa_Jacobi1},
    $\psi_{2m+1}\in \NN_0\iff\psi_{2m-1}\in \NN_{\kappa_{m}}$,
    where~$\kappa_{m}\coloneqq (1-\varepsilon_{2m+1})/2$. Analogously,
    Lemma~\ref{lemma:Nkappa_Jacobi1} yields
    that~$\psi_{2j+1}\in \NN_{\kappa_{j+1}} \iff \psi_{2j-1}\in \NN_{\kappa_{j}}$
    with~\( \kappa_{j}\coloneqq\kappa_{j+1}+(1-\varepsilon_{2j+1})/2\) as~$j$ runs
    over~$m-1,m-2,\dots,1$. As
    \[
        \varepsilon_{0}\psi(z)=
        -\frac{\varepsilon_{1}|\alpha_0|}{z-\beta_0+\varepsilon_{1}\psi_1(z)}
    \]
    Lemma~\ref{lemma:Nkappa_Jacobi1} also implies~$\varepsilon_0\psi\in \NN_\kappa$, where the
    index~$\kappa$ is determined by the signs of~$\varepsilon_{1},\dots,\varepsilon_{2m-1}$:
    \[
        \kappa
        =\kappa_{1}+\frac{1-\varepsilon_{1}}2
        =\kappa_{2}+\frac{2-\varepsilon_{1}-\varepsilon_{3}}2=\cdots
        =\frac {m+1}2 - \sum_{j=0}^{m}\frac{\varepsilon_{2j+1}}2
        .
    \]

    If~$\psi_{2m+1}(z)\equiv 0$, the corresponding assertion of the lemma follows by noting that
    $\varepsilon_{2m+1}\psi_{2m+1}\in \NN_{0}$ for both $\varepsilon_{2m+1}=1$
    and~$\varepsilon_{2m+1}=-1$. With~$\varepsilon_{2m+1}=1$, we
    obtain~$\varepsilon_0\psi\in N_\kappa$, where
    \[
        \kappa=\frac {m}2 -\sum_{j=0}^{m-1}\frac{\varepsilon_{2j+1}}2
        \quad\text{due to}\quad
        \frac{1-\varepsilon_{2j+m}}2=0.
    \]
    On the other hand, there are~$m+1-\kappa$ positive numbers
    among~$\varepsilon_1,\varepsilon_3,\dots,\varepsilon_{2m+1}$. So, the
    inclusion~$-\varepsilon_0\psi\in N_{m+1-\kappa}$ follows by an application of the previous
    part of the proof, where~$\varepsilon_{2m+1}$ and,
    hence,~$\varepsilon_{0},\varepsilon_{1},\dots,\varepsilon_{2m-1}$ are chosen to have the
    opposite signs. (Alternatively, one can use Lemma~\ref{lm:Nkappa_rational}.)
\end{proof}

Consider a terminating or non-terminating continued fraction
\begin{equation}\label{eq:phi_n_cf}
    \varphi_{n}(z)\coloneqq
    \cfrac{\alpha_{n}}{1-\cfrac{\alpha_{n+1}z}{1-\cfrac{\alpha_{n+2}z}{1-\raisebox{-.4em}{$\ddots$}}}}
    \quad\left(=\frac{\alpha_n}{1-z\varphi_{n+1}\left(z\right)}\right)
    ,\quad
    n=0,1,\dots.
\end{equation}
Fractions of this form are called C-fractions. Here $\varphi_{n}(z)$ is terminating
when~$\alpha_j=0$ for some~$j\ge n$, in that case we assume%
\footnote{This is possible since we start with a continued fraction~$\varphi_n(z)$ with finite
    coefficients (in our application to~$R_{0,1,1}(z)$, the finiteness follows from the
    assumption~$-c\notin\N_0$). If~$\varphi_n(z)$ would be introduced as a (formal) power
    series, and some coefficient~$\alpha_j$ calculated via~\eqref{eq:alpha_via_Delta} would
    vanish, then a neighbouring coefficient might need to be assumed infinite. In such a case,
    the series cannot be expressed as a regular C-fraction (although one can use the general
    C-fraction~\cite[\S~21]{Perron}, or follow~\cite{DeKo,KL3} an use the P-fractions).}%
~$\varphi_j(z)\equiv0$). For every~$n$,
\[
    \varphi_{n}(z)=
    \cfrac{\alpha_{n}}{1-\cfrac{\alpha_{n+1}z}{1-z\varphi_{n+2}\left(z\right)}}
    \sim
    \cfrac{\alpha_{n}(1-z\varphi_{n+2}\left(z\right))}%-\alpha_{n+1}z+\alpha_{n+1}z
    {1-z\varphi_{n+2}\left(z\right)-\alpha_{n+1}z}
    \sim
    \alpha_{n}+\cfrac{\alpha_{n}\alpha_{n+1}z}{1-\alpha_{n+1}z-z\varphi_{n+2}\left(z\right)},
\]
where~`$\sim$' denotes the `correspondence', i.e.\ that the corresponding (formal) power series
coincide. This is the so-called contraction of continued fractions, cf.~\cite[p.~135]{Perron};
it connects regular C-fractions and J-fractions as in~\eqref{eq:JFracRational}.
For our goals, it is enough to observe that
\begin{equation}\label{eq:psi_n_recur}
    \psi_{n}(z)\coloneqq\alpha_n-\varphi_{n}\left(\frac1z\right)
    \sim
    -\cfrac{\alpha_{n}\alpha_{n+1}\frac 1z}
    {1-\alpha_{n+1}\frac 1z-\frac 1z\varphi_{n+2}\left(\frac 1z\right)}
    \sim
    -\cfrac{\alpha_{n}\alpha_{n+1}}
    {z-\alpha_{n+1}-\alpha_{n+2}+\psi_{n+2}\left(z\right)}.
\end{equation}
This formula implies that defining 
\begin{equation}\label{eq:psi_psi1}
\psi(z)\coloneqq-\frac{\alpha_0}{z-\alpha_1+\psi_1(z)}
\end{equation}
we obtain \eqref{eq:JFracRational} with $\beta_0=\alpha_{1}$    and~$\beta_j=\alpha_{2j}+\alpha_{2j+1}$ for~$j=1,\dots,m$. On the other hand, in view of \eqref{eq:psi_n_recur} and \eqref{eq:phi_n_cf}, we have 
\begin{equation}\label{eq:psi_phi}
\psi(z)\sim-\frac{\alpha_0}{z-\phi_1(1/z)}=-\frac{\alpha_0/z}{1-(1/z)\phi_1(1/z)}=-\frac{1}{z}\phi_0\left(\frac{1}{z}\right).
\end{equation}

\begin{corollary}\label{cr:SFracRational}
    Consider the terminating continued fraction
    \[
        \varphi(z)=
        \cfrac{\alpha_0}{1-\cfrac{\alpha_1z}{1-\raisebox{-.45em}{$\ddots$}
                {{\genfrac{}{}{0pt}{}{\vphantom{a}}
                        {\displaystyle{}-\dfrac{\alpha_{k-1}z}{1-\alpha_{k}z}}}}}},
    \]
    whose coefficients~$\alpha_0,\dots,\alpha_k$ are real and non-zero.
    Denote~$\varepsilon_{j}\coloneqq \sign \prod_{l=j}^{k} \alpha_l$ and
    put~$\alpha_{k+1}\coloneqq 0$. Then
    \[
        \varepsilon_0\varphi\in \NN_\kappa^\lambda
        \quad\text{and}\quad
        -\varepsilon_0\varphi\in \NN_{K-\kappa}^{\Lambda-\lambda}
        ,
    % \quad
    % m\coloneqq\left\lfloor\frac k2\right\rfloor\quad
    % n\coloneqq\left\lfloor\frac{k-1}2\right\rfloor
    \]
    where~$\Lambda\coloneqq\left\lfloor\frac {k+2}2\right\rfloor$
    and~$K\coloneqq\left\lfloor\frac{k+1}2\right\rfloor$, while~$\lambda$ and~$\kappa$ are the
    numbers of negative entries in the
    sequences~\( \varepsilon_1,
    \varepsilon_3,\dots,\varepsilon_{2\Lambda-3},\alpha_{2\Lambda-1}\)
    and~\( \varepsilon_2, \varepsilon_4,\dots,\varepsilon_{2K-2},\alpha_{2K} \), respectively.
    \textup{(}Here,~$K$ and~$\Lambda$ are resp. equal to the degrees of the rational
    functions~$\varphi(z)$ and~$z\varphi(z)$.\textup{)}
\end{corollary}
\begin{proof}
    Suppose that~$\varphi(z)\not\equiv0$, otherwise the corollary is trivial.
    Then~$\varphi_{2m+1}(z)=\alpha_{2m+1}$ for~$m\coloneqq\Lambda-1$ in view
    of~\eqref{eq:phi_n_cf}; hence,~\eqref{eq:psi_n_recur} implies~$\psi_{2m+1}(z)\equiv 0$ and,
    on account of \eqref{eq:psi_psi1}--\eqref{eq:psi_phi},
    \[
        \psi(z)\coloneqq
        -\frac 1z\varphi\left(\frac 1z\right)
        =
        -\cfrac{\alpha_0\frac 1z}{1-\frac 1z\varphi_1\left(\frac 1z\right)}
        =
        -\cfrac{\alpha_0}{z-\alpha_1+\psi_1\left(z\right)}
        ,
    \]
    which is~\eqref{eq:JFracRational} with~$\beta_0=\alpha_{1}$
    and~$\beta_j=\alpha_{2j}+\alpha_{2j+1}$ for~$j=1,\dots,m$. By
    Lemma~\ref{lemma:Nkappa_Moebius2}, if one of the
    functions~$\varepsilon_0z\varphi(z)=\varepsilon_0z\varphi_0(z)$ or~$\varepsilon_0\psi(z)$
    lies in~$\NN_\lambda$, then the other is also in~$\NN_\lambda$.
    Analogously,~$-\varepsilon_0z\varphi(z)$ lies in~$\NN_{\Lambda-\lambda}$ if and only
    if~$-\varepsilon_0\psi\in \NN_{\Lambda-\lambda}=\NN_{m+1-\lambda}$. For~$k=2m$,
    Lemma~\ref{lemma:JFracRational} with~$\varepsilon_{2m+1}=1$ and~$\psi_{2m+1}(z)\equiv 0$
    applied to~$\psi(z)$ immediately yields~$\varepsilon_0\psi\in \NN_\lambda$
    and~$-\varepsilon_0\psi\in \NN_{m+1-\lambda}$, where~$\lambda$ is as required.

    For~$k=2m+1$, we also put~$\psi_{2m+1}(z)\equiv 0$ and apply
    Lemma~\ref{lemma:JFracRational}: to get~$\varepsilon_0\psi\in \NN_\lambda$ it is enough to
    take~$\varepsilon_{2m+1}=\sign\alpha_{2m+1}$, which makes the present values
    of~$\varepsilon_j$ equal to those of Lemma~\ref{lemma:JFracRational}. In turn, the
    inclusion~$-\varepsilon_0\psi\in \NN_{m+1-\lambda}$ follows by taking the opposite signs
    of~$\varepsilon_{2m+1}$ and,
    hence, of~$\varepsilon_{0},\varepsilon_{1},\dots,\varepsilon_{2m-1}$ as well.
            
    Observe that~$\psi_0(z)$ defined by~\eqref{eq:psi_n_recur} may be expressed as
    \begin{equation*}
        \psi_0(z)
        =
        -\cfrac{\alpha_0\alpha_1}{z-(\alpha_1+\alpha_2)-\cfrac{\alpha_2\alpha_3}{z-(\alpha_3+\alpha_4)
                -\raisebox{-.45em}{$\ddots$}{{\genfrac{}{}{0pt}{}{\vphantom{a}}
                        {\displaystyle{}-\frac{\alpha_{2K-2}\alpha_{2K-1}}
                            {z-(\alpha_{2K-1}+\alpha_{2K})}}}}}}
        .
    \end{equation*}
    By Lemmas~\ref{lemma:Nkappa_Moebius} and~\ref{lemma:Nkappa_Moebius2}, the required
    inclusions~$\varepsilon_0\varphi\in \NN_\kappa$
    and~$-\varepsilon_0\varphi\in \NN_{K-\kappa}$ are equivalent
    to~$\varepsilon_0\psi_0\in \NN_\kappa$ and~$-\varepsilon_0\psi_0\in \NN_{K-\kappa}$.
    If~$k=2K-1$, the latter pair of inclusions readily follows from
    Lemma~\ref{lemma:JFracRational} (in which we take~$\varepsilon_{2m+1}=1$ and
    $\psi_{2m+1}(z)\equiv0$) applied to~$\psi_0(z)$.

    If~$k=2K$, letting~$\varepsilon_{2K}=\sign\alpha_{2K}$ and~$\psi_{2K}(z)\equiv 0$
    yields~$\varepsilon_0\psi_0\in \NN_\kappa$. The
    inclusion~$-\varepsilon_0\psi_0\in \NN_{K-\kappa}$ follows in a similar way by taking the
    opposite signs of~$\varepsilon_{2K}$ and,
    hence,~$\varepsilon_{0},\varepsilon_{1},\dots,\varepsilon_{2K-2}$.
\end{proof}

\begin{theorem}\label{th:N_kl_CFraction}
    Let the continued fraction~$\varphi(z)\coloneqq \varphi_0(z)$ defined in \eqref{eq:phi_n_cf}
    be non-terminating and have real coefficients satisfying $\sup_{n}|\alpha_n|<\infty$. If
    there exists~$m\in\N$ such that~$\alpha_{n}>0$ for all~$n>2m$,
    then~$\varepsilon_0\varphi\in \NN_\kappa^\lambda$
    and~$-\varepsilon_0\varphi\notin\NU$, where~$\varepsilon_0$ is found from
    \begin{equation}\label{eq:eps_alpha_rel_inf}
        \varepsilon_{j}\coloneqq \sign\prod_{k=j}^{2m} \alpha_k=\prod_{k=j}^\infty \sign \alpha_k,
    \end{equation}
    and $\lambda$ and $\kappa$ are the number of negative entries in the sequences
    \(\varepsilon_1, \varepsilon_3,\dots,\varepsilon_{2m-1}\) and
    \( \varepsilon_2, \varepsilon_4,\dots,\varepsilon_{2m} \), respectively. Conversely,%
    \footnote{We still may have~$\varepsilon\varphi\in\NU$ when there are infinitely
        large~$n$ such that~$\alpha_{n}<0$: for instance, if~$\alpha_j=(-1)^j$ for all~$j$.} %
    $\varepsilon\varphi\notin\SU$ for both~$\varepsilon\in\{-1,1\}$ if and only if there are
    infinitely large~$n$ such that~$\alpha_{n}<0$.
\end{theorem}

\begin{proof}%[Proof of Lemma~\ref{lemma:JfracGen}]
    Theorem~\ref{th:CF_conv}~\eqref{item:Worpitzky} says that, for each~$n=0,1,\dots$ and any
    positive~$\gamma<\gamma_0\coloneqq 1/(4\sup_n|\alpha_n|)$, the continued
    fraction~$\varphi_n(z)$ defined in~\eqref{eq:phi_n_cf} uniformly converges in the
    disc~$D_\gamma\coloneqq\big\{z\in\C:|z|\le\gamma\big\}$ to a function analytic
    in~$D_\gamma$.

    If~$m$ is a number satisfying the assumptions of the theorem, then the functions determined
    by the continued fractions~$\varphi_{2m+1}(z)$ and~$\varphi_{2m+2}(z)$ allow analytic
    continuations to~$\C\setminus[\gamma_0,+\infty)$ by
    Theorem~\ref{th:CF_conv}~\eqref{item:Stieltjes_H}; moreover, on denoting these continuations
    also by~$\varphi_{2m+1}(z)$ and~$\varphi_{2m+2}(z)$, respectively, we
    have~$\varphi_{2m+1},\varphi_{2m+2}\in\NN_0^0$. Then, according to \eqref{eq:psi_n_recur}
    \begin{equation}\label{eq:phi_via_psi}
        \varphi_{2m+1}\left(z\right)=
        \alpha_{2m+1}-\psi_{2m+1}\left(\frac1z\right)
        \quad
        \text{and}
        \quad
        \varphi_{2m+2}\left(z\right)=
        \alpha_{2m+2}-\psi_{2m+2}\left(\frac1z\right),
    \end{equation}
    Lemmas~\ref{lemma:Nkappa_Moebius}~and~\ref{lemma:Nkappa_Moebius2} yields~$\psi_{2m+1}\in \NN_0$ and~$\psi_{2m+2}\in \NN_0$.
    So, if we use each of the last two functions and~$\varepsilon_{2m+1}=\varepsilon_{2m+2}=1$ as the
    `initial data' for Lemma~\ref{lemma:JFracRational}, we
    obtain~$\varepsilon_0\psi_0\in \NN_\kappa$ and~$\varepsilon_0\psi\in \NN_\lambda$ with $\kappa$ and $\lambda$ as defined in the theorem. Since
    \[
        z\varphi(z) = -\psi\Big(\frac{1}{z}\Big)
        \quad
        \text{and}
        \quad
        \varphi\left(z\right)=
        \varphi_{0}\left(z\right)=
        \alpha_0-\psi_0\left(\frac1z\right),
    \]
    Lemmas~\ref{lemma:Nkappa_Moebius}~and~\ref{lemma:Nkappa_Moebius2} imply that the pair of
    conditions~$\varepsilon_0\psi\in \NN_{\lambda}$ and~$\varepsilon_0\psi_0\in \NN_{\kappa}$ is
    equivalent to~$\varepsilon_0\varphi\in \NN_\kappa^\lambda$.
    
    Now let us prove the converse: assuming~$\varepsilon\varphi\in \NN_\kappa^\lambda$ for a
    certain~$\varepsilon=\pm1$ and~$\kappa,\lambda\in\N_0$, and
    hence~$\varepsilon\psi\in \NN_\lambda$ and~$\varepsilon\psi_0\in \NN_\kappa$, we will show
    that there is some~$m>0$ such that $\alpha_n>0$ for all~$n>2m$, and that the signs
    of~$\alpha_1,\dots\alpha_{2m}$ determine~$\kappa,\lambda$ via~\eqref{eq:eps_alpha_rel_inf}.
    Put~$\varepsilon_0\coloneqq\varepsilon$, ~$\varepsilon_1\coloneqq\varepsilon\sign\alpha_0$
    and~$\varepsilon_2\coloneqq\varepsilon_1\sign\alpha_1$ so that in view of \eqref{eq:psi_n_recur} and \eqref{eq:psi_psi1}
    \[
        \varepsilon\psi(z)=
        -\cfrac{\varepsilon_{1}\left|\alpha_{0}\right|}
        {z-\alpha_{1}+\varepsilon_{1}^2\psi_{1}(z)}
        \quad
        \text{and}
        \quad
        \varepsilon\psi_0(z)=
        -\cfrac{\varepsilon_{2}\left|\alpha_{0}\alpha_{1}\right|}
        {z-(\alpha_{1}+\alpha_2)+\varepsilon_{2}^2\psi_{2}(z)}
        .
    \]
    
    According to Lemma~\ref{lemma:Nkappa_Jacobi1}, we
    have~$\varepsilon_1\psi_1\in\NN_{\kappa_1}$ and~$\varepsilon_2\psi_2\in\NN_{\kappa_2}$,
    where~$\kappa_1=\lambda-\frac{1-\varepsilon_1}2$
    and~$\kappa_2=\kappa-\frac{1-\varepsilon_2}2$. Analogously,
    given~$\varepsilon_{n-2}\psi_{n-2}\in \NN_{\kappa_{n-2}}$ with some
    known~$\varepsilon_{n-2}=\pm 1,$ and~$\kappa_{n-2}$, we
    put~$\varepsilon_{n}\coloneqq\varepsilon_{n-2}\sign\left(\alpha_{n-2}\alpha_{n-1}\right)$
    and~$\kappa_n\coloneqq \kappa_{n-2}-\frac{1-\varepsilon_n}2$, so that~$\psi_n(z)$ satisfies
    \[
        \varepsilon_{n-2}\psi_{n-2}(z)
        =-\frac{\varepsilon_{n-2}\alpha_{n-2}\alpha_{n-1}}
        {z-(\alpha_{n-1}+\alpha_{n})+\psi_{n}(z)}
        =-\frac{\varepsilon_{n}
            \left|\alpha_{n-2}\alpha_{n-1}\right|}
        {z-(\alpha_{n-1}+\alpha_{n})+\varepsilon_{n}^2\psi_{n}(z)}
        \implies \varepsilon_{n}\psi_{n}(z)\in \NN_{\kappa_n}
    \]
    due to Lemma~\ref{lemma:Nkappa_Jacobi1}. The sequence of indices~$\kappa_1,\kappa_3,\dots$
    is nonnegative and non-increasing, as well as~$\kappa_2,\kappa_4,\dots$. So, proceeding further
    this way, one obtains
    \[
        \kappa_*=\lim_{n\to\infty}\kappa_{2n}\ge0
        \quad\text{and}\quad
        \lambda_*=\lim_{n\to\infty}\kappa_{2n+1}\ge0
        .
    \]
    As these sequences are built from integers, there exists~$m>0$ such that
    \[
        \kappa_{2m}=\kappa_{2m+2}=\cdots=\kappa_*,\quad
        \kappa_{2m+1}=\kappa_{2m+3}=\cdots=\lambda_*,
        \quad\text{and thus}\quad
        \varepsilon_{2m+1}=\varepsilon_{2m+2}=\cdots=1.
    \]
    For~$n>1$ the formulae~$\varepsilon_0=\varepsilon$,
    ~$\varepsilon_1=\varepsilon\sign\alpha_0$
    and~$\varepsilon_{n}=\varepsilon_{n-2}\sign\left(\alpha_{n-2}\alpha_{n-1}\right)$ give
    \[
        \varepsilon_{n}
        =\varepsilon_{n-2}\prod_{j=n-2}^{n-1}\sign\alpha_j
        =\varepsilon_{n-4}\prod_{j=n-4}^{n-1}\sign\alpha_j
        =\cdots
        =\varepsilon_{0}\prod_{j=0}^{n-1}\sign\alpha_j
    \]
    and, therefore,~$\sign\alpha_{n}=\varepsilon_{n}\varepsilon_{n+1}=1$ as soon as~$n>2m$. In
    particular, our new definition of~$\varepsilon_j$ coincides with the above
    formula~\eqref{eq:eps_alpha_rel_inf}.

    Now, to complete the proof it is enough to show that~$\kappa_*=\lambda_*=0$. By virtue of
    Theorem~\ref{th:CF_conv}~\eqref{item:Stieltjes_H}, both continued
    fractions~$\varphi_{2m+1}(z)$ and~$\varphi_{2m+2}(z)$ converge
    in~$\C\setminus[\gamma_0,+\infty)$ to analytic functions; moreover, on keeping the same
    labelling for these functions, we also have~$\varphi_{2m+1},\varphi_{2m+2}\in \NN_0^0$. By
    Lemma~\ref{lemma:Nkappa_Moebius2}, the expressions~\eqref{eq:phi_via_psi}
    yield~$\psi_{2m+1}\in\NN_{0}$ and~$\psi_{2m+2}\in\NN_{0}$. The latter, however, contradicts
    to~$\psi_{2m+1}\in\NN_{\lambda_{2m+1}}=\NN_{\lambda_*}$
    and~$\psi_{2m+2}\in\NN_{\kappa_{2m+2}}=\NN_{\kappa_*}$ unless~$\kappa_*=\lambda_*=0$.
\end{proof}
\begin{remark}
    The condition~$\sup_n|\alpha_n|<\infty$ may be relaxed. Nevertheless,
    Theorem~\ref{th:N_kl_CFraction} fails if one drops this condition completely: roughly
    speaking,~$\lambda_*$ and~$\kappa_*$ in the proof could remain positive. This depends on
    whether the corresponding moment problem determinate or not, see~\cite[Theorem~5]{Derkach}
    for the details.
\end{remark}

\begin{proof}[Proof of Theorem~\ref{th:R011Nkappa}]
    Let~\(\{-a,-b-1,a-c-1,b-c\}\cap\N_0\ne\varnothing.\) Then the
    coefficients~\eqref{eq:gen_cont_fr_cf} of the C-fraction~\eqref{eq:gen_cont_fr}
    corresponding to~$R_{0,1,1}(z)$ satisfy~$\alpha_0,\dots,\alpha_s\ne 0$ and~$\alpha_{s+1}=0$,
    where
    \[
        s_1\coloneqq\min\big(\{-a,b-c\}\cap\N_0\big),
        \quad
        s_2\coloneqq\min\big(\{-b,a-c\}\cap\N\big)
        \quad\text{and}\quad
        s\coloneqq\min\{2s_1,2s_2-1\};
    \]
    here we assume that~$\min(\varnothing)=+\infty$. In order to apply
    Corollary~\ref{cr:SFracRational}, let us calculate the
    numbers~$\varepsilon_j=\sign\prod_{l=j}^s\alpha_l$ for~$j=0,\dots,s$.

    Due to~$\alpha_0=1$, we have~$ \varepsilon\coloneqq\varepsilon_0 =1$ if~$s=0$, and
    \( \varepsilon\coloneqq\varepsilon_0 =\sign(\alpha_{0})\varepsilon_{1} =\varepsilon_{1} \) otherwise.
    Moreover, given any integer numbers~$m>j\ge 0$
    \begin{equation}\label{eq:finite_prod_alphas}
        \begin{aligned}
            \xi_j(m)
            \coloneqq{}&
            \sign\prod_{l=2j+1}^{2m}\alpha_l
            =\sign\prod_{n=j}^{m-1}\alpha_{2n+1}\alpha_{2n+2}
            \\
            ={}&
            \sign\prod_{n=j}^{m-1}
            \frac{(a+n)(c-b+n)\cdot(b+n+1)(c-a+n+1)}{(c+2n)(c+2n+1)\cdot(c+2n+1)(c+2n+2)}
            \\
            ={}&
            \sign
            \frac{(a+j)_{m-j}(c-b+j)_{m-j}(b+j+1)_{m-j}(c-a+j+1)_{m-j}}
            {(c+2j)\big[(c+2j+1)_{2(m-j)-1}\big]^2(c+2m)}
            \\
            ={}&
            \sign
            \frac{(a+j)_{m-j}(c-b+j)_{m-j}(b+j+1)_{m-j}(c-a+j+1)_{m-j}}
            {(c+2j)(c+2m)}
            ,
        \end{aligned}
    \end{equation}

    In the case~$s=2s_1$ we have~$\varepsilon_{2j+1} = \xi_j(s_1)$ and
    \[
        \varepsilon_{2j+2}=\frac{\varepsilon_{2j+1}}{\sign\alpha_{2j+1}}
        =
        \sign
        \frac{(a+j+1)_{s_1-j-1}(c-b+j+1)_{s_1-j-1}(b+j+1)_{s_1-j}(c-a+j+1)_{s_1-j}}
        {(c+2j+1)(c+2s_1)}
    \]
    for~$j=0,\dots,s_1-1$. Corollary~\ref{cr:SFracRational} yields
    \[
        \varepsilon_1 R_{0,1,1}\in\NN_{\kappa}^\lambda
                \quad\text{and}\quad
        -\varepsilon_1 R_{0,1,1}\in\NN_{K-\kappa}^{\Lambda-\lambda},
    \]
    where~$\Lambda=\big\lfloor\frac{s+2}{2}\big\rfloor=s_1+1$,
    ~$K=\big\lfloor\frac{s+1}{2}\big\rfloor=s_1$, and~$\lambda$ (resp.~$\kappa$) is the number
    of negative entries in the sequence $\varepsilon_1$, $\varepsilon_3$,
    \dots, $\varepsilon_{2s_1-1}$ (resp. $\varepsilon_{2}$, $\varepsilon_{4}$,
    \dots, $\varepsilon_{2s_1}$). In particular, for~$s_1=0$ we have~$\lambda=\kappa=0$.
    
    In the case~$s=2s_2-1$, if~$s_2=1$ we obtain~$\varepsilon_{1}=\sign\alpha_{1}$,
    otherwise~$s_2>1$ and for~$j=0,\dots, s_2-1$ the expression~\eqref{eq:finite_prod_alphas}
    implies
    \[
        \varepsilon_{2j+1}=\xi_j(s_2-1)\sign\alpha_{2s_2-1}
        % \\
        % &=
        % \sign
        % \frac{(a+s_2-1)(c-b+s_2-1)(a+j)_{s_2-1-j}(c-b+j)_{s_2-1-j}(b+j+1)_{s_2-1-j}(c-a+j+1)_{s_2-1-j}}
        % {(c+2s_2-2)(c+2s_2-1)(c+2j)(c+2s_2-2)}
        % \\
        =
        \sign
        \frac{\!(a+j)_{s_2-j}(c-b+j)_{s_2-j}(b+j+1)_{s_2-j-1}(c-a+j+1)_{s_2-j-1}\!}
        {(c+2j)(c+2s_2-1)}
        ,
    \]
    and hence for~$j=1,\dots, s_2-1$
    \[
        \begin{aligned}
            \varepsilon_{2j}&=\sign\alpha_{2j}\,\varepsilon_{2j+1}
            % \\ &=
            % \sign
            % \frac{(b+j)(c-a+j)(a+j)_{s_2-j}(c-b+j)_{s_2-j}(b+j+1)_{s_2-j-1}(c-a+j+1)_{s_2-j-1}}
            % {(c+2j-1)(c+2j)(c+2j)(c+2s_2-1)}
            % \\
            = % &=
            \sign
            \frac{(a+j)_{s_2-j}(c-b+j)_{s_2-j}(b+j)_{s_2-j}(c-a+j)_{s_2-j}}
            {(c+2j-1)(c+2s_2-1)}
            .
        \end{aligned}
    \]
    Moreover,~$K=\big\lfloor\frac{s+1}{2}\big\rfloor=s_2$
    and~$\Lambda=\big\lfloor\frac{s+2}{2}\big\rfloor=s_2$. By
    Corollary~\ref{cr:SFracRational},
    \[
        \varepsilon_1 R_{0,1,1}\in\NN_{\kappa}^\lambda
                \quad\text{and}\quad
        -\varepsilon_1 R_{0,1,1}\in\NN_{K-\kappa}^{\Lambda-\lambda},
    \]
    where~$\lambda$ (resp.~$\kappa$) is the number of negative entries in the sequence
    $\varepsilon_{1}$, $\varepsilon_{3}$, \dots, $\varepsilon_{2s_2-1}$ (resp.
    $\varepsilon_{2}$, $\varepsilon_{4}$, \dots, $\varepsilon_{2s_2-2}$). Here we
    obtain~$\kappa=0$ when~$s_2=1$.

    In both cases~$s=2s_1$ and~$s=2s_2-1$, the ratio~$R_{0,1,1}(z)$ is a rational function of
    degree~$K$, and~$\Lambda$ is the degree of~$zR_{0,1,1}(z)$, see
    Corollary~\ref{cr:SFracRational}.

    Now, let~\(\{-a,-b-1,a-c-1,b-c\}\cap\N_0=\varnothing.\) None of the
    coefficients~\eqref{eq:gen_cont_fr_cf} of the C-fraction~\eqref{eq:gen_cont_fr} vanishes,
    that is the function~$R_{0,1,1}(z)$ is not rational. Furthermore, $\alpha_n>0$ for all~$n>m$
    provided that~$m$ is large enough. So, Theorem~\ref{th:N_kl_CFraction} implies
    that~$\varepsilon R_{0,1,1}\in\NN_\kappa^\lambda$ for~$\varepsilon=\varepsilon_0$
    and~$\kappa,\lambda\in\N_0$ determined by the
    numbers~$\varepsilon_{j}\coloneqq\prod_{n=j}^\infty\sign\alpha_n$. These numbers may be
    calculated via the identity~\eqref{eq:finite_prod_alphas}:
    \[
        \begin{aligned}
        \varepsilon_{2j+1}&=\!\lim_{m\to\infty}\!\xi_j(m)
        %\\
        % =
        % \lim_{m\to\infty}\sign
        % \frac{(c+2m)\Gamma(a+m)\Gamma(c-b+m)\Gamma(b+m+1)\Gamma(c-a+m+1)}
        % {(c+2j)\Gamma(a+j)\Gamma(c-b+j)\Gamma(b+j+1)\Gamma(c-a+j+1)}
        % \\
        =% &=
        \frac{\lim_{m\to\infty}\sign\big[(c+2m)\Gamma(a+m)\Gamma(c-b+m)\Gamma(b+m+1)\Gamma(c-a+m+1)\big]}
        {\sign\big[(c+2j)\Gamma(a+j)\Gamma(c-b+j)\Gamma(b+j+1)\Gamma(c-a+j+1)\big]}
        \\
        &=\sign(c+2j)\Gamma(a+j)\Gamma(c-b+j)\Gamma(b+j+1)\Gamma(c-a+j+1)
        \end{aligned}
    \]
    and
    \[
        \varepsilon_{2j+2}= \frac{\varepsilon_{2j+1}}{\sign\alpha_{2j+1}}
        =\sign(c+2j+1)\Gamma(a+j+1)\Gamma(c-b+j+1)\Gamma(b+j+1)\Gamma(c-a+j+1)
    \]
    where~$j\in\N_0$. Furthermore,
    $\varepsilon_0 =\sign(\alpha_{0})\varepsilon_{1} =\varepsilon_{1}$.
\end{proof}

\subsection{Multiplicative factorizations}\label{sec:mult-fact}

Recall the striking description of the generalized Nevanlinna classes given in~\cite{DHdS,DLLS}:

\begin{theorem}\label{th:DLLS}
    A function~$f(z)$ belongs to the class~$\NN_\kappa$ if and only if it may be written in the
    form~$f(z)=Q(z)g(z)$, where~$g\in \NN_0$ and~$Q(z)$ is a real rational function of
    degree~$2\kappa$ nonnegative on the real line \emph{(}except for its poles\emph{)}.
\end{theorem}
Under the assumptions of this theorem, for~$x\in\R$ excluding the real poles and zeros of~$Q(z)$
we immediately obtain
\begin{equation}\label{eq:f_Nkappa_pos}
    \liminf_{y\to+0}\Im f(x+iy) = Q(x)\liminf_{y\to+0}\Im g(x+iy)\in[0,+\infty]
    ,
\end{equation}
where the right-hand side is due to~$\Im g(x+iy)\ge 0$ for all~$y>0$.

The representation offered by Theorem~\ref{th:DLLS} is unique up to multiplication of $Q(z)$ by
a positive constant: the poles and zeros of~$Q(z)$ in~$\C\cup\{\infty\}$ determine the local
analytic properties of~$f(z)$ uniquely. If~$f\in\NN_\kappa\subset\NU$ with $\NU$ defined in \eqref{eq:Nevanlinna-U} is written
as~$f(z)=Q(z)g(z)$ according to Theorem~\ref{th:DLLS} and~$T(z)$ is a rational function
nonnegative on the real axis, then also~$Tf=TQg\in\NN_{\frac12\deg(TQ)}\subset\NU$. The following theorem states the conditions, under which a similar fact holds
without the assumption~$T(z)\ge 0$ on~$\R$.

\begin{theorem}\label{th:NkappaM2}
    Let~$g\in\NU$, and let~$T(z)$ be a real rational function. Denote
    \begin{equation}\label{eq:def_I_via_T}
        \mathcal I\coloneqq\big\{t\in\R:-\infty<T(t)< 0\big\}.
    \end{equation}
    Then $f\coloneqq Tg\in\NU$ if and only if $g(z)$ is analytic and real on~$\mathcal I$
    excluding at most a finite number of poles.
\end{theorem}
\begin{proof}
    By Theorem~\ref{th:DLLS}, we can write~$g(z)=Q(z)\widetilde g(z)$ for
    some~$\widetilde g\in \NN_0$ and some real rational function~$Q(z)$ nonnegative on the real
    line. The set of poles of~$g(z)$ is finite if and only if the same is true
    for~$\widetilde g(z)$. Thus, the proof reduces to an application of a particular case of
    Theorem~\ref{th:NkappaM2} stated here as Lemma~\ref{lm:NkappaM2a} to the
    function~$f(z)=\widetilde Q(z)\widetilde g(z)$, where~$\widetilde Q(z) \coloneqq T(z)Q(z)$.
\end{proof}

\begin{lemma}\label{lm:NkappaM2a}
    Let~$g\in \NN_0$, and let~$T(z)$ be a real rational function. Define~$\mathcal I$
    via~\eqref{eq:def_I_via_T}. Then in order for~$f(z)\coloneqq T(z)g(z)$ to lie in~$\NU$ it is necessary and
    sufficient that~$g(z)$ is analytic and real on~$\mathcal I$ excluding at most a finite
    number of poles.
    
    Moreover, if~$f\in \NN_\kappa$ for some~$\kappa$, then~$\mathcal I$ may contain at
    most~$\kappa$ of its poles.
\end{lemma}

\begin{proof}\refstepcounter{dummy}\label{pg:proof_lm:NkappaM2a}
    Write~$g(z)$ in the form \eqref{eq:N0canonical}
    \begin{equation}\label{eq:g_via_sigma}
        g(z)
        = \nu_1z+\nu_2
        + \int_{\R}\left( \frac{1}{t-z}-\frac{t}{t^2+1}\right) \,d\sigma(t)
        = \nu_1z+\nu_2
        + \int_{\R}\frac{1+tz}{t-z}\cdot \frac{d\sigma(t)}{t^2+1},
    \end{equation}
    where~$\nu_1\ge 0$, ~$\nu_2\in\R$ and~$\sigma(t)$ is a non-decreasing function satisfying~$\int_{-\infty}^{\infty}(1+t^2)^{-1}d\sigma(t)<\infty$. 
    According to the
    Stieltjes-Perron inversion formula (see e.g.~\cite[pp. 124--126]{Akhiezer}), if~$(a,b)$ is a
    real interval, then
    \begin{equation} \label{eq:Stieltjes-Perron}
        \frac{\sigma(b+)-\sigma(b-)}{2}-
        \frac{\sigma(a+)-\sigma(a-)}{2}
        =\lim_{y\to+0}\frac 1\pi\int_{a}^b\Im g(x+iy)\,dx.
    \end{equation}
    In particular,~$g(z)$ is analytic and real on~$(a,b)$ precisely
    when~$\sigma(a+)=\sigma(b-)$, which is equivalent to that~$\sigma'(t)$ is defined and equals
    zero for all~$t\in(a,b)$.

    We prove the necessity part of the lemma by contradiction. Suppose that~$f\in \NN_\kappa$ and yet there are
    points~$t_0,\dots,t_\kappa\in\mathcal I$ such that
    \begin{equation} \label{eq:sigma_density}
        \limsup_{\varepsilon\to0}
        \frac{\sigma(t_j+\varepsilon)-\sigma(t_j-\varepsilon)}{2\varepsilon} \in(0,+\infty]
        ,
        \qquad
        j=0,\dots,\kappa,
    \end{equation}
    (this ratio is non-negative since~$\sigma(t)$ is non-decreasing). Then all poles of~$g(z)$ lying in~$\mathcal{I}$ are among the numbers~$t_0,\dots,t_\kappa$.    For~$y\in(0,1)$,
%    \[
%        \begin{aligned}
%            \left|y g(t_j+iy) + i\big(\sigma(t_j+)-\sigma(t_j-)\big)\right|
%            &=\left|
%                \int_{\R\setminus\{t_j\}}\frac{1+t(t_j+iy)}{t-(t_j+iy)}\cdot \frac{y\,d\sigma(t)}{t^2+1}
%            \right|
%            \\
%            &\le 
%            \int_{\R\setminus\{t_j\}}
%            \frac{1+|t|(|t_j|+1)}
%            {\sqrt{(t-t_j)^2y^{-2}+1}}
%            \cdot \frac{d\sigma(t)}{t^2+1}
%            .
%        \end{aligned}
%    \]
\begin{multline*}
yg(t_j+iy)=y\nu_1(t_j+iy)+y\nu_2+y\int_{\R}\frac{1+t(t_j+iy)}{t-(t_j+iy)}\frac{d\sigma(t)}{t^2+1}
\\
=y\nu_1(t_j+iy)+y\nu_2+y\int_{\R\setminus\{t_j\}}\frac{1+t(t_j+iy)}{t-(t_j+iy)}\frac{d\sigma(t)}{t^2+1}
+i\big(\sigma(t_j+)-\sigma(t_j-))
-\big(\sigma(t_j+)-\sigma(t_j-))
\frac{t_{j}y}{t_{j}^2+1},
\end{multline*}    
so that
\begin{multline*}
\left|yg(t_j+iy)-i\big(\sigma(t_j+)-\sigma(t_j-)\big)\right|
\\
\le y\left|\nu_1(t_j+iy)+\nu_2-\big(\sigma(t_j+)-\sigma(t_j-))
\frac{t_{j}}{t_{j}^2+1}\right|
+\left|\int_{\R\setminus\{t_j\}}\frac{1+t(t_j+iy)}{t-(t_j+iy)}\cdot\frac{y\,d\sigma(t)}{t^2+1}\right|
\\
\le Ay
+\int_{\R\setminus\{t_j\}}
            \frac{1+|t|(|t_j|+1)}
            {\sqrt{(t-t_j)^2y^{-2}+1}}
            \cdot \frac{d\sigma(t)}{t^2+1},
\end{multline*}    
where $A\ge0$ does not depend on $y$. On the right-hand side, the second factor in the integrand
satisfies~$\int_{-\infty}^{\infty} (1+t^2)^{-1}d\sigma(t)<\infty$, while the first factor is
bounded uniformly in $y\in(0,1)$ and~$t\in\R$ and vanishes for each~$t\ne t_j$ as~$y\to0$. Thus,
the right-hand side tends to zero by the dominated convergence theorem. This leads to the
asymptotic formula
\begin{equation}\label{eq:yg_is_im_o1}
    yg(t_j+iy)=i\big(\sigma(t_j+)-\sigma(t_j-)\big)+o(1),
    \quad\text{where } y\to0 \text{ and }
    j=0,\dots,\kappa.
\end{equation}
According to definition \eqref{eq:def_I_via_T} the rational function~$T(z)$ is bounded
on~$\mathcal I$, and hence for~$y\to +0$ and~$j=0,\dots,\kappa$ we have~$T(t_j+iy)=T(t_j)+O(y)$
and~$f(t_j+iy)=\big(T(t_j)+O(y)\big)g(t_j+iy)$. So, the formula~\eqref{eq:yg_is_im_o1} gives
    \begin{equation}\label{eq:yPsi_est}
        \begin{aligned}
            y f(t_j+iy)
            &=iy\Im f(t_j+iy)+\Re\left[\big(T(t_j)+O(y)\big)y g(t_j+iy)\right]
            \\[2pt]
            &=iy\Im f(t_j+iy)+o(1)
            \quad\text{as}\quad
            y\to +0
            .
        \end{aligned}
    \end{equation}

Taking the imaginary part of \eqref{eq:g_via_sigma} we get the Poisson representation
$$
\Im{g(t_j+iy)}=\nu_1{y}+\int_{\R}\frac{yd\sigma(t)}{(t-t_{j})^2+y^2}.
$$  
In view of $\nu_1\ge0$ and $y>0$, the above equality leads to the estimate 
    \[
        \Im
        g(t_j+iy)
        \ge
        \int_{\left(t_j-y,t_j+y\right)}
        \frac{y\,
            d\sigma(t)}{(t-t_j)^2+y^2}
        \ge
        \int_{\left(t_j-y,t_j+y\right)}
        \frac{y\,
            d\sigma(t)}{y^2+y^2}
        =
        \frac{\sigma(t_j+y-)-\sigma(t_j-y+)}{2y}.
    \]
    Therefore, from~$T(t_j)<0$ and~$T(t_j+iy)=T(t_j)+T'(t_j)iy+O(y^2)$ we obtain
    \begin{equation*}
        \begin{aligned}
            \Im f(t_j+iy)
            &=
            \big(T(t_j)+O(y^2)\big)
            \Im g(t_j+iy)
            +
            \big(T'(t_j)+O(y)\big)
            \Re\big(y g(t_j+iy)\big)
            \\ &
            \le
            \big(T(t_j)+O(y^2)\big)
            \frac{\sigma(t_j+y-)-\sigma(t_j-y+)}{2y}
            +o\left(1\right)
            \quad\text{as}\quad
            y\to +0
            ,
        \end{aligned}
    \end{equation*}
    where the last inequality is due to~\eqref{eq:yg_is_im_o1}. Now, the last
    estimate %~\eqref{eq:yImPsi_est}
    along with~\eqref{eq:sigma_density} implies that there exist a number~$C>0$ and a
    sequence~$y_1,y_2,\dots$ of positive reals such that
    \begin{equation}\label{eq:ImPsi_est}
        \lim_{l\to+\infty}y_l=0
        \quad\text{and}\quad
        \Im f(t_j+i y_l)\le -C
    \end{equation}
    for all~$l=1,2,\dots$ and~$j=0,\dots,\kappa$. Denote $z_{j,l}\coloneqq t_j+ iy_l$ and
    observe that the matrix
    \[
        M_l= (m_{j,k})_{j,k=0}^{\kappa},
        \quad
        m_{j,k}\coloneqq
        \frac{y_l}{\sqrt{\Im f(z_{j,l})\Im f(z_{k,l})}}\cdot
        \frac{f(z_{j,l})-\overline{f(z_{k,l})}}
        {z_{j,l}-\overline{z_{k,l}}},
    \]
    has the diagonal entries~$m_{j,j}=-1$ for all~$j$. Due to~\eqref{eq:yPsi_est}
    and~\eqref{eq:ImPsi_est}, the off-diagonal entries of~$M_l$ as~$l\to\infty$ are
    \[
        \begin{aligned}
        m_{j,k}
        &=
        \frac{y_l}
        {z_{j,l}-\overline{z_{k,l}}}\cdot
        \frac{f(z_{j,l})-\overline{f(z_{k,l})}}{\sqrt{\Im f(z_{j,l})\Im f(z_{k,l})}}
        =
        \frac{i}{z_{j,l}-\overline{z_{k,l}}}\cdot
        \frac{y_l\Im f(z_{j,l})+y_l\Im f(z_{k,l})+o(1)}
        {\sqrt{\Im f(z_{j,l})\Im f(z_{k,l})}}
        \\
        &=
        \frac{i\sqrt{y_l}}{t_{j}-t_{k}+2iy_{l}}\cdot
        \left(\sqrt{y_l\frac{\Im f(z_{j,l})}{\Im f(z_{k,l})}}
            +\sqrt{y_l\frac{\Im f(z_{k,l})}{\Im f(z_{j,l})}}\right)
        +o(1)\to 0
        .
        \end{aligned}
    \]        
    The ultimate $o(1)$ is due to the estimate 
    $$
     \frac{1}{\sqrt{\Im f(z_{j,l})\Im f(z_{k,l})}}\le\frac{1}{C}
    $$
    following from \eqref{eq:ImPsi_est}. For~$l\to\infty$, the matrix~$M_l$ therefore satisfies
    \[
        D\cdot H_f(t_0+i y_l,\dots,t_{\kappa}+i y_l)\cdot D
        =M_l
        \to -I_{\kappa+1}
        % \quad\text{as}\quad l\to\infty
        ,
        \quad\text{where}\quad
        D\coloneqq
        \diag\left(\sqrt{\frac{y_l}{-\Im f(t_j+i y_l)}}\right)_{j=0}^{\kappa}
    \]
    and~$I_{\kappa+1}$ is the~$(\kappa+1)\times(\kappa+1)$ identity matrix. For~$l$ large
    enough, the left-hand side is negative definite, which contradicts our
    assumption~$f\in \NN_{\kappa}$.

    Now, let us prove the sufficiency part. Since~$g(z)$ may only have a certain finite number of poles in~$\mathcal I$  (say~$m$), it may be written in the form~$g(z)=g_1(z)+g_2(z)$,
    where
    \[
        g_1(z)\coloneqq
        \sum_{k=1}^m\frac{\sigma(t_k+)-\sigma(t_k-)}{t_k-z}
        \quad\text{and}\quad
        g_2(z)\coloneqq
        \nu_1z +\nu_2+\int_{\R\setminus\mathcal I} \left(
            \frac1{t-z}-\frac t{t^2+1}\right)\,d\sigma(t)
        ,
    \]
    and~$t_1,\dots,t_m\in\mathcal I$. Let the co-prime polynomials~$p(z)$ and~$q(z)$ denote,
    respectively, the numerator and the denominator of~$T(z)$, so that~$T(z)=\frac{p(z)}{q(z)}$.
    Choose arbitrary~$t_*\in\mathcal I$, then the
    function~$T_1(z)\coloneqq p(z)q(z)(z-t_*)^{-2\deg T}$ is nonnegative, bounded and analytic
    on~$\R\setminus\mathcal I$. Therefore, the new positive
    measure~$d\tau(t)\coloneqq T_1(t)\,d\sigma(t)$ on~$\R\setminus\mathcal I$ is properly
    defined and satisfies~$\int_{\R\setminus\mathcal I} (1+t^2)^{-1}d\tau(t)<\infty$, so the
    integral
    \[
        g_3(z)\coloneqq
        \int_{\R\setminus\mathcal I} \left( \frac1{t-z}-\frac t{t^2+1}\right)\,d\tau(t)
    \]
    defines a function~$g_3\in\NN_0$. Therefore,
    \[
        S_1(z)\coloneqq T_1(z)g_2(z)-g_3(z)
        =
        T_1(z)\left(\nu_1z +\nu_2\right)+
        \int_{\R\setminus\mathcal I}
        \frac{(1+tz)(T_1(z)-T_1(t))}{t-z}\frac{d\sigma(t)}{t^2+1}
        ,
    \]
    where the integrand is analytic and bounded in~$z$ outside each small disc centred at~$t_*$.
    So,~$S_1(z)$ is manifestly
    analytic in~$\C\setminus\{t_*\}$ and real on~$\R\setminus\{t_*\}$. We want to show that it
    is actually rational. On writing the Laurent expansion of~$T_1(z)$ around~$t_*$ as
    \[
        T_1(z)=\sum_{j=0}^{2\deg T} \frac{c_{-j}}{(z-t_*)^j}
    \]
    with some coefficients~$c_{-{2\deg T}},\dots,c_0\in \R$, we see that
    \[
        \frac{T_1(z)-T_1(t)}{t-z}
        =
        \frac{\sum_{j=1}^{2\deg T} c_{-j}\left((z-t_*)^{-j}-(t-t_*)^{-j}\right)}{(t-t_*)-(z-t_*)}
        =
        \sum_{j=1}^{2\deg T} \sum_{k=0}^{j-1}\frac {c_{-j}}{(t-t_*)^{k+1}(z-t_*)^{j-k}}.
    \]
    Put~$\delta\coloneqq\operatorname{dist}\{t_*,\R\setminus\mathcal I\}$, then
    for~$|t-t_*|\ge\delta$ and~$|z-t_*|\ge\delta$
    \[
        \left|\frac{T_1(z)-T_1(t)}{t-z}\right|
        \le
        \sum_{j=1}^{2\deg T}\sum_{k=0}^{j-1}\frac {|c_{-j}|}{\delta^{j} |t-t_*|}
        =
        \frac{C}{|t-t_*|},
        \quad\text{where}\quad
        C\coloneqq\sum_{j=1}^{2\deg T} j\frac {|c_{-j}|}{\delta^{j}}.
    \]
    Therefore,
    \[
        |S_1(z)|\le
        \left|T_1(z)\left(\nu_1z +\nu_2\right)\right|+
        \int_{\R\setminus\mathcal I}
        \frac{C+C\left|zt\right|}{|t-t_*|}\cdot\frac{d\sigma(t)}{t^2+1}
        =O(z)
        \quad\text{as}\quad z\to \infty.
    \]
    By definition of~$S_1(z)$, its only finite singularity at the point~$t_*$ is a pole of order~$2\deg T$. So, Liouville's theorem implies that~$S_1(z)$ may only be a real rational function of degree at
    most~$2\deg T+1$.   Now, for the real rational functions
    \[
        Q(z) \coloneqq  \frac{(z-t_*)^{\deg T}}{q(z)}
        \quad\text{and}\quad
        S(z)\coloneqq T(z)g_1(z)+Q^2(z)S_1(z)
    \]
    satisfying~$\deg S\le (\deg T+m) + (2\deg T +1) = 3\deg T+m+1$, we obtain
    \[
        f(z)
        =T(z)g(z)
        =T(z)g_1(z)+Q^2(z)T_1(z)g_2(z)
        =S(z) + Q^2(z) g_3(z)
        .
    \]
    It remains to employ a reasoning analogous to~\cite[Section~3]{DHdS} to show
    that~$f\in\NU$: given some numbers~$z_1,\dots,z_n\in\C_+$, write the entries of the
    matrix~$H_f(z_1,\dots,z_n)$ in the following form:
    \[
        \begin{aligned}
            \frac{f(z_i)-\overline{f(z_j)}}{z_i-\overline{z_j}}
            &=
            \frac{S(z_i)-\overline{S(z_j)}}{z_i-\overline{z_j}}
            +
            \frac{
                Q(z_i) g_3(z_i)\overline{Q(\overline z_i)}
                -
                \overline{Q(z_j)} \overline{g_3(z_j)} Q(\overline z_j)
            }
            {z_i-\overline{z_j}}
            \\
            &=
            \frac{S(z_i)-\overline{S(z_j)}}{z_i-\overline{z_j}}
            +
            \overline{Q(\overline z_i)}
            \frac{g_3(z_i) - \overline{g_3(z_j)}}{z_i-\overline{z_j}}
            Q(\overline{z_j})
            \\
            &+
            g_3(z_i)\overline{Q(\overline z_i)}
            \frac{Q(z_i)-Q(\overline{z_j})}{z_i-\overline{z_j}}
            +
            \overline{g_3(z_j)}Q(\overline{z_j})
            \frac{\overline{Q(z_j)}-\overline{Q(\overline z_i)}}{\overline{z_j}-z_i}
            .
        \end{aligned}
    \]
    The right-hand side is the $(i,j)$-th entry of a sum of three Hermitian matrices:
    \[
        \begin{aligned}
            H_f(z_1,\dots,z_n)
            &=
            H_S(z_1,\dots,z_n) + D_1\cdot H_{g_3}(z_1,\dots,z_n)\cdot D_1^*
            \\
            &+
            \Big(D_2\cdot H_{Q}(z_1,\dots,z_n)
            +\big(D_2\cdot H_{Q}(z_1,\dots,z_n)\big)^*\Big)
            ,
        \end{aligned}
    \]
    where~$D_1$, $D_2$ are given by $D_1=D_1^*=\diag\big(Q(z_j)\big)_{j=1}^{n}$
    and~$ D_2=\diag\big(g_3(z_j)Q(z_j)\big)_{j=1}^{n}$. The first matrix~$H_S(z_1,\dots,z_n)$
    comes from a real rational function~$S$ and hence its rank is at most~$\deg S$. The second
    is positive semidefinite, as it is diagonally similar to~$H_{g_3}(z_1,\dots,z_n)$
    and~$g_3\in \NN_0$. The rank of the third matrix is bounded from above by~$2\deg Q$: it is
    the Hermitian part of the product~$D_2\cdot H_{Q}(z_1,\dots,z_n)$ whose rank does not
    exceed~$\deg Q$. By Lemma~\ref{lm:Hsum}, the number of negative eigenvalues
    of~$H_f(z_1,\dots,z_n)$ cannot exceed~$\deg S+2\deg Q$.
\end{proof}

\begin{corollary}\label{cr:N_kappa_lambda_M2}
    Let~$f\in \NU$. Then~$f\in\SU$ if and only if~$f(z)$ is real and analytic for~$z<0$
    excluding at most finitely many poles.
\end{corollary}
\begin{proof}
    The corollary follows from Theorem~\ref{th:NkappaM2} by choosing~$T(z)=z$, since~$f\in \NN_\kappa^\lambda$
    precisely when~$f(z)$ and~$zf(z)$ belong to~$\NN_\kappa$ and~$\NN_\lambda$, respectively.
\end{proof}

\begin{corollary}\label{cr:N_kappa_lambda_M3}
    Let~$f\in\SU$, and let~$h(z)$ and~$g(z)$ be two real rational functions.
    Then~$hf+g\in\SU$ if and only if~$f(z)$ is real and analytic
    in~$\mathcal I\coloneqq\big\{t\in\R:-\infty<h(t)< 0\big\}$ excluding at most a finite number of poles.
\end{corollary}
\begin{proof}
    Recall that~$f\in\SU$ if and only if~$f,\widetilde f\in \NU$,
    where~$\widetilde{f}(z)\coloneqq zf(z)$. The function~$\widetilde f(z)$ has the same
    singular points in~$\mathcal I$ as~$f(z)$, possibly excluding the origin. So, by
    Theorem~\ref{th:NkappaM2}, one has both~$hf\in\NU$
    and~$h\widetilde{f}\in\NU$ if and only if~$f(z)$ is real and analytic
    in~$\mathcal I$ except for at most a finite number of poles.

    Now, for any points~$z_1,\dots, z_n\subset \C_+$
    \begin{align*}
      H_{hf+g}(z_1,\dots,z_n)
      &=H_{hf}(z_1,\dots,z_n)+H_{g}(z_1,\dots,z_n)
        \quad\text{and}
      \\
      H_{h\widetilde{f}+\widetilde{g}}(z_1,\dots,z_n)
      &=H_{h\widetilde{f}}(z_1,\dots,z_n)+H_{\widetilde{g}}(z_1,\dots,z_n)
        ,
    \end{align*}
    where~$\widetilde{g}(z)\coloneqq zg(z)$. Since both matrices~$H_g(z_1,\dots,z_n)$
    and~$H_{\widetilde g}(z_1,\dots,z_n)$ have rank~$\le\deg g+1$ according to
    Lemma~\ref{lm:Nkappa_rational}, the conditions~$hf+g\in\SU$ and~$hf\in\SU$
    are equivalent.
\end{proof}

\begin{proof}[Proof of Theorem~\ref{th:RnmNkappa}]
    \refstepcounter{dummy}\label{pg:pr_th_RnmNkappa}%
    If~$R_{n_1,n_2,m}(z)$ is rational, then we immediately have both: the equality
    in~\eqref{eq:boundIneq} in view of \eqref{eq:2F1ratioboundary} and~$\pm R_{n_1,n_2,m}\in\SU$
    by Lemma~\ref{lm:Nkappa_rational}. Thus, it is enough only to consider the case
    when~$R_{n_1,n_2,m}(z)$ is not rational. Under this assumption we will prove both directions
    of the theorem simultaneously.

    Choose the integer~$\delta\ge 0$ to be large enough so that
    \[
        0\le a+\delta \le c+2\delta
        \quad\text{and}\quad
        0\le b+\delta+1 \le c+2\delta
        .
    \]
   Then the coefficients~\eqref{eq:gen_cont_fr_cf} of the continued fraction~\eqref{eq:gen_cont_fr} corresponding to
    \[
        \widetilde R_{0,1,1}(z)
        \coloneqq
        \frac{{}_2F_1(a+\delta,b+\delta+1;c+2\delta+1;z)}{{}_2F_1(a+\delta,b+\delta;c+2\delta;z)}
    \]
    are all strictly positive, and hence~$\widetilde R_{0,1,1}(z)$ is not rational. Moreover,
    Theorem~\ref{th:CF_conv} implies~$\widetilde R_{0,1,1}\in\mathcal S\subset\SU$. The
    condition~\eqref{item:Run4} of Theorem~\ref{th:2F1zeros} is satisfied, so~$\widetilde R_{0,1,1}(x+i0)$ is
    continuous for all~$x>1$ and, according to  Theorem~\ref{th:2F1ratioboundary} and Example~1 below, has strictly positive
    imaginary part.%
    \footnote{One can avoid Theorem~\ref{th:2F1zeros} here by invoking
        Theorem~\ref{th:CF_conv}~\eqref{item:Blumenthal}.}

    Given shifts~$n_1,n_2,m\in\mathbb Z$ and~$\delta\in\N_0$, there exist
    (see~\cite[pp.~130--133]{Gauss} or~\cite[eq. (1.1)]{Ebisu}) nontrivial real coprime
    polynomials~$p_\delta(z)$, $q_\delta(z)$ and~$r_\delta(z)$ in~$z$ (as well as in~$a,b,c$)
    such that%~$\gcd(p_\delta,q_\delta,r_\delta)=1$ and
    \begin{equation}\label{eq:R_shifts_via_G}
        p_\delta(z)
        \frac{{}_2F_1(a+n_1,b+n_2;c+m;z)}{{}_2F_1(a+\delta,b+\delta;c+2\delta;z)}
        = q_\delta(z)
        \widetilde R_{0,1,1}(z)
        %\\
        + r_\delta(z)
        .
    \end{equation}
    In our case, the numbers~$a,b,c$ are fixed. Since~$\widetilde R_{0,1,1}(z)$ is non-rational for our choice of~$\delta$, the polynomial~$p_\delta(z)$ cannot vanish identically.%
    \footnote{Under `nontrivial' polynomials we understand that one cannot simultaneously
        have~$p_\delta(z)$, $q_\delta(z)$ and~$r_\delta(z)$ identically equal to zero (as
        polynomials in~$z$). The universal choice~$\delta=0$ would shorten the proof, but then
        it would be possible that~$\widetilde R_{0,1,1}(z)=R_{0,1,1}(z)$ is rational. The latter
        would lead to the undesired case~$p_0(z)\equiv 0$ and~$q_0(z),r_0(z)\not\equiv 0$,
        since~$R_{n_1,n_2,m}(z)$ is not rational.}

    The ratio of two version of the equality~\eqref{eq:R_shifts_via_G} -- the second one
    with~$n_1=n_2=m=0$ and polynomials $\widetilde p_\delta(z)\not\equiv 0$, ~$\widetilde q_\delta(z)$
    and~$\widetilde r_\delta(z)$ -- may be written as
    \begin{equation}\label{eq:R_shifts_via_G_delta}
        \frac {p_\delta(z)}{\widetilde p_\delta(z)}
        R_{n_1,n_2,m}(z)
        =
        \frac{q_\delta(z)\widetilde R_{0,1,1}(z) + r_\delta(z)}
        {\widetilde q_\delta(z)\widetilde R_{0,1,1}(z) + \widetilde r_\delta(z)}.
    \end{equation}
    If~$\widetilde q_\delta(z)\equiv 0$, then we obtain
    \begin{equation}\label{eq:R_shifts_via_G0}
        R_{n_1,n_2,m}(z) = \widetilde h(z) \widetilde R_{0,1,1}(z) + \widetilde g(z)
    \end{equation}
    with real rational functions~$\widetilde h(z)$ and~$\widetilde g(z)$.
    Here~$\widetilde h(z)\not\equiv 0$, since~$R_{n_1,n_2,m}(z)$ is non-rational. 
    Then for all~$x>1$
    excluding zeros and poles of~$\widetilde h(x)$,
    \begin{equation}\label{eq:R_shifts_via_G1}
        0<\Im[\widetilde R_{0,1,1}(x+i0)]
        =
        \Im\frac{R_{n_1,n_2,m}(x+i0)}{\widetilde h(x)}
        =
        \frac{\pi B_{n_1,n_2,m}P_{r}(1/x)
            x^{l-\underline{n}-c}(x-1)^{c-a-b-l}}{\widetilde h(x)\cdot|{}_{2}F_{1}(a,b;c;x)|^{2}},
    \end{equation}
    where the last equality is~\eqref{eq:2F1ratioboundary}. The inequality~\eqref{eq:boundIneq}
    holds if and only if $\widetilde h(x)>0$ for all~$x>1$ excluding possible poles and zeros.
    The latter is equivalent to~$R_{n_1,n_2,m}\in\SU$ due to~\eqref{eq:R_shifts_via_G0} and
    Corollary~\ref{cr:N_kappa_lambda_M3}.
    
    If~$\widetilde q_\delta(z)\not\equiv 0$, then~\eqref{eq:R_shifts_via_G_delta} yields
    \begin{equation}\label{eq:R_shifts_via_G2}
        R_{n_1,n_2,m}(z)
        =
        \frac{\widetilde p_\delta(z)}{p_\delta(z)}\left(
            \frac{r_\delta(z)- \widetilde r_\delta(z)q_\delta(z)/\widetilde q_\delta(z)}
            {\widetilde q_\delta(z)\widetilde R_{0,1,1}(z) + \widetilde r_\delta(z)}
            +\frac{q_\delta(z)}{\widetilde q_\delta(z)}
        \right)
        \eqqcolon \frac{-h(z)}{\widetilde R_{0,1,1}(z)+g(z)} +g_1(z)
    \end{equation}
    for certain real rational functions~$h(z),g(z)$ and~$g_1(z)$: note that~$h(z)\not\equiv 0$,
    otherwise~$R_{n_1,n_2,m}(z)$ would be rational. Analogously to~\eqref{eq:R_shifts_via_G1} we
    have
    \begin{multline*}
        0<
        \frac{\Im\widetilde R_{0,1,1}(x+i0)}{\big|\widetilde R_{0,1,1}(x+i0)+g(x)\big|^2}
        =
        \Im\frac{-1}{\widetilde R_{0,1,1}(x+i0)+g(x)}
        \\=
        \Im\frac{R_{n_1,n_2,m}(x+i0)}{h(x)}
        =
        \frac{\pi B_{n_1,n_2,m}P_{r}(1/x)
            x^{l-\underline{n}-c}(x-1)^{c-a-b-l}}{h(x) \cdot|{}_{2}F_{1}(a,b;c;x)|^{2}}
    \end{multline*}
    for all~$x>1$ excluding possible poles and zeros of~$h(x)$. Consequently,~$h(x)>0$ if and
    only if~$B_{n_1,n_2,m}P_{r}(1/x)>0$; the latter inequality is equivalent
    to~\eqref{eq:boundIneq} combined with~$B_{n_1,n_2,m}P_{r}(x)\not\equiv 0$.
    From~\eqref{eq:R_shifts_via_G2} we obtain
    \[
        \widetilde R_{0,1,1}\in\mathcal S
        \implies
        \widetilde R_{0,1,1}+g\in\SU
        \iff
       R_{n_1,n_2,m}-g_1\in\SU
        \iff
        R_{n_1,n_2,m}\in\SU
    \]
    by applying Lemma~\ref{lm:Hsum}, then
    Lemma~\ref{lemma:Nkappa_Moebius2} with Corollary~\ref{cr:N_kappa_lambda_M3}, and finally
    Lemma~\ref{lm:Hsum} again.
\end{proof}

\begin{proof}[Proof of Theorem~\ref{th:B_P_Rnm_Nkappa}]
    Suppose that~$R_{n_1,n_2,m}(z)$ is not rational, since the case when it is rational reduces
    to Theorem~\ref{th:RnmNkappa}. In the proof of Theorem~\ref{th:RnmNkappa},
    for the case~$\widetilde q_\delta(z)\equiv 0$ identity~\eqref{eq:R_shifts_via_G0} yields
    \[
        B_{n_1,n_2,m}P_{r}(1/z)R_{n_1,n_2,m}(z)
        =
        \widetilde h(z)B_{n_1,n_2,m}P_{r}(1/z) \widetilde R_{0,1,1}(z)
        + B_{n_1,n_2,m}P_{r}(1/z)\widetilde g(z).
    \]
    The right-hand side is in~$\SU\subset\NU$ due to~$\widetilde h(z)B_{n_1,n_2,m}P_{r}(1/z)>0$
    for all~$z>1$ except for poles and zeros. In particular, the product
    \begin{equation*}%\label{eq:BPRsq0}
        f(z)\coloneqq B_{n_1,n_2,m}P_{r}\big(1/(z+\omega)\big)R_{n_1,n_2,m}(z+\omega)
    \end{equation*}
    lies in the class~$\NU$ whose definition does not depend on real shifts. According to the ultimate claim of
    Theorem~\ref{th:2F1zeros}, the function~$f(z)$ is real and has at most finitely many zeros and
    poles in~$z<1-\omega$. So, Corollary~\ref{cr:N_kappa_lambda_M3} implies that in
    fact~$f\in\SU$. The same conclusion~$f\in\SU$ holds
    for~$\widetilde q_\delta(z)\not\equiv 0$: it is enough to use~\eqref{eq:R_shifts_via_G2}
    and Lemma~\ref{lemma:Nkappa_Moebius} instead of~\eqref{eq:R_shifts_via_G0}.

    Now, note that
    \[
        g(z)\coloneqq
        \frac{R_{n_1,n_2,m}(z+\omega)}{B_{n_1,n_2,m}P_{r}\big(1/(z+\omega)\big)}
        =
        f(z)\Big[B_{n_1,n_2,m}P_{r}\Big(\frac1{z+\omega}\Big)\Big]^{-2},
    \]
    and therefore~$f\in\SU$ implies~$g\in\SU$ by Theorem~\ref{th:NkappaM2}.
\end{proof}

\section{Examples}\label{sec:examples}
Recall that $R_{n_1,n_2,m}(z)$ is defined in \eqref{eq:gen-ratio-def}. Below we apply our main theorems  to 15 specific triples $n_1,n_2,m$. The resulting integral representations are only valid if~$R_{n_1,n_2,m}(z)$ is well-behaved near~$z=1$ and its denominator~${}_2F_1(a,b;c;z)\ne0$ in the cut plane~$\C\setminus[1,+\infty)$ and on the banks of the branch cut. Conditions for the latter are given in Theorem~\ref{th:2F1zeros}, while the former in ensured by the inequality~$\nu>-1$ with $\nu$ defined in~\eqref{eq:asymp1}. To relax these restrictions, one needs a kind of regularization near the point~$z=1$, as well as near all zeros of the denominator. We plan to tackle these issues in part~II of our work.

\textbf{Example~1}.
\refstepcounter{dummy}\label{example:1}
For the Gauss ratio $R_{0,1,1}(z)$ according to \eqref{eq:nmrelated} we obtain $p=l=r=0$. Theorem~\ref{th:2F1identity} and definition \eqref{eq:B-defined} yield:
$$
B_{0,1,1}P_0(t)\equiv\frac{\Gamma(c)\Gamma(c+1)}{\Gamma(a)\Gamma(b+1)\Gamma(c-a+1)\Gamma(c-b)}.
$$
Next, using \eqref{eq:asymp-infnolog} and \eqref{eq:asymp-inflog} or directly it is easy to verify  that 
$$
Q_{a,b,c}=\lim\limits_{z\to\infty}R_{0,1,1}(z)=\Biggl\{\!\!\!\begin{array}{l}0,~~
b\leq{a}\\[5pt][c(b-a)]/[b(c-a)],~~b>a.\end{array}
$$
Then Theorem~\ref{th:2F1ratio-repr} with $N=0$ yields:
\begin{equation*}
R_{0,1,1}(z)=Q_{a,b,c}+\frac{\Gamma(c)\Gamma(c+1)}{\Gamma(a)\Gamma(b+1)\Gamma(c-b)\Gamma(c-a+1)}\int\limits_{0}^{1}\frac{t^{a+b-1}(1-t)^{c-a-b}dt}{(1-zt)|{}_2F_1(a,b;c;t^{-1})|^2}.
\end{equation*}
In order for this representation to hold we need to assume that any of the conditions \eqref{item:Run1}-\eqref{item:Run5} of Theorem~\ref{th:2F1zeros} is satisfied.  Under this restriction the condition $\nu>-1$ from Theorem~\ref{th:2F1ratio-repr} holds automatically since the parameter $q=m-n_1-n_2$ in Lemma~\ref{lm:ratio-asymp1} vanishes, so that $R_{0,1,1}(z)$ is integrable in the neighborhood of $1$.  We remark that the integrand is symmetric with respect to the interchange of $a$ and $b$ and the asymmetry of $R_{0,1,1}(z)$ is only reflected in the constants $Q_{a,b,c}$ and $B_{0,1,1}$.

The above integral representation was first found by V.\:Belevitch in \cite[formula~(72)]{Bel}
under the restrictions $0\le{a},{b}\le{c}$, $c\ge1$ (there is a small mistake in Belevitch's
paper - a superfluous $2$ in the denominator of the constant $Q_{a,b,c}$). Independently, using
the Gauss continued fraction \eqref{eq:gen_cont_fr} and Wall's theorem K\"{u}stner
\cite[Theorem~1.5]{Kuestner} proved that $R_{0,1,1}(z)$ is a generating function of a Hausdorff
moment sequence if $0<a\le{c+1}$, $0<b\le{c}$. As we mentioned in introduction, the coefficients
of the Gauss continued fraction \eqref{eq:gen_cont_fr} for $R_{0,1,1}(z)$ are all positive if
(a) $-1<a<0$ and either $-1<b<c<0$ or $0<c<b<c+1$ or (b) $0<a<c+1$, $c>0$ and $-1<b<c$. If these
conditions hold while conditions of Runckel's theorem~\ref{th:2F1zeros} are violated, then
Theorem~\ref{th:CF_conv}\eqref{item:Stieltjes_H} implies that representation
\eqref{eq:gen_cont_fr_int} is true while the above integral representation is not. Hence, in
this situation $R_{0,1,1}(z)$ has pole(s) in the interval $(0,1)$ which are reflected by the
atoms of the representing measure in \eqref{eq:gen_cont_fr_int}. This is the case, for instance,
if $0<c<a<c+1$ and $-1<b<0$.

For all~$\omega\le1$ and~$a,b,c\in\R$ such that~$-c\notin\N_0$, Theorem~\ref{th:B_P_Rnm_Nkappa}
implies that~$B_{0,1,1}P_0R_{0,1,1}(z+\omega)$ belongs to~$\SU$; the case~$\omega=0$ is
considered in detail in Theorem~\ref{th:R011Nkappa}.

\medskip

\textbf{Example~2}.
\refstepcounter{dummy}\label{example:2}
For the ratio $R_{0,1,0}(z)$ according to \eqref{eq:nmrelated} we obtain
$l=1$, $p=r=0$. Theorem~\ref{th:2F1identity} and definition \eqref{eq:B-defined}
yield:
$$
B_{0,1,0}P_0(t)\equiv\frac{[\Gamma(c)]^2}{\Gamma(a)\Gamma(b+1)\Gamma(c-a)\Gamma(c-b)}.
$$
Next, using \eqref{eq:asymp-infnolog} and \eqref{eq:asymp-inflog} or directly we can  verify that
$$
Q_{a,b}=\lim\limits_{z\to\infty}R_{0,1,0}(z)=\Biggl\{\!\!\!\begin{array}{l}0,~~
b\leq{a}\\[3pt](b-a)/b,~~b>a.\end{array}
$$
Then Theorem~\ref{th:2F1ratio-repr} with $N=0$ yields:
$$
R_{0,1,0}(z)=Q_{a,b}+\frac{[\Gamma(c)]^2}{\Gamma(a)\Gamma(b+1)\Gamma(c-a)\Gamma(c-b)}
\int_0^1\frac{t^{a+b-1}(1-t)^{c-a-b-1}dt}{(1-zt)|{}_2F_1(a,b;c;1/t)|^2}.
$$
Note that similarly to Example~\hyperref[example:1]{1}, the integrand is symmetric with respect
to the interchange of $a$ and $b$ and the asymmetry of the left hand side is only reflected in
the constants. In order for this representation to hold, we need to assume that any of the
conditions \eqref{item:Run1}-\eqref{item:Run5} of Theorem~\ref{th:2F1zeros} is satisfied. Under this restriction and except
for the degenerate cases $ab=0$, $(c-a)(c-b)=0$ the condition $\nu>-1$ from \eqref{eq:asymp1}
reads
$$
(c-a-b-1)_{-}-(c-a-b)_{-}>-1,
$$
which is easily seen to be equivalent to $c>a+b$.  The above set of conditions holds, for example, if $-1<a<0$ and  $0<b<c$ or $a>0$ and $-1<b<c-a$.

Using continued fractions K\"{u}stner \cite[Theorem~1.5]{Kuestner} proved that  $R_{0,1,0}(z)$ is a generating function of a Hausdorff moment sequence if $-1\le{b}\le{c}$ and $0<a\le{c}$.
Askitis \cite[Lemma~6.2.2]{Askitis} found another proof for the this claim (without use of continued fractions). We also remark that the continued fraction for $R_{0,1,0}$ was also found by Gauss, see~\cite[eq.~26]{Gauss} or~\cite[eq.~(2.7)]{Kuestner}, in the form
$$
\cfrac{1}{1-\cfrac{\alpha_1z}{1-\cfrac{\alpha_2z}{1-\raisebox{-.4em}{$\ddots$}}}},
$$
where $\alpha_1=a/c$, and for $k\ge1$
$$
\alpha_{2k}=\frac{(b+k)(c-a+k-1)}{(c+2k-2)(c+2k-1)},
\quad
\alpha_{2k+1}=\frac{(a+k)(c-b+k-1)}{(c+2k-1)(c+2k)}.
$$
From these formulae, it is not difficult to formulate sufficient conditions for $\alpha_n\ge0$
ensuring that $R_{0,1,0}\in\mathcal{S}$ (the Stieltjes class). For general values of~$a,b,c\in\R$, ~$-c\notin\N_0$,
Theorem~\ref{th:B_P_Rnm_Nkappa} yields that~$B_{0,1,0}P_0R_{0,1,0}(z+\omega)$ belongs to~$\SU$
whenever~$\omega\le 1$. In particular,~$B_{0,1,0}P_0R_{0,1,0}\in\NN_{\kappa}^\lambda$ for
non-rational~$R_{0,1,0}(z)$ where the indices~$\kappa,\lambda$ are computed in
Theorem~\ref{th:N_kl_CFraction}; the case of rational~$\pm R_{0,1,0}$ is treated in
Corollary~\ref{cr:SFracRational}.

\bigskip

\textbf{Example~3}.
\refstepcounter{dummy}\label{example:3}
For the ratio $R_{1,1,1}(z)$ according to \eqref{eq:nmrelated} we obtain
$l=1$, $p=r=0$. Theorem~\ref{th:2F1identity} and definition \eqref{eq:B-defined}
yield:
$$
B_{1,1,1}P_0(t)=\frac{\Gamma(c)\Gamma(c+1)}{\Gamma(a+1)\Gamma(b+1)\Gamma(c-a)\Gamma(c-b)}
$$
Next, it is easy to verify using \eqref{eq:asymp-infnolog} and \eqref{eq:asymp-inflog} or directly that 
$$
Q_{a,b,c}=\lim\limits_{z\to\infty}R_{1,1,1}(z)=0.
$$
Then according to the case $N=0$ of Theorem~\ref{th:2F1ratio-repr} we obtain:
$$
R_{1,1,1}(z)=\frac{\Gamma(c)\Gamma(c+1)}{\Gamma(a+1)\Gamma(b+1)\Gamma(c-a)\Gamma(c-b)}
\int_0^1\frac{t^{a+b}(1-t)^{c-a-b-1}dt}{(1-zt)|{}_2F_1(a,b;c;1/t)|^2}.
$$
In order for this representation to hold we need to assume that any of the conditions \eqref{item:Run1}-\eqref{item:Run5}
of Theorem~\ref{th:2F1zeros} is satisfied. Under this restriction and except for the degenerate
cases $ab=0$, ~$(c-a)(c-b)=0$ the condition $\nu>-1$ from \eqref{eq:asymp1} reads
$$
(c-a-b-1)_{-}-(c-a-b)_{-}>-1,
$$
which is easily seen to be equivalent to $c>a+b$. All these conditions are satisfied, for example, if (a) $-1<a<0$ and $0<b<c$ or (b) $0<a<c$ and $-1<b<c-a$.

In this case, the corresponding C-fraction is not as simple as in the previous two examples. Instead of trying to
write it directly, let us employ the following contiguous relation (\cite[eq.~17]{Gauss}
or~\cite[eq.~(2.2)]{Kuestner}):
\begin{equation*}\label{eq:R111_via_R010}
    R_{1,1,1}(z)
    =
    \frac{{}_2F_1(a+1,b+1;c+1;z)}{{}_2F_1(a,b;c;z)}
    =
    \frac{c}{az}\left(\frac{{}_2F_1(a,b+1;c;z)}{{}_2F_1(a,b;c;z)}-1\right)
    =
    \frac{c}{az}\left(R_{0,1,0}(z)-1\right)
    .
\end{equation*}
Since~$R_{1,1,1}(z)$ is analytic near the origin and only has finitely many poles for~$z<1$, we
obtain
\[
    \frac{a}c R_{1,1,1}\in\NN_{\kappa}^{\lambda} \iff R_{0,1,0}\in\NN_{\lambda}^{\kappa+\delta}
    \quad\text{and}\quad
    -\frac{a}c R_{1,1,1}\in\NN_{\kappa}^{\lambda} \iff -R_{0,1,0}\in\NN_{\lambda}^{\kappa+\delta},
\]
where~$\delta\in\{-1,0,1\}$ is a certain number depending on the asymptotics at infinity.%
\footnote{By Theorem~\ref{th:DLLS}, ~$\frac{a}c R_{1,1,1}\in\NN_\kappa$ implies
    that~$\frac{a}c z^2 R_{1,1,1}(z)$ lies
    in~$\NN_{\kappa-1}\cup\NN_{\kappa}\cup\NN_{\kappa+1}$. Lemma~\ref{lm:Hsum} shows how the
    number of negative eigenvalues of the Pick matrix changes when a function is increased
    by~$z$ or by~$1/z$. In particular,~$zR_{0,1,0}(z)=\frac{a}c z^2 R_{1,1,1}(z)+z$ belongs
    to~$\NN_{\kappa-1}\cup\NN_{\kappa}\cup\NN_{\kappa+1}$. To obtain the exact value of~$\delta$
    and the converse implication, one can check what happens with the
    \emph{generalized pole of nonpositive type} of~$R_{1,1,1}(z)$ at infinity (if any), since
    its generalized poles at finite points remain unaffected, see the definition and further
    details in~\cite{DHdS,DLLS}.} %
Furthermore, Theorem~\ref{th:B_P_Rnm_Nkappa} yields that~$B_{1,1,1}P_0R_{1,1,1}(z+\omega)$
belongs to~$\SU$ whenever~$\omega\le 1$.

In particular, for $-1\le{a}\le{c}$ and $0<b\le{c}$ K\"{u}stner proved
\cite[Theorem~1.5]{Kuestner} that the ratio~$R_{1,1,1}(z)$ is the generating function of a
Hausdorff moment sequence, which is equivalent to~$z\mapsto R_{1,1,1}(z+\omega)\in\mathcal S=\NN_0^0$ for each~$\omega\le 1$.

\bigskip

\textbf{Example~4}.  For the ratio $R_{1,1,2}(z)$ according to \eqref{eq:nmrelated} we obtain $l=p=r=0$.  Theorem~\ref{th:2F1identity} and definition \eqref{eq:B-defined} yield: 
$$
B_{1,1,2}P_0(t)=B_{1,1,2}P_0=\frac{\Gamma(c+1)\Gamma(c+2)}{\Gamma(a+1)\Gamma(b+1)\Gamma(c-a+1)\Gamma(c-b+1)}
.
$$
Next, it is easy to verify using \eqref{eq:asymp-infnolog} and \eqref{eq:asymp-inflog} or directly that 
$$
Q_{a,b,c}=\lim\limits_{z\to\infty}R_{1,1,2}(z)=0.
$$
Then according to the case $N=0$ of Theorem~\ref{th:2F1ratio-repr} we obtain:
$$
R_{1,1,2}(z)=\frac{\Gamma(c+1)\Gamma(c+2)}{\Gamma(a+1)\Gamma(b+1)\Gamma(c-a+1)\Gamma(c-b+1)}
\int_0^1\frac{t^{a+b}(1-t)^{c-a-b}dt}{(1-zt)|{}_2F_1(a,b;c;1/t)|^2}.
$$
In order for this representation to hold we need to assume that any of the conditions \eqref{item:Run1}-\eqref{item:Run5} of Theorem~\ref{th:2F1zeros} is satisfied. 
Under this restriction and except for the degenerate cases $ab=0$, $(c-a)(c-b)=0$ the condition $\nu>-1$ from \eqref{eq:asymp1} reads 
$$
(c-a-b)_{-}-(c-a-b)_{-}>-1
$$
and is trivially satisfied.  

For general values of~$a,b,c\in\R$ such that~$-c\notin\N_0$, we again employ a contiguous
relation (see~\cite[eq.~19]{Gauss} or~\cite[p.~122 eq.~(2)]{Perron}) instead of looking for the
corresponding C-fraction:
\begin{equation}\label{eq:R112_via_R011}
    R_{1,1,2}(z)
    =
    \frac{{}_2F_1(a+1,b+1;c+2;z)}{{}_2F_1(a,b;c;z)}
    =
    \frac{c(c+1)}{a(c-b)z}\left(\frac{{}_2F_1(a,b+1;c+1;z)}{{}_2F_1(a,b;c;z)}-1\right).
\end{equation}
Since~$R_{1,1,2}(z)$ is analytic near the origin and only has finitely many poles for~$z<1$, we
obtain
\[
    \frac{a(c-b)}{c(c+1)} R_{1,1,2}\in\NN_{\kappa}^{\lambda}
    \iff
    R_{0,1,1}\in\NN_{\lambda}^{\kappa+\delta}
    \quad\text{and}\quad
    -\frac{a(c-b)}{c(c+1)} R_{1,1,2}\in\NN_{\kappa}^{\lambda}
    \iff
    -R_{0,1,1}\in\NN_{\lambda}^{\kappa+\delta},
\]
where~$\delta\in\{-1,0,1\}$ is a certain number depending on the asymptotics at infinity, cf.
Example~\hyperref[example:3]{3}. This relation for~$a(c-b)\ne 0$ refines what
Theorem~\ref{th:RnmNkappa} states:
\[
    B_{1,1,2}P_0\ge0 \implies R_{1,1,2}\in\SU,
    \quad
    B_{1,1,2}P_0\le0 \implies -R_{1,1,2}\in\SU
    .
\]
However, the latter also holds for~$R_{1,1,2}(z+\omega)$ for all $\omega\le1$ due to Theorem~\ref{th:B_P_Rnm_Nkappa}.

\bigskip

\textbf{Example~5}.
\refstepcounter{dummy}\label{example:5}
For the ratio $R_{0,2,2}(z)$ according to \eqref{eq:nmrelated} we obtain $l=p=0$, $r=1$.  Theorem~\ref{th:2F1identity} and definition \eqref{eq:B-defined} yield:
$$
B_{0,2,2}P_1(t)=\frac{\Gamma(c)\Gamma(c+2)(ct+b-a+1)}{\Gamma(a)\Gamma(b+2)\Gamma(c-a+2)\Gamma(c-b+2)}.
$$
Next, it is easy to verify using \eqref{eq:asymp-infnolog} and \eqref{eq:asymp-inflog} or directly that 
$$
Q_{a,b,c}=\lim\limits_{z\to\infty}R_{0,2,2}(z)=\Biggl\{\!\!\!\begin{array}{l}0,~~
b\leq{a}\\[3pt]c(c+1)(b-a)(b-a+1)/[b(b+1)(c-a)(c-a+1)],~~b>a.\end{array}
$$
Then according to the case $N=0$ of  Theorem~\ref{th:2F1ratio-repr} we obtain:
$$
R_{0,2,2}(z)=Q_{a,b,c}+\frac{\Gamma(c)\Gamma(c+2)}{\Gamma(a)\Gamma(b+2)\Gamma(c-a+2)\Gamma(c-b)}
\int_0^1\frac{t^{a+b-1}(ct+b-a+1)(1-t)^{c-a-b}dt}{(1-zt)|{}_2F_1(a,b;c;1/t)|^2}.
$$
In order for this representation to hold we need to assume that any of the conditions \eqref{item:Run1}-\eqref{item:Run5}
of Theorem~\ref{th:2F1zeros} is satisfied. Under this restriction and except for the degenerate
cases $ab=0$, $(c-a)(c-b)=0$ the condition $\nu>-1$ from \eqref{eq:asymp1} reads
$$
(c-a-b)_{-}-(c-a-b)_{-}>-1
$$
and is trivially satisfied.  

Further, $B_{0,2,2}P_1(t)$ does not change sign on $(0,1)$ if the zero $t_*=(a-b-1)/c$ of the
polynomial $P_1(t)=ct+b-a+1$ does not lie in $(0,1)$, which is the case if $a\le{b+1}$ or
$a\ge{c+b+1}$. If this case, according to Theorem~\ref{th:RnmNkappa} $R_{0,2,2}\in\SU$ if
$B_{0,2,2}P_1(t)\ge0$ on $(0,1)$ or $-R_{0,2,2}\in\SU$ if $B_{0,2,2}P_1(t)\le0$ on $(0,1)$.
Furthermore, according to Theorem~\ref{th:B_P_Rnm_Nkappa} the functions
\[
    B_{0,2,2}P_1\Big(\frac1{z+\omega}\Big)R_{0,2,2}(z+\omega)
    \quad\text{and}\quad
    \frac{R_{0,2,2}(z+\omega)}{B_{0,2,2}P_1\big(1/(z+\omega)\big)}
\]
lie in~$\SU$ for all real values of parameters  $a,b,c,\omega$ such that~$-c\notin\N_0$ and $\omega\le1$.

\medskip

\textbf{Example~6}. 
For the ratio $R_{0,2,0}(z)$ according to \eqref{eq:nmrelated} we obtain $p=0$, $l=2$, $r=1$. Theorem~\ref{th:2F1identity} and definition \eqref{eq:B-defined} yield: 
$$
B_{0,2,0}P_1(t)=\frac{[\Gamma(c)]^2(t(c-2b-2)+b+1-a)}{\Gamma(a)\Gamma(b+2)\Gamma(c-a)\Gamma(c-b)}.
$$
Next, it is easy to verify using \eqref{eq:asymp-infnolog} and \eqref{eq:asymp-inflog} or directly that 
$$
Q_{a,b}=\lim\limits_{z\to\infty}R_{0,2,0}(z)=\Biggl\{\!\!\!\begin{array}{l}0,~~
b\leq{a}\\[3pt](b-a)(b-a+1)/[b(b+1)],~~b>a.\end{array}
$$
Then according to the case $N=0$ of Theorem~\ref{th:2F1ratio-repr} we obtain:
$$
R_{0,2,0}(z)=Q_{a,b}+
\frac{[\Gamma(c)]^2}{\Gamma(a)\Gamma(b+2)\Gamma(c-a)\Gamma(c-b)}
\!\int\limits_{0}^{1}\!\frac{t^{a+b-1}(b-a+1+t(c-2b-2))(1-t)^{c-a-b-2}dt}{(1-zt)|{}_2F_1(a,b;c;1/t)|^2}.
$$
In order for this representation to hold we need to assume that any of the conditions \eqref{item:Run1}-\eqref{item:Run5}
of Theorem~\ref{th:2F1zeros} is satisfied. Under this restriction and except for the degenerate
cases $ab=0$, $(c-a)(c-b)=0$ the condition $\nu>-1$ from \eqref{eq:asymp1} reads
$$
(c-a-b-2)_{-}-(c-a-b)_{-}>-1
$$
which is easily seen to be equivalent to $c>a+b+1$.

Further, $B_{0,2,0}P_1(t)$ does not change sign on $(0,1)$ if the zero $t_*=(a-b-1)/(c-2b-2)$ of the polynomial $P_1(t)=t(2b+2-c)+a-b-1$ does not lie in $(0,1)$, which is the case if $a\le{b+1}$ or
$a+b+1\ge{c}$. If this case, according to Theorem~\ref{th:RnmNkappa} $R_{0,2,0}\in\SU$ if
$B_{0,2,0}P_1(t)\ge0$ on $(0,1)$ or $-R_{0,2,0}\in\SU$ if $B_{0,2,0}P_1(t)\le0$ on $(0,1)$.
Furthermore, according to Theorem~\ref{th:B_P_Rnm_Nkappa} the functions
\[
    B_{0,2,0}P_1\Big(\frac1{z+\omega}\Big)R_{0,2,0}(z+\omega)
    \quad\text{and}\quad
    \frac{R_{0,2,0}(z+\omega)}{B_{0,2,0}P_1\big(1/(z+\omega)\big)}
\]
lie in~$\SU$ for all real values of parameters  $a,b,c,\omega$ such that~$-c\notin\N_0$ and $\omega\le1$.

\medskip

\textbf{Example~7}. For the ratio $R_{1,1,0}(z)$ according to \eqref{eq:nmrelated} we obtain $p=0$, $l=2$, $r=0$.   Theorem~\ref{th:2F1identity} and definition \eqref{eq:B-defined} yield: $$
B_{1,1,0}P_0(t)=-\frac{[\Gamma(c)]^2(c-a-b-1)}{\Gamma(a+1)\Gamma(b+1)\Gamma(c-a)\Gamma(c-b)}
.
$$
Next, it is easy to verify using \eqref{eq:asymp-infnolog} and \eqref{eq:asymp-inflog} or directly that 
$$
Q_{a,b,c}=\lim\limits_{z\to\infty}R_{1,1,0}(z)=0.
$$
Then according to the case $N=0$ of Theorem~\ref{th:2F1ratio-repr} we obtain:
$$
R_{1,1,0}(z)=-\frac{[\Gamma(c)]^2(c-a-b-1)}{\Gamma(a+1)\Gamma(b+1)\Gamma(c-a)\Gamma(c-b)}
\int_0^1\frac{t^{a+b}(1-t)^{c-a-b-2}dt}{(1-zt)|{}_2F_1(a,b;c;1/t)|^2}.
$$
In order for this representation to hold we need to assume that any of the conditions \eqref{item:Run1}-\eqref{item:Run5} of Theorem~\ref{th:2F1zeros} is satisfied. Under this restriction and except for the degenerate cases $ab=0$, $(c-a)(c-b)=0$ the condition $\nu>-1$ from \eqref{eq:asymp1} reads
$$
(c-a-b-2)_{-}-(c-a-b)_{-}>-1
$$
which is easily seen to be equivalent to $c>a+b+1$.

Further, $B_{1,1,0}P_0(t)=B_{1,1,0}P_0$ does not depend on $t$, and thus does not change sign on $(0,1)$.  Hence, according to Theorem~\ref{th:RnmNkappa} $R_{1,1,0}\in\SU$ if
$B_{1,1,0}P_0(t)\ge0$ on $(0,1)$ or $-R_{1,1,0}\in\SU$ if $B_{1,1,0}P_0(t)\le0$ on $(0,1)$.
Furthermore, according to Theorem~\ref{th:B_P_Rnm_Nkappa} the function
$B_{1,1,0}P_0R_{1,1,0}(z+\omega)$ lies in $\SU$ for all real values of parameters  $a,b,c,\omega$  such that~$-c\notin\N_0$ and $\omega\le1$.

\medskip

\textbf{Example~8}.
\refstepcounter{dummy}\label{example:8}
For the ratio $R_{0,0,1}(z)$ according to \eqref{eq:nmrelated} we obtain $p=1$, $l=r=0$.   Theorem~\ref{th:2F1identity} and definition \eqref{eq:B-defined} yield: $$
B_{0,0,1}P_0(t)=-\frac{\Gamma(c)\Gamma(c+1)}{\Gamma(a)\Gamma(b)\Gamma(c-a+1)\Gamma(c-b+1)}.
$$
Next, it is easy to verify using \eqref{eq:asymp-infnolog} and \eqref{eq:asymp-inflog} or directly that 
$$
Q_{a,b,c}=\lim\limits_{z\to\infty}R_{0,0,1}(z)=\Biggl\{\!\!\!\begin{array}{l}c/(c-b),~~
b\leq{a}\\[3pt]c/(c-a),~~b>a.\end{array}
$$
Then the case $N=0$ of Theorem~\ref{th:2F1ratio-repr} leads to the representation:
$$
R_{0,0,1}(z)=Q_{a,b,c}-\frac{\Gamma(c)\Gamma(c+1)}{\Gamma(a)\Gamma(b)\Gamma(c-a+1)\Gamma(c-b+1)}
\int_0^1\frac{t^{a+b-1}(1-t)^{c-a-b}dt}{(1-zt)|{}_2F_1(a,b;c;1/t)|^2}.
$$
In order for this representation to hold we need to assume that any of the conditions \eqref{item:Run1}-\eqref{item:Run5} of Theorem~\ref{th:2F1zeros} is satisfied. Under this restriction and except for the degenerate cases $ab=0$, $(c-a)(c-b)=0$ the condition $\nu>-1$ from \eqref{eq:asymp1} reads
$$
(c-a-b+1)_{-}-(c-a-b)_{-}>-1
$$
which is easily seen to be satisfied for all real $a,b,c$.

For general values of~$a,b,c\in\R$ such that~$-c\notin\N_0$ and~$b\ne0$, we can use the contiguous
relation~\cite[eq.~12]{Gauss} instead of calculating the corresponding C-fraction directly:
\[
    \frac{c-b}{b}R_{0,0,1}(z)
    =
    \frac{c-b}{b}\,\frac{{}_2F_1(a,b;c+1;z)}{{}_2F_1(a,b;c;z)}
    =
    \frac{c}{b}-\frac{{}_2F_1(a,b+1;c+1;z)}{{}_2F_1(a,b;c;z)}
    =
    \frac{c}{b}-R_{0,1,1}(z).
\]
Since~$R_{0,1,1}(z)$ is analytic near the origin and only has finitely many poles for~$z<1$, we
obtain
\[
    -\frac{c-b}{b} R_{0,0,1}\in\NN_{\kappa}^{\lambda}
    \iff
    R_{0,1,1}\in\NN_{\kappa}^{\lambda+\delta}
    \quad\text{and}\quad
    \frac{c-b}{b} R_{0,0,1}\in\NN_{\kappa}^{\lambda}
    \iff
    -R_{0,1,1}\in\NN_{\kappa}^{\lambda+\delta}
\]
for certain~$\delta\in\{-1,0,1\}$ depending on~$c/b$ and the asymptotics of~$R_{0,0,1}(z)$ at
infinity, cf. Example~\hyperref[example:3]{3}. This relation for~$b\ne 0$ refines what
Theorem~\ref{th:RnmNkappa} states:
\[
    B_{0,0,1}P_0\ge0 \implies R_{0,0,1}\in\SU,
    \quad
    B_{0,0,1}P_0\le0 \implies -R_{0,0,1}\in\SU
    .
\]
However, the latter also holds for~$R_{0,0,1}(z+\omega)$ due to Theorem~\ref{th:B_P_Rnm_Nkappa}.

\medskip

\textbf{Example~9}. For the ratio $R_{0,0,-1}(z)$ according to \eqref{eq:nmrelated} we obtain $l=1$, $p=r=0$.  Theorem~\ref{th:2F1identity} and definition \eqref{eq:B-defined} then yield: 
$$
B_{0,0,-1}P_0(t)=\frac{\Gamma(c)\Gamma(c-1)}{\Gamma(a)\Gamma(b)\Gamma(c-a)\Gamma(c-b)}
$$
Next, it is easy to verify using \eqref{eq:asymp-infnolog} and \eqref{eq:asymp-inflog} or directly that 
$$
Q_{a,b,c}=\lim\limits_{z\to\infty}R_{0,0,-1}(z)=\Biggl\{\!\!\!\begin{array}{l}(c-b-1)/(c-1),~~
b\leq{a}\\[3pt](c-a-1)/(c-1),~~b>a.\end{array}
$$
Then the case $N=0$ of Theorem~\ref{th:2F1ratio-repr} leads to the representation:
$$
R_{0,0,-1}(z)=Q_{a,b,c}+\frac{\Gamma(c)\Gamma(c-1)}{\Gamma(a)\Gamma(b)\Gamma(c-a)\Gamma(c-b)}
\int_0^1\frac{t^{a+b-1}(1-t)^{c-a-b-1}dt}{(1-zt)|{}_2F_1(a,b;c;1/t)|^2}.
$$
In order for this representation to hold we need to assume that any of the conditions \eqref{item:Run1}-\eqref{item:Run5} of Theorem~\ref{th:2F1zeros} is satisfied. Under this restriction and except for the degenerate cases $ab=0$, $(c-a)(c-b)=0$ the condition $\nu>-1$ from \eqref{eq:asymp1} reads
$$
(c-a-b-1)_{-}-(c-a-b)_{-}>-1
$$
which is easily seen to be equivalent to $c>a+b$.  All these conditions are satisfied, for example, if (a) $-1<a<0$ and $0<b<c-a$ or (b) $0<a<c$ and $-1<b<c-a$.

For general values of~$a,b,c\in\R$ such that~$-c\notin\N_0$, one may express~$R_{0,0,-1}(z)$
via~$R_{0,1,0}(z)$ using the contiguous relation~\cite[eq.~12]{Gauss}:
\[
    \frac{c-1}b R_{0,0,-1}(z)=\frac{c-b-1}b + R_{0,1,0}(z)
\]
and further apply Example~\hyperref[example:2]{2}. Another option is to relate to the case
considered in Example~\hyperref[example:8]{8}:
\[
    \frac{1}{R_{0,0,-1}(z)}
    =
    \frac{{}_2F_1(a,b;c;z)}{{}_2F_1(a,b;c-1;z)}
    \eqqcolon \widetilde R_{0,0,1}(z)
    ,
\]
here the right-hand side is~$R_{0,0,1}(z)$ with~$c$ decreased by~$1$. Therefore, from
Lemma~\ref{lemma:Nkappa_Moebius2} and Theorem~\ref{th:DLLS} we get
\[
    - R_{0,0,-1}\in\NN_{\kappa}^{\lambda}
    \iff
    \widetilde R_{0,0,1}\in\NN_{\kappa}^{\lambda+\delta}
    \quad\text{and}\quad
    R_{0,0,-1}\in\NN_{\kappa}^{\lambda}
    \iff
    -\widetilde R_{0,0,1}\in\NN_{\kappa}^{\lambda+\delta}
\]
for certain~$\delta\in\{-1,0,1\}$ depending on the asymptotics at infinity, cf.
Example~\hyperref[example:3]{3}. By Theorem~\ref{th:B_P_Rnm_Nkappa}, $B_{0,0,-1}P_0\ge0$ implies
that~$R_{0,0,-1}(z+\omega)$ is in~$\SU$, and the opposite inequality~$B_{0,0,-1}P_0\le0$ implies
that $-R_{0,0,-1}(z+\omega)$ is in~$\SU$ for all $\omega\le1$.

\bigskip

\textbf{Example~10}. For the ratio $R_{0,0,2}(z)$ according to \eqref{eq:nmrelated} we obtain $p=2$, $l=0$, $r=1$.  Application of Theorem~\ref{th:2F1identity} and definition \eqref{eq:B-defined} yield: $$
B_{0,0,2}P_1(t)=\frac{\Gamma(c)\Gamma(c+2)[ct+a+b-2c-1]}{\Gamma(a)\Gamma(b)\Gamma(c-a+2)\Gamma(c-b+2)}.
$$
Next, it is easy to verify using \eqref{eq:asymp-infnolog} and \eqref{eq:asymp-inflog} or directly that 
$$
Q_{a,b,c}=\lim\limits_{z\to\infty}R_{0,0,2}(z)=\Biggl\{\!\!\!\begin{array}{l}c(c+1)/[(c-b)(c-b+1)],~~
b\leq{a}\\[3pt]c(c+1)/[(c-a)(c-a+1)],~~b>a.\end{array}
$$
Then the case $N=0$ of Theorem~\ref{th:2F1ratio-repr} leads to the representation:
$$
R_{0,0,2}(z)=Q_{a,b,c}+\frac{\Gamma(c)\Gamma(c+2)}{\Gamma(a)\Gamma(b)\Gamma(c-a+2)\Gamma(c-b+2)}
\int_0^1\frac{t^{a+b-1}(ct+a+b-2c-1)(1-t)^{c-a-b}dt}{(1-zt)|{}_2F_1(a,b;c;1/t)|^2}.
$$
In order for this representation to hold we need to assume that any of the conditions \eqref{item:Run1}-\eqref{item:Run5} of Theorem~\ref{th:2F1zeros} is satisfied. Under this restriction and except for the degenerate cases $ab=0$, $(c-a)(c-b)=0$ the condition $\nu>-1$ from \eqref{eq:asymp1} reads
$$
(c-a-b+2)_{-}-(c-a-b)_{-}>-1
$$
which is true for all real $a,b,c$.

Further, $B_{0,0,2}P_1(t)$ does not change sign on $(0,1)$ if the zero $t_*=(2c+1-a-b)/c$ of the polynomial $P_1(t)=2c+1-a-b-ct$ does not lie in $(0,1)$, which is the case if $c>0$ and either $2c\le{a+b-1}$ or $c\ge{a+b-1}$;
or $c<0$ and either $2c\ge{a+b-1}$ or $c\ge{a+b-1}$. If this case, according to Theorem~\ref{th:RnmNkappa} $R_{0,0,2}\in\SU$ if
$B_{0,0,2}P_1(t)\ge0$ on $(0,1)$ or $-R_{0,0,2}\in\SU$ if $B_{0,0,2}P_1(t)\le0$ on $(0,1)$.
Furthermore, according to Theorem~\ref{th:B_P_Rnm_Nkappa} the functions
\[
    B_{0,0,2}P_1\Big(\frac1{z+\omega}\Big)R_{0,0,2}(z+\omega)
    \quad\text{and}\quad
    \frac{R_{0,0,2}(z+\omega)}{B_{0,0,2}P_1\big(1/(z+\omega)\big)}
\]
lie in~$\SU$ for all real values of parameters  $a,b,c,\omega$ such that~$-c\notin\N_0$ and $\omega\le1$.

\medskip

\textbf{Example~11}.   For the ratio $R_{0,1,2}(z)$ according to \eqref{eq:nmrelated} we obtain $p=1$, $l=0$, $r=1$.   Theorem~\ref{th:2F1identity} and definition \eqref{eq:B-defined} yield: $$
B_{0,1,2}P_1(t)=-\frac{\Gamma(c)\Gamma(c+2)(ct+b-c)}{\Gamma(a)\Gamma(b+1)\Gamma(c-a+2)\Gamma(c-b+1)}.
$$
Next, it is easy to verify using \eqref{eq:asymp-infnolog} and \eqref{eq:asymp-inflog} or directly that 
$$
Q_{a,b,c}=\lim\limits_{z\to\infty}R_{0,1,2}(z)=\Biggl\{\!\!\!\begin{array}{l}0,~~
b\leq{a}\\[3pt]c(c+1)(b-a)/[b(c-a)(c-a+1)],~~b>a.\end{array}
$$
Then the case $N=0$ of  Theorem~\ref{th:2F1ratio-repr} leads to the representation:
$$
R_{0,1,2}(z)=Q_{a,b,c}-\frac{\Gamma(c)\Gamma(c+2)}{\Gamma(a)\Gamma(b+1)\Gamma(c-a+2)\Gamma(c-b+1)}
\int_0^1\frac{t^{a+b-1}(ct+b-c)(1-t)^{c-a-b}dt}{(1-zt)|{}_2F_1(a,b;c;1/t)|^2}.
$$
In order for this representation to hold we need to assume that any of the conditions \eqref{item:Run1}-\eqref{item:Run5} of Theorem~\ref{th:2F1zeros} is satisfied. Under this restriction and except for the degenerate cases $ab=0$, $(c-a)(c-b)=0$ the condition $\nu>-1$ from \eqref{eq:asymp1} reads
$$
(c-a-b+1)_{-}-(c-a-b)_{-}>-1
$$
which is true for all real $a,b,c$.

Further, $B_{0,1,2}P_1(t)$ does not change sign on $(0,1)$ if the zero $t_*=(c-b)/c$ of the polynomial $P_1(t)=(ct+b-c)/b$ does not lie in $(0,1)$, i.e. $(c-b)/c\in(-\infty,0]\cup[1,\infty)$.  If this case, according to Theorem~\ref{th:RnmNkappa} $R_{0,1,2}\in\SU$ if
$B_{0,1,2}P_1(t)\ge0$ on $(0,1)$ or $-R_{0,1,2}\in\SU$ if $B_{0,1,2}P_1(t)\le0$ on $(0,1)$.
Furthermore, according to Theorem~\ref{th:B_P_Rnm_Nkappa} the functions
\[
    B_{0,1,2}P_1\Big(\frac1{z+\omega}\Big)R_{0,1,2}(z+\omega)
    \quad\text{and}\quad
   \frac{R_{0,1,2}(z+\omega)}{B_{0,1,2}P_1\big(1/(z+\omega)\big)}
\]
lie in~$\SU$ for all real values of parameters  $a,b,c,\omega$ such that~$-c\notin\N_0$ and $\omega\le1$.

\medskip

\textbf{Example~12}.   For the ratio $R_{0,-1,0}(z)$ according to \eqref{eq:nmrelated} we obtain $l=1$, $p=r=0$.  Theorem~\ref{th:2F1identity} and definition \eqref{eq:B-defined} yield: 
$$
B_{0,-1,0}P_0(t)=-\frac{[\Gamma(c)]^2}{\Gamma(a)\Gamma(b)\Gamma(c-a)\Gamma(c-b+1)}.
$$
Using Lemmas~\ref{lm:asymp-infnolog} and~\ref{lm:asymp-inflog} or by direct, albeit tedious calculation, we get the asymptotic approximations

\smallskip
\begin{compactenum}[\indent(1)]\itemsep1.5pt plus .7pt minus .2pt

\item if $b+1<a$, then $R_{0,-1,0}(z)=Az+B+o(1)~\text{as}~z\to\infty$;

\item if $b<a\le{b+1}$, then $R_{0,-1,0}(z)=Az+o(z)~\text{as}~z\to\infty$;

\item if $b-1\le{a}\le{b}$, then $R_{0,-1,0}(z)=o(z)~\text{as}~z\to\infty$;

\item if $a<b-1$, then  $R_{0,-1,0}(z)=C+o(1)~\text{as}~z\to\infty$,

\end{compactenum}
where
$$
A=\frac{b-a}{c-b},~~B=\frac{b(b+1)-2ab+c(a-1)}{(c-b)(a-b-1)},~~~C=\frac{b-1}{b-a-1}.
$$

Hence, if $|a-b|>1$, we have $R_{0,-1,0}(z)=\beta{z}+\alpha+o(1)$ as $z\to\infty$,  with $(\beta,\alpha)=(A,B)$ if $a>b+1$ and $(\beta,\alpha)=(0,C)$ if  $a<b-1$.
Then for $|a-b|>1$ we can choose $N=0$ in Theorem~\ref{th:2F1ratio-repr} leading to the representation:
\begin{equation}\label{eq:example12}
R_{0,-1,0}(z)=\alpha+\beta{z}-\frac{[\Gamma(c)]^2}{\Gamma(a)\Gamma(b)\Gamma(c-a)\Gamma(c-b+1)}
\int_0^{1}\frac{t^{a+b-2}(1-t)^{c-a-b}}{(1-zt)|{}_{2}F_{1}(a,b;c;1/t)|^{2}}dt.
\end{equation}
In addition to the condition $|a-b|>1$ we need to assume any of the conditions \eqref{item:Run1}-\eqref{item:Run5} of Theorem~\ref{th:2F1zeros}. Under these restrictions and except for the degenerate cases $ab=0$, $(c-a)(c-b)=0$ the condition $\nu>-1$ from \eqref{eq:asymp1} reads
$$
(c-a-b+1)_{-}-(c-a-b)_{-}>-1
$$
which is true for all real $a,b,c$.

For arbitrary $a,b$ we obtain $R_{0,-1,0}(z)=\beta{z}+o(z)$ as~$z\to\infty$, with $\beta=A$ if $b<a$ and $\beta=0$ if  ${a}\le{b}$. Hence, we can remove the restriction $|a-b|>1$ by taking $N=1$ in Theorem~\ref{th:2F1ratio-repr} which leads to 
\begin{equation}\label{eq:example12-1}
R_{0,-1,0}(z)=1+\beta{z}-\frac{z[\Gamma(c)]^2}{\Gamma(a)\Gamma(b)\Gamma(c-a)\Gamma(c-b+1)}
\int_0^{1}\frac{t^{a+b-1}(1-t)^{c-a-b}}{(1-zt)|{}_{2}F_{1}(a,b;c;1/t)|^{2}}dt,
\end{equation}
or, taking $N=2$ to get 
\begin{equation}\label{eq:example12-2}
R_{0,-1,0}(z)=1-\frac{ac}{c^2}z-\frac{z^2[\Gamma(c)]^2}{\Gamma(a)\Gamma(b)\Gamma(c-a)\Gamma(c-b+1)}
\int_0^{1}\frac{t^{a+b}(1-t)^{c-a-b}}{(1-zt)|{}_{2}F_{1}(a,b;c;1/t)|^{2}}dt.
\end{equation}
In particular, for $a=b=1$, $c=2$ we arrive at the curious identity
$$
\frac{z}{\Log(1+z)}=1+z\int\limits_{1}^{\infty}\frac{dx}{(\log^2(x-1)+\pi^2)(x+z)}.
$$
If $z$ is replaced by $1$ in the numerator on the left hand side, a similar representation can be found in \cite[(34)]{BergPedersen2011}.  

Further, $B_{0,-1,0}P_0(t)=B_{0,-1,0}P_0$ does not depend on $t$, and thus does not change sign on $(0,1)$.  Hence, according to Theorem~\ref{th:RnmNkappa} $R_{0,-1,0}\in\SU$ when
$B_{0,-1,0}P_0\ge0$  or $-R_{0,-1,0}\in\SU$ when $B_{0,-1,0}P_0\le0$.
Furthermore, according to Theorem~\ref{th:B_P_Rnm_Nkappa} the function
$B_{0,-1,0}P_0R_{0,-1,0}(z+\omega)$ lies in $\SU$ for all real values of parameters  $a,b,c,\omega$  such that~$-c\notin\N_0$ and $\omega\le1$.

\bigskip

\textbf{Example~13}. For the ratio $R_{-1,-1,0}(z)$ according to \eqref{eq:nmrelated} we obtain $p=2$, $l=r=0$.   Theorem~\ref{th:2F1identity} and definition \eqref{eq:B-defined} yields: 
$$
B_{-1,-1,0}P_0(t)=-\frac{[\Gamma(c)]^2(c-a-b+1)}{\Gamma(a)\Gamma(b)\Gamma(c-a+1)\Gamma(c-b+1)}.
$$
Using Lemmas~\ref{lm:asymp-infnolog} and Lemma~\ref{lm:asymp-inflog} or by direct, albeit tedious calculation we get the asymptotic approximations
\smallskip
\begin{compactenum}[\indent(1)]\itemsep1.5pt plus .7pt minus .2pt
\item if $a>b+1$, then $R_{-1,-1,0}(z)=B(a,b)z+A(a,b)+o(1)$ as $z\to\infty$;
\item if $b\le{a}\le{b+1}$, then $R_{-1,-1,0}(z)=B(a,b)z+o(z)$ as $z\to\infty$;
\item if $b-1\le{a}\le{b}$, then $R_{-1,-1,0}(z)=B(b,a)z+o(z)$ as $z\to\infty$;
\item if $a<b-1$, then  $R_{1,-1,0}(z)=B(b,a)z+A(b,a)+o(1)$ as $z\to\infty$,
\end{compactenum}
\smallskip
where
$$
B(a,b)=\frac{a-1}{b-c},~~A(a,b)=\frac{(a-1)(2b-c)}{(c-b)(1+b-a)}.
$$

Hence,  if $|a-b|>1$, then $R_{1,-1,0}(z)=\beta{z}+\alpha+o(1)$ as $z\to\infty$,  where
$(\beta,\alpha)=(B(a,b),A(a,b))$ if $a>b+1$ and 
$(\beta,\alpha)=(B(b,a),A(b,a))$ if $a<b-1$. Hence, for $|a-b|>1$ the $N=1$ case of Theorem~\ref{th:2F1ratio-repr} leads to the representation:
\begin{equation}\label{eq:example13}
R_{-1,-1,0}(z)=\alpha+\beta{z}-\frac{[\Gamma(c)]^2(c-a-b+1)}{\Gamma(a)\Gamma(b)\Gamma(c-a+1)\Gamma(c-b+1)}
\int_0^{1}\frac{t^{a+b-2}(1-t)^{c-a-b}}{|{}_{2}F_{1}(a,b;c;1/t)|^{2}(1-zt)}dt.
\end{equation}
In addition to the condition $|a-b|>1$ we need to assume any of the conditions \eqref{item:Run1}-\eqref{item:Run5} of Theorem~\ref{th:2F1zeros}. Under these restrictions and except for the degenerate cases $ab=0$, $(c-a)(c-b)=0$ the condition $\nu>-1$ from \eqref{eq:asymp1} reads
$$
(c-a-b+2)_{-}-(c-a-b)_{-}>-1
$$
which is true for all real $a,b,c$. As $R_{-1,-1,0}(z)=\beta{z}+o(z)$ as $z\to\infty$,  where
$\beta=B(a,b)$ if $a\ge{b}$ and 
$\beta=B(b,a)$ if $a\ge{b}$, we can lift the restriction $|a-b|>1$ by taking $N=1$ in Theorem~\ref{th:2F1ratio-repr} which leads to 
\begin{equation}\label{eq:example13-1}
R_{-1,-1,0}(z)=1+\beta{z}-\frac{z[\Gamma(c)]^2(c-a-b+1)}{\Gamma(a)\Gamma(b)\Gamma(c-a+1)\Gamma(c-b+1)}
\int_0^{1}\frac{t^{a+b-1}(1-t)^{c-a-b}}{|{}_{2}F_{1}(a,b;c;1/t)|^{2}(1-zt)}dt,
\end{equation}
or, taking $N=2$, we get 
\begin{equation}\label{eq:example13-2}
R_{-1,-1,0}(z)=1+\frac{(a+b-1)c}{c^2}z-\frac{z^2[\Gamma(c)]^2(c-a-b+1)}{\Gamma(a)\Gamma(b)\Gamma(c-a+1)\Gamma(c-b+1)}
\int_0^{1}\frac{t^{a+b}(1-t)^{c-a-b}}{|{}_{2}F_{1}(a,b;c;1/t)|^{2}(1-zt)}dt.
\end{equation}

Further, $B_{-1,-1,0}P_0(t)=B_{-1,-1,0}P_0$ does not depend on $t$, and thus does not change sign on $(0,1)$.  Hence, according to Theorem~\ref{th:RnmNkappa} $R_{-1,-1,0}\in\SU$ when
$B_{-1,-1,0}P_0\ge0$  or $-R_{-1,-1,0}\in\SU$ when $B_{-1,-1,0}P_0\le0$.
Furthermore, according to Theorem~\ref{th:B_P_Rnm_Nkappa} the function
$B_{-1,-1,0}P_0R_{-1,-1,0}(z+\omega)$ lies in $\SU$ for all real values of parameters $a,b,c,\omega$  such that~$-c\notin\N_0$ and $\omega\le1$.

\medskip 

\textbf{Example~14}.  For the ratio $R_{-1,1,0}(z)$ according to \eqref{eq:nmrelated} we obtain $p=l=r=0$.  Theorem~\ref{th:2F1identity} and definition \eqref{eq:B-defined} yield: $$
B_{-1,1,0}P_0=\frac{[\Gamma(c)]^2(a-b-1)}{\Gamma(a)\Gamma(b+1)\Gamma(c-a+1)\Gamma(c-b)}.
$$
The asymptotic behavior of $R_{-1,1,0}(z)$ as $z\to\infty$ is rather complicated and depends on the relation between $a$ and $b$. Application of Lemmas~\ref{lm:asymp-infnolog} and \ref{lm:asymp-inflog}  yield:
\smallskip
\begin{compactenum}[\indent(1)]\itemsep1.5pt plus .7pt minus .2pt
\item if $b+1<a$, then  $R_{-1,1,0}(z)=o(1)$ as $z\to\infty$;
\item if $b\le{a}\le{b+1}$, then  $R_{-1,1,0}(z)=o(z)$ as $z\to\infty$;
\item if $b-1\le{a}<b$, then  $R_{-1,1,0}(z)=B{z}+o(z)$ as $z\to\infty$;
\item if $a<b-1$, then  $R_{-1,1,0}(z)=B{z}+C+o(1)$ as $z\to\infty$,
\end{compactenum}
\smallskip
where
$$
B=\frac{(b-a)(b-a+1)}{b(a-c)},~~~C=\frac{(b-a)(b-a+1)(c(a+b-1)-2ab)}{b(c-a)(a-b-1)(a-b+1)}.
$$
Hence,  if $|a-b|>1$ we have $R_{-1,1,0}(z)=\beta{z}+\alpha+o(1)$ as $z\to\infty$, where $(\beta,\alpha)=(0,0)$ when $a>b+1$ and $(\beta,\alpha)=(B,C)$  when $a<b-1$.
Then for $|a-b|>1$ the $N=0$ case of Theorem~\ref{th:2F1ratio-repr} leads to the representation:
\begin{equation}\label{eq:example14}
R_{-1,1,0}(z)=\beta{z}+\alpha+\frac{[\Gamma(c)]^2(a-b-1)}{\Gamma(a)\Gamma(b+1)\Gamma(c-a+1)\Gamma(c-b)}
\int_0^{1}\frac{t^{a+b-2}(1-t)^{c-a-b}}{|{}_{2}F_{1}(a,b;c;1/t)|^{2}(1-zt)}dt.
\end{equation}
In addition to the condition $|a-b|>1$ we need to assume any of the conditions \eqref{item:Run1}-\eqref{item:Run5} of Theorem~\ref{th:2F1zeros}. Under these restrictions and except for the degenerate cases $ab=0$, $(c-a)(c-b)=0$ the condition $\nu>-1$ from \eqref{eq:asymp1} reads
$$
(c-a-b)_{-}-(c-a-b)_{-}>-1
$$
which is true for all real $a,b,c$.  

For arbitrary values of $a,b$ we have $R_{-1,1,0}(z)=\beta{z}+o(z)$ as $z\to\infty$, where
$\beta=0$ when $a\ge{b}$ and $\beta=B$  when ${a}<b$.
Hence, we can use representation \eqref{eq:mainrepresentionN} with $N=1$ yielding
\begin{equation}\label{eq:example14-1}
R_{-1,1,0}(z)=1+\beta{z}+\frac{z[\Gamma(c)]^2(a-b-1)}{\Gamma(a)\Gamma(b+1)\Gamma(c-a+1)\Gamma(c-b)}
\int_0^{1}\frac{t^{a+b-1}(1-t)^{c-a-b}}{|{}_{2}F_{1}(a,b;c;1/t)|^{2}(1-zt)}dt
\end{equation}
or with $N=2$ yielding 
\begin{equation}\label{eq:example14-2}
R_{-1,1,0}(z)=1+\frac{(a-b-1)c}{c^2}z+\frac{z^2[\Gamma(c)]^2(a-b-1)}{\Gamma(a)\Gamma(b+1)\Gamma(c-a+1)\Gamma(c-b)}
\int_0^{1}\frac{t^{a+b}(1-t)^{c-a-b}}{|{}_{2}F_{1}(a,b;c;1/t)|^{2}(1-zt)}dt.
\end{equation}
Further, $B_{-1,1,0}P_0(t)=B_{-1,1,0}P_0$ does not depend on $t$, and thus does not change sign on $(0,1)$.  Hence, according to Theorem~\ref{th:RnmNkappa} $R_{-1,1,0}\in\SU$ when
$B_{-1,1,0}P_0\ge0$  or $-R_{-1,1,0}\in\SU$ when $B_{-1,1,0}P_0\le0$.
Furthermore, according to Theorem~\ref{th:B_P_Rnm_Nkappa} the function
$B_{-1,1,0}P_0R_{-1,1,0}(z+\omega)$ lies in $\SU$ for all real values of parameters $a,b,c,\omega$  such that~$-c\notin\N_0$ and $\omega\le1$.

\medskip

\textbf{Example~15}.  For the ratio $R_{-2,-2,0}(z)$ according to \eqref{eq:nmrelated} we obtain $p=4$, $l=0$, $r=1$.   Theorem~\ref{th:2F1identity} and definition \eqref{eq:B-defined} yield: $$
B_{-2,-2,0}P_1(t)=-\frac{[\Gamma(c)]^2(c-a-b+2)(\rho_0+\rho_1t)}{\Gamma(a)\Gamma(b)\Gamma(c-a+2)\Gamma(c-b+2)},
$$
where  $\rho_0=a^2+b^2-(c+2)(a+b)+3c+1$, $\rho_1=c(c-a-b+1)+2(ab-a-b+1)$.
Using Lemmas~\ref{lm:asymp-infnolog} and \ref{lm:asymp-inflog} or by direct, albeit tedious calculation we obtain the asymptotic approximations
\smallskip
\begin{compactenum}[\indent(1)]\itemsep2.3pt plus .7pt minus .2pt
\item if $a>b+2$, then $R_{-2,-2,0}(z)=\gamma_{a,b,c}{z^2}+\beta_{a,b,c}{z}+\alpha_{a,b,c}+o(1)$ as $z\to\infty$;

\item if $b+1<a\le{b+2}$, then $R_{-2,-2,0}(z)=\gamma_{a,b,c}{z^2}+\beta_{a,b,c}{z}+o(z)$ as $z\to\infty$;

\item if $b\le{a}\le{b+1}$, then $R_{-2,-2,0}(z)=\gamma_{a,b,c}{z^2}+o(z^2)$ as $z\to\infty$;

\item if $b-1\le{a}\le{b}$, then $R_{-2,-2,0}(z)=\gamma_{b,a,c}{z^2}+o(z^2)$ as $z\to\infty$;

\item  if $b-2\le{a}<b-1$, then $R_{-2,-2,0}(z)=\gamma_{b,a,c}{z^2}+\beta_{b,a,c}{z}+o(z)$ as $z\to\infty$;

\item if $a<b-2$, then $R_{-2,-2,0}(z)=\gamma_{b,a,c}{z^2}+\beta_{b,a,c}{z}+\alpha_{b,a,c}+o(1)$ as $z\to\infty$,

\end{compactenum}
where
$$
\gamma_{a,b,c}=\frac{(a-2)(a-1)}{(c-b)(c-b+1)},~~~\beta_{a,b,c}=\frac{2(a-2)(a-1)(c+1-2b)}{(c-b)(c-b+1)(b-a+1)},
$$
$$
\alpha_{a,b,c}=\gamma_{a,b,c}\frac{c(c+1)(a-1)+2b^2(a+4c-3b)-2ab(c+2)-b(3c^2-c-6)}{(a-b-2)(a-b-1)^2}.
$$
Hence, for $|a-b|>2$ we have $R_{-2,-2,0}(z)=\gamma{z^2}+\beta{z}+\alpha+o(1)$ as $z\to\infty$, where $(\gamma,\beta,\alpha)=(\gamma_{a,b,c},\beta_{a,b,c},\alpha_{a,b,c})$ when $a>b+2$ and $(\gamma,\beta,\alpha)=(\gamma_{b,a,c},\beta_{b,a,c},\alpha_{b,a,c})$  when $a<b-2$.
Then for $|a-b|>2$ the case $N=0$ of Theorem~\ref{th:2F1ratio-repr} leads to the representation:
\begin{equation}\label{eq:example15}
R_{-2,-2,0}(z)=\gamma{z^2}+\beta{z}+\alpha-\frac{[\Gamma(c)]^2(c-a-b+2)}{\Gamma(a)\Gamma(b)\Gamma(c-a+2)\Gamma(c-b+2)}
\int_0^{1}\frac{(\rho_0+\rho_1t)t^{a+b-3}(1-t)^{c-a-b}}{|{}_{2}F_{1}(a,b;c;1/t)|^{2}(1-zt)}dt.
 \end{equation}
In addition to the condition $|a-b|>2$ we need to assume any of the conditions \eqref{item:Run1}-\eqref{item:Run5} of Theorem~\ref{th:2F1zeros}. Under these restrictions and except for the degenerate cases $ab=0$, $(c-a)(c-b)=0$ the condition $\nu>-1$ from \eqref{eq:asymp1} reads
$$
(c-a-b+4)_{-}-(c-a-b)_{-}>-1
$$
which is true for all real $a,b,c$.  If $1<|a-b|\le2$, we see  that the  asymptotics takes the form  $R_{-2,-2,0}(z)=\gamma{z^2}+\beta{z}+o(z)$ as $z\to\infty$, where $(\gamma,\beta)=(\gamma_{a,b,c},\beta_{a,b,c})$ when $a>b+1$ and $(\gamma,\beta)=(\gamma_{b,a,c},\beta_{b,a,c})$  when $a<b-1$. Hence,  for $1<|a-b|$ according to \eqref{eq:mainrepresentionN} with $N=1$ we get
\begin{equation}\label{eq:example15-1}
R_{-2,-2,0}(z)=\gamma{z^2}+\beta{z}+1-\frac{z[\Gamma(c)]^2(c-a-b+2)}{\Gamma(a)\Gamma(b)\Gamma(c-a+2)\Gamma(c-b+2)}
\int_0^{1}\frac{(\rho_0+\rho_1t)t^{a+b-2}(1-t)^{c-a-b}}{|{}_{2}F_{1}(a,b;c;1/t)|^{2}(1-zt)}dt.
 \end{equation}
Similarly, for $|a-b|\le1$  the  asymptotics takes the form $R_{-2,-2,0}(z)=\gamma{z^2}+o(z^2)$, where $\gamma=\gamma_{a,b,c}$ when $a\ge{b}$ and $\gamma=\gamma_{b,a,c}$  when $a\le{b}$. Hence, without additional restrictions according to \eqref{eq:mainrepresentionN} with $N=2$ we get
\begin{multline}\label{eq:example15-2}
R_{-2,-2,0}(z)=1+\frac{2(2-a-b)}{c^2}z+\gamma{z^2}
\\
-\frac{z^2[\Gamma(c)]^2(c-a-b+2)}{\Gamma(a)\Gamma(b)\Gamma(c-a+2)\Gamma(c-b+2)}
\int_0^{1}\frac{(\rho_0+\rho_1t)t^{a+b-1}(1-t)^{c-a-b}}{|{}_{2}F_{1}(a,b;c;1/t)|^{2}(1-zt)}dt
 \end{multline}
under any of the conditions \eqref{item:Run1}-\eqref{item:Run5} of Theorem~\ref{th:2F1zeros}, but without any other restrictions.

Further, $B_{-2,-2,0}P_1(t)$ does not change sign on $(0,1)$ if the zero $t_*=-\rho_{0}/\rho_{1}$ of the polynomial $P_1(t)=\rho_0+\rho_{1}t$ does not lie in $(0,1)$, i.e. $-\rho_{0}/\rho_{1}\in(-\infty,0]\cup[1,\infty)$.  If this case, according to Theorem~\ref{th:RnmNkappa} $R_{-2,-2,0}\in\SU$ if
$B_{-2,-2,0}P_1(t)\ge0$ on $(0,1)$ or $-R_{-2,-2,0}\in\SU$ if $B_{-2,-2,0}P_1(t)\le0$ on $(0,1)$.
Furthermore, according to Theorem~\ref{th:B_P_Rnm_Nkappa} the functions
\[
    B_{-2,-2,0}P_1\Big(\frac1{z+\omega}\Big)R_{-2,-2,0}(z+\omega)
    \quad\text{and}\quad
    \frac{R_{-2,-2,0}(z+\omega)}{B_{-2,-2,0}P_1\big(1/(z+\omega)\big)}
\]
lie in~$\SU$ for all real values of parameters  $a,b,c,\omega$ such that~$-c\notin\N_0$ and $\omega\le1$.

\paragraph*{Acknowledgements.}
The first author was supported by the research fellowship DY 133/1-1 of the German Research
Foundation~(DFG). He is also grateful to HTWG Konstanz and the Konstanz University: the very
last lines of this paper were written during his research visit to their wonderful town.

\end{document}